%% file: DQDG.tex
\documentclass[11pt,twoside]{amsart}

\usepackage{latexsym}
\usepackage{amssymb}
\usepackage{vmargin}
\usepackage{amscd}
\usepackage{stmaryrd}
\usepackage{euscript}
\usepackage{mathrsfs}
\usepackage[all]{xy}

\usepackage{xr}
\usepackage{hyperref}
\hypersetup{colorlinks,
linkcolor=black,
citecolor=black,
urlcolor=blue}

\setmargins{32mm}{20mm}{14.6cm}{22cm}{1cm}{1cm}{1cm}{1cm}

\include{newcommands}

\include{newthms}

\begin{document}

\begin{abstract}
 We explain how to translate several recent results in derived algebraic geometry to derived differential geometry.  These concern shifted 
Poisson  structures on NQ-manifolds, Lie groupoids, smooth stacks  and derived generalisations, and include existence and classification of various deformation  quantisations.
\end{abstract}

\title[Shifted Poisson structures in derived differential geometry]{An outline of shifted Poisson structures and deformation quantisation in derived differential geometry}
\author{J.P.Pridham}

\maketitle

\section*{Introduction}

In recent years, there have been many developments in the study of shifted Poisson structures and deformation quantisations in derived algebraic geometry, beginning with the  systematic study of shifted symplectic structures in \cite{PTVV}. Translating these results into a differential geometric setting is fairly straightforward, but in most cases this has not been done explicitly, a notable exception being \cite{PymSafronov}. 
This situation has arisen partly because the most suitable setting for derived differential geometry in which to write down these results is that built on dg $\C^{\infty}$-rings, for which the foundations have only recently been written down in 
 \cite{CarchediRoytenbergHomological,nuitenThesis}.

The aim of this manuscript is to  explain how to formulate shifted Poisson  structures and various deformation quantisations in differential geometric settings, and to indicate how to adapt existing algebro-geometric  proofs, in most cases with a summary of the argument. In places we have imposed unnecessarily strong hypotheses for the purposes of exposition, with pointers which we hope will enable readers who need more general statements  to recover them from  the cited results in the literature.

The first section is concerned with shifted symplectic structures. These should be familiar as natural generalisations of the homotopy symplectic structures of \cite{KhudaverdianVoronov,bruceGeomObjects}. We start defining these for NQ-manifolds, which  should be the most familiar of the objects we will consider. We then consider dg manifolds with differentials going in the opposite direction, set up as the analogue of the algebraic dg manifolds of \cite{Quot}; derived critical loci in the form of classical BV complexes give rise to examples of such dg manifolds. The obvious difference between the formulation of NQ-manifolds and of dg manifolds is in the direction of the differentials $Q$ and $\delta$,  but the more important distinction is that we use $\delta$ to define equivalences via homology isomorphisms. Homological considerations then lead to major differences in the behaviour of $Q$ and $\delta$   (see Remarks \ref{derivedstackyrmk} and Appendix \ref{equivapp}). We then formulate shifted symplectic structures for dg NQ-
manifolds and super dg NQ-manifolds, where the main difficulty 
is in keeping track of all the different gradings. 

In Section \ref{poisssn}, we introduce shifted Poisson structures on all these objects, and establish the equivalence between shifted symplectic structures and non-degenerate shifted Poisson structures (Theorems \ref{compatthmLie} and \ref{compatthmdgLie}). On an  NQ-manifold, a shifted Poisson structure is essentially just a shifted $L_{\infty}$-algebra structure on the dg  algebra of smooth functions, with each operation acting as a smooth multiderivation. The description for dg manifolds is similar, while for dg NQ-manifolds the formulation has some subtleties arising from the multiple gradings. 

Section \ref{quantnsn} then discusses deformation quantisation of $n$-shifted Poisson structures. We focus our attention on the cases  $n=0$  (Theorem \ref{derived0quantthm}) and $n=-1$ (Theorem \ref{quantpropsdneg1}). These quantisations respectively correspond to curved $A_{\infty}$ and $BV_{\infty}$ deformations of the  dg  algebra of smooth functions.  We then briefly sketch the deformation quantisation of $0$-shifted Lagrangians (Theorem \ref{quantpropsdlag}). We also look at the case $n=-2$ (Theorem  \ref{uniqueconn}), in which setting quantisations are solutions of a quantum master equation. 

The final section then explains how  these results  translate to Lie groupoids, including higher and derived  Lie groupoids. 
For  smooth Artin stacks, including higher stacks, the corresponding stacky CDGAs of \cite{poisson} or  graded mixed cdgas of \cite{CPTVV} are just given by dg NQ-manifolds, and the formulation of shifted Poisson structures for such stacks comes down to  establishing a simplicial resolution of a higher Lie groupoid  by NQ-manifolds in which all the maps induce quasi-isomorphisms on cotangent complexes. We begin by outlining the subtleties of functoriality in \S \ref{functsn}. In \S\S \ref{liegpdsn}, \ref{higherliegpdsn}, we then show how to resolve  Lie groupoids, higher Lie groupoids and derived higher Lie groupoids by suitable simplicial dg NQ-manifolds, and thus to extend all our constructions to these objects.

Appendix A explains in more detail the notions of equivalence which motivate the different treatment of dg and NQ manifolds, and resulting homotopical considerations.  Appendix B summarises standard obstruction arguments and their application to spaces of Poisson structures and their quantisations.

I would like to thank Ping Xu and Ted Voronov for helpful comments; unfortunately, I was unable to implement all their suggested changes to terminology without generating clashes elsewhere. I would also like to thank all those who have asked questions in related talks, motivating much additional explanation, and the anonymous referee for identifying oversights and providing detailed comments.
 
\tableofcontents

\subsection*{Notation and terminology}\label{notnsn}

From the outset, we will be working with differential graded superalgebras. Thus our objects are initially $\Z \by \Z/2$-graded, and we later encounter objects which are $\Z^2 \by \Z/2$- or even  $\Z^3 \by \Z/2$-graded. However, our indexing conventions differ from those usually found in supermathematics (e.g. in \cite[4.6]{CarchediRoytenbergHomological} or \cite{voronovMackenzie}), in that for us the parity of an element is the  mod $2$ sum of its indices. We accordingly denote our copy of $\Z/2$ by $\{=,\ne\}$, so the parity of  $(m,=)$ is $m \mod 2$ and  parity of  $(m, \ne)$ is $m+1 \mod 2$. We refer to the $\Z$-gradings as degrees (chain or cochain denoted by subscripts and superscripts, respectively) rather than weights, and to the indices $\{=,\ne\}$ as equal and unequal parity. 

In particular, a \emph{super chain complex} will be a $\Z \by \{=,\ne\}$-graded vector space, equipped with a square-zero operation  $\delta$ of degree $-1$ and odd (hence equal) parity.  
For a super chain complex $(V,\delta)$, the subcomplexes  of equal and unequal parity are thus given by
\begin{align*}
V_{\bt,=} =( \ldots \xra{\delta} V_{3}^{\odd} \xra{\delta}  V_{2}^{\even} \xra{\delta}  V_{1}^{\odd} \xra{\delta} V_{0}^{\even} \xra{\delta} \ldots), \\ 
V_{\bt,\ne} = (\ldots \xra{\delta} V_{3}^{\even} \xra{\delta}  V_{2}^{\odd} \xra{\delta}  V_{1}^{\even} \xra{\delta} V_{0}^{\odd} \xra{\delta} \ldots). 
\end{align*}

Similarly, a  \emph{super chain complex} will be a $\Z \by \{=,\ne\}$-graded vector space equipped with a square-zero operation  $Q$ (corresponding to the $\pd$ of \cite{poisson}) or $d$ of degree $+1$ and odd (hence equal) parity. For a super cochain complex $(V, Q)$, we thus have subcomplexes
\begin{align*}
V^{\bt}_{=} =( \ldots \xra{Q} V^{0,\odd} \xra{Q}  V^{1,\even} \xra{Q}  V^{2,\odd} \xra{Q} V^{3,\even} \xra{Q} \ldots),  \\
V^{\bt}_{\ne} = (\ldots \xra{Q} V^{0,\even} \xra{Q}  V^{1,\odd} \xra{Q}  V^{2,\even} \xra{Q} V^{3,\odd} \xra{Q} \ldots). 
\end{align*}
of equal and unequal parity.




For a chain (resp. cochain) complex $M$, we write $M_{[i]}$ (resp. $M^{[j]}$) for the complex $(M_{[i]})_m= M_{i+m}$ (resp. $(M^{[j]})^m = M^{j+m}$). We also denote the parity reversion operator by $\Pi$, so  for a super chain complex $M$, we have $(\Pi M)^{=}_i:= M^{\ne}_i$ and $(\Pi M)^{\ne}_i:= M^{=}_i$.  

We define tensor products to follow the usual super/graded conventions, so for super chain complexes $M,N$, we have
\begin{align*}
 (M\ten N)_n^{=} &= \bigoplus_{i+j=n} ((M_i^{=}\ten N_j^{=})\oplus (M_i^{\ne}\ten N_j^{\ne}))\\
(M\ten N)_n^{\ne} &= \bigoplus_{i+j=n} ((M_i^{=}\ten N_j^{\ne})\oplus (M_i^{\ne}\ten N_j^{=})),
\end{align*}
and we define symmetric powers $\Symm^nM$ by passing to $S_n$-coinvariants of tensor powers $M^{\ten n}$, where the $S_n$-action on $M^{\ten n}$ is twisted by the sign of the permutation on terms of odd parity.
We define alternating powers similarly, but with reversed sign.

Given a commutative algebra $A$ in super chain complexes, and $A$-modules $M,N$ in super chain complexes, we write $\HHom_A(M,N)$ for the super chain complex given by setting $\HHom_A(M,N)_i = \HHom_A(M,N)_i^{=} \oplus \HHom_A(M,N)_i^{\ne}$ to consist of $A$-linear  morphisms from $M$ to $N$ of chain degree $i$, with the decomposition corresponding to  equal and unequal parity;
the differential  on  $\HHom_A(M,N)$ is given by $\delta f= \delta_N \circ f \pm f \circ \delta_M$,
where $V_{\#}$ denotes the graded vector space underlying a chain complex $V$. We follow analogous conventions for  internal $\Hom$s in super cochain complexes, and in super chain cochain complexes.



We will use the symbol $\cong$ to denote isomorphism and $\simeq$ to denote quasi-isomorphism or weak equivalence.

By ``manifold'', we will mean a real $\C^{\infty}$ manifold satisfying the conditions of Whitney's embedding theorem, so admitting a closed immersion to some affine space $\R^n$.

\section{Symplectic structures on stacky and derived enhancements of supermanifolds}

In derived algebraic geometry, derived stacks are  enhancements of schemes in two different ways. Derived structures give analogues of Kuranishi structures, while  stacky structures give analogues of Lie groupoids. Much of the complexity in formulating Poisson structures for derived stacks \cite{poisson,CPTVV} arises from considering the derived and stacky structures simultaneously. Since stacky structures  in the form of Lie algebroids or NQ-manifolds are more likely to be familiar to readers (cf. \cite{PymSafronov}), we start with them (whereas the initial emphasis in \cite{poisson} was on derived structures).

\subsection{Super NQ-manifolds}

The following definition broadly corresponds to the $\infty$-Lie algebroids  of \href{https://ncatlab.org/nlab/show/Lie+infinity-algebroid}{the nlab},  or to the dg manifolds of \cite{roytenbergDGMfdsStacks}.

\begin{definition}\label{NQdef}
 Define an NQ-manifold $X$ to be a pair $(X_0, \O_X)$ where $X_0$ is a real differentiable manifold and
 \[
 \O_X=(\O_X^0 \xra{Q} \O_X^1 \xra{Q} \O_X^2 \xra{Q} \ldots )
\]
 is a graded-commutative $\R$-algebra in cochain complexes of sheaves on $X_0$. We require that $\O_{X}^0$ be the sheaf of smooth functions on $X_0$, that the cochain differential $Q \co \O_X^0 \to \O_X^1$ be a $\cC^{\infty}$-derivation  and that $\O_X$ be locally semi-free in the sense that the underlying graded-commutative sheaf $\O_{X}^{\#}$ is locally of the form $\O_{X_0}\ten \Symm_{\R}U$ for some finite-dimensional graded vector space $U= U^1 \oplus \ldots \oplus U^N$.

Write $\cC^{\infty}(X):= \Gamma(X_0,\O_X)$ for the cochain complex of global sections of $\O_X$.
\end{definition}

In particular,  NQ-manifolds $X$ for which $\O_X$ is generated in cochain degrees $0,1$ correspond to Lie algebroids (see e.g. \cite[Lemma 1.13]{ZambonZhu}).

Beware that we are denoting the underlying manifold $X_0$ with a subscript $0$, although its ring of functions $\O_X^0$ has a superscript $0$. This essentially arises from contravariance, and is part of a general indexing convention in \cite{ddt1,stacks2} inherited from related simplicial and cosimplicial objects.

\begin{remark}
 For each NQ-manifold, the commutative dg algebra $\cC^{\infty}(X)$ naturally has the structure of a  dg $\cC^{\infty}$-ring in the sense of \cite{CarchediRoytenbergHomological}. However, we have slightly more structure because $\cC^{\infty}(X)^0$ is a $\cC^{\infty}$-ring in the sense of \cite{dubuc,MoerdijkReyes} and  $Q \co  \cC^{\infty}(X)^0 \to\cC^{\infty}(X)^1$ is a $\cC^{\infty}$-derivation, while the definition of \cite{CarchediRoytenbergHomological} would only require that $\H^0\cC^{\infty}(X)$ be a $\cC^{\infty}$-ring. 
\end{remark}


The following differs slightly from usual conventions, which tend to regard NQ-manifolds and NQ-supermanifolds as synonymous, since we have $\N_0 \by \Z/2$-gradings rather than just $\N_0$-gradings.

\begin{definition}
Define a super NQ-manifold  $X$ to be a pair $(X_0^{=}, \O_X)$ where $X_0^{=}$ is a real differentiable  manifold and 
$\O_X= \O_{X}^{\ge 0}$
is a graded-commutative $\R$-algebra in super cochain complexes of sheaves on $X_0^{=}$. We require that   the cochain differential $Q \co \O_X^{0,=} \to \O_X^{1,=}$ be a $\cC^{\infty}$-derivation  and that  $\O_X$ be locally semi-free in the sense that the underlying super graded-commutative sheaf $\O_{X}^{\#}$ is locally of the form $\O_{X_0^{=}}\ten \Symm_{\R}U$   for some finite-dimensional super  graded vector space $U= U^{0,\ne} \oplus U^1 \oplus \ldots \oplus U^N$ (with $U^i= U^{i,=}\oplus U^{i,\ne}$). In particular, this implies that $\O_{X}^{0}/(\O_X^{0, \ne})$ is isomorphic to the sheaf of smooth functions on $X_0^{=}$.

Write $\cC^{\infty}(X):= \Gamma(X_0^{=},\O_X)$ for the super cochain complex of global sections of $\O_X$.
\end{definition}

\begin{definition}\label{Omegadef}
Given a  super NQ-manifold  $X$, define $\Omega^1_X$ to be the sheaf of smooth $1$-forms of $\O_X$.
 This is a sheaf of $\O_X$-modules in super cochain complexes (with basis locally given by $\{dx^i\}_i$ when $\O_X$ has local co-ordinates $\{x^i \}_i$), and we write $\Omega^p_X:= \L^p_{\O_X}\Omega^1_X$. We also write
$\Omega^p_{\cC^{\infty}(X)}$ for the super cochain complex $\Gamma(X, \Omega^p_X)$ of global sections of $\Omega^p_X$.  

Define $T_X$ to be the sheaf $\hom_{\O_X}(\Omega^1_X,\O_X)$  of smooth $1$-vectors of $\O_X$, and write $T_{\cC^{\infty}(X)}$ for the super cochain complex $\Gamma(X, T_X)$ of global sections; equivalently, this is the internal $\Hom$ space $\HHom_{\cC^{\infty}(X)}(\Omega^1_{\cC^{\infty}(X)}, \cC^{\infty}(X))$.
\end{definition}

\begin{remark}\label{abstractOmegarmk}
Beware that $\Omega^1_{\cC^{\infty}(X)}$ is not the module of K\"ahler differentials of the abstract super dg algebra $\cC^{\infty}(X)$, since we have constraints requiring that the derivation $d \co \cC^{\infty}(X) \to \Omega^1_{\cC^{\infty}(X)}$ restrict to a $\cC^{\infty}$-derivation on $ \cC^{\infty}(X_0^=)$. In the special case when $X=X_0^=$ is just a manifold, our $\Omega^1_{\cC^{\infty}(X)}$ is the module $\Omega_{\cC^{\infty}(X)}$ of \cite[\S 5.3]{joyceAGCinfty}.

Explicitly, $\Omega^1_{\cC^{\infty}(X)}$ is the $\cC^{\infty}(X)$-module in super cochain complexes generated by elements $da$ (given the same degree and parity as $a$), for homogeneous elements $a \in \cC^{\infty}(X)$, subject to the relations
\begin{enumerate}
 \item $d(ab)= (da)b + (-1)^{\bar{a}} a(db)$, where $\bar{a}$ denotes the parity of $a$, and 
\item for $a_1, \ldots, a_n \in \cC^{\infty}(X)^{0,=}$ and $f \in \cC^{\infty}(\R^n)$, we have
\[
 d(f(a_1, \ldots, a_n)) = \sum_{i=1}^n \frac{\pd f}{\pd x_i}(a_1, \ldots, a_n) d a_i.
\]
\end{enumerate}
The cochain differential $Q$ on $\Omega^1_{\cC^{\infty}(X)}$ is then given by $Q (da)= d(Q a)$; this is often denoted $\cL_Q$, the Lie derivative. The $\cC^{\infty}(X)$-module $T_{\cC^{\infty}(X)}$ is given by derivations on $\cC^{\infty}(X)$, with a similar restriction on $\cC^{\infty}(X)^{0,=}$.
\end{remark}

\begin{definition}\label{DRdefLie}
Given a super NQ-manifold  $X$,  define the de Rham complex $\DR(X)$ to be the  total super cochain complex  of the double complex
\[
 \cC^{\infty}(X) \xra{d} \Omega^1_{\cC^{\infty}(X)} \xra{d} \Omega^2_{\cC^{\infty}(X)}\xra{d} \ldots,
\]
so $\DR(X)^m = \prod_{i+j=m}(\Omega^i_{\cC^{\infty}(X)})^j$ with  total differential  $d \pm Q$.

We define a filtration $\Fil$ on  $\DR(X)$ by setting $\Fil^p\DR(X) \subset \DR(X)$ to consist of terms $\Omega^i_{X}$ with $i \ge p$.
\end{definition}
In our formal arguments, the filtration $\Fil$ will play the same r\^ole as the Hodge filtration of \cite{poisson}, but beware that it is very different in situations where both are defined (such as on complex manifolds).

\begin{remark}\label{DRrmkNQ}
 Since the sheaf  $\Omega^{\bt}_X$ of double supercomplexes on $X_0^=$ is exact in the de Rham direction, note that the complex $\Fil^p\DR(X)$ models derived global sections $\oR\Gamma(X_0^=, \ker(d \co  \Omega^p_X \to \Omega^{p+1}_X))[-p]$ of the sheaf of closed $p$-forms; when $p=0$, this is just $\oR\Gamma(X_0^=, \R)$.
\end{remark}

The complex $\DR(X)$ has the natural structure of a commutative DG super algebra, filtered in the sense that $\Fil^i\Fil^j \subset \Fil^{i+j}$. 

The following definitions are adapted from \cite{PTVV} (where pre-symplectic structures are referred to as closed $p$-forms, and all objects have equal parity), although as noted in \cite{BouazizGrojnowski}, precursors exist in the mathematical physics literature (cf. \cite[Definition 2]{KhudaverdianVoronov} and \cite[Definition 5.2.1]{bruceGeomObjects}, which respectively consider even and odd structures for $\Z/2$-gradings rather than $\Z \by \Z/2$-gradings, with $Q=0$): 

\begin{definition}\label{presymplecticdefLie}
Define an $n$-shifted pre-symplectic structure $\omega$ on a super NQ-manifold  $X$ to be an element
\[
 \omega \in \z^{n+2}\Fil^2\DR(X)^{=} = \{ \omega \in \Fil^2\DR(X)^{n+2,=} ~:~  (d \pm Q)\omega =0\}.
\]

Define a parity-reversed  $n$-shifted pre-symplectic structure to be an element
\[
 \omega \in \z^{n+2}\Fil^2\DR(X)^{\ne}.
\]

Two pre-symplectic structures are regarded as equivalent if they induce the same cohomology class in $\H^{n+2}\Fil^2\DR(X)^{=}$ (resp. $\H^{n+2}\Fil^2\DR(X)^{\ne}$).
\end{definition}

Explicitly, this means that $\omega$ is given by a  sum $\omega = \sum_{i \ge 2} \omega_i$, with $\omega_i \in (\Omega^i_{X})^{n+2-i}$ (of equal or unequal parity, respectively)  and with  $d\omega_i = \pm Q \omega_{i+1}$ and $Q \omega_2=0$. Thus $\omega_2$ is $d$-closed up to a homotopy given by $\omega_3$, and so on.

\begin{definition}\label{NQsympdef}
 Define a (parity-reversed) $n$-shifted symplectic structure $\omega$ on $X$ to be a (parity-reversed) $n$-shifted pre-symplectic structure $\omega$ for which contraction with the component $\omega_2 \in \z^n\Omega^2_{\cC^{\infty}(X)}$ induces a quasi-isomorphism
\begin{align*}
 \omega_2^{\flat} \co T_X\ten_{\O_X}\O_X^0 &\to (\Omega^1_{X}\ten_{\O_X}\O_X^0)^{[n]}\quad \text{resp.}\\
\omega_2^{\flat} \co T_X\ten_{\O_X}\O_X^0  &\to \Pi(\Omega^1_{X}\ten_{\O_X}\O_X^0)^{[n]}
\end{align*}
of complexes of sheaves (i.e. the maps  $\sH^i(T_X\ten_{\O_X}\O_X^0) \to (\Pi) \sH^{n+i}(\Omega^1_{X}\ten_{\O_X}\O_X^0)$   on the associated homology sheaves are isomorphisms); 
note that since $Q \omega_2=0$, these are automatically chain maps by the tensor-Hom adjunction for chain complexes.
\end{definition}

\begin{remarks}\label{NQnondegrmks}
Note that for super NQ-manifolds, bounds on the generators of $\Omega^1_X$ mean that  $n$-shifted symplectic structures can only exist for $n \in [0,N]$.

The non-degeneracy condition in Definition \ref{NQsympdef} is more subtle than the one we will see  in Definition \ref{derivedsympdef} for dg manifolds with differentials in the opposite direction. This is because we need to ensure that $\omega_2^{\flat}$ induces quasi-isomorphisms between tensor powers of tangent complexes and cotangent complexes, which our definition gives by \cite[Lemma \ref{poisson-binondeglemma}]{poisson}.

 Because non-degeneracy is defined as a quasi-isomorphism rather than an isomorphism,  we have to treat negatively and positively graded cochain complexes differently, since they interact differently with homological constructions. In technical terms, the $\O_X$-modules $\Omega^1_X$ in cochain complexes are not usually cofibrant in the projective model structure when $X$ is an NQ-manifold, whereas the corresponding modules for dg manifolds will be.
%
%

\end{remarks}

\begin{examples}\label{exBg}
Many examples of shifted symplectic structures on NQ-manifolds are given in \cite{PymSafronov}.
Prototypical examples are formal completions  $\hat{T}^*[n]M$ of shifted cotangent bundles  of manifolds $M$ for $n\ge 0$, with $\cC^{\infty}(\hat{T}^*[n]M)$ given by the free graded-commutative algebra over $\cC^{\infty}(M)$ generated by the module $T_{\cC^{\infty}(M)}:=\cC^{\infty}(M,TM)_M$ of smooth sections  of the tangent bundle placed in cochain degree $n$, and with trivial differential $Q$. This NQ-manifold carries a natural $n$-shifted symplectic structure $\omega$ given by the canonical closed  $2$-form $\omega_2 \in T_{\cC^{\infty}(M)}\ten_{\cC^{\infty}(M)}\Omega^1_M \subset  (\Omega^2_{\hat{T}^*M[n]})^{[n]}$. 

Similarly, there is a super NQ-manifold  $\Pi \hat{T}^*[n]M$ with $\cC^{\infty}(\Pi \hat{T}^*[n]M)$  freely generated over $\cC^{\infty}(M)$  by  $\Pi T_{\cC^{\infty}(M)}^{[-n]}$, and this carries a natural parity-reversed $n$-shifted symplectic structure.

For a finite-dimensional real Lie algebra $\g$, we can define an NQ-manifold $B\g$ with underlying manifold a point, and functions $\O(B\g)$ given by  the Chevalley--Eilenberg complex
\[
\mathrm{CE}(\g,\R)=(  \R \xra{Q} \g^* \xra{Q} \Lambda^2\g^*\xra{Q} \ldots).
\]
of $\g$ with coefficients in $\R$.
We then have that $\Omega^p_{B\g} \cong  \mathrm{CE}(\g,S^p(\g^*))^{[-p]}$, so consideration of degrees shows that the space of $2$-shifted pre-symplectic structures is just given by $\g$-invariant symmetric bilinear forms $S^2(\g^*)^{\g}$ on $\g$, and that these are symplectic whenever the form is non-degenerate.

A more subtle example is given by the super NQ-manifold associated to the tangent Lie algebroid of a supermanifold $M$. This is given by $X:=(M, \Omega^{\bt}_M)$, with $Q$ corresponding to the de Rham differential. Here, $0$ gives an $n$-shifted symplectic structure for all $n$, because  $\Omega^1_{X}\ten_{\O_X}\O_X^0$ is isomorphic to the acyclic super cochain complex $\id \co \Omega^1_M \to \Omega^1_M$, so $0$ is a quasi-isomorphism between it and all shifts of its dual.
\end{examples}

\subsection{Differential graded supermanifolds}

We now recall a derived generalisation of the affine $\cC^{\infty}$-schemes of \cite{joyceAGCinfty}, in the form of a $\cC^{\infty}$ analogue of the algebraic dg manifolds of \cite[\S 2.5]{Quot}; these should not be confused with the dg manifolds of \cite{roytenbergDGMfdsStacks}, which correspond to our NQ-manifolds. The homotopy theory of such objects is studied in detail in \cite{CarchediRoytenbergHomological}. 
Our dg manifolds correspond to the affine derived manifolds of finite presentation in \cite{nuitenThesis};  of the other approaches to derived differential geometry, this formulation   is closely related to Joyce's $d$-manifolds \cite{joyceDmanIntro} (which however discard much of the derived structure), and is essentially equivalent to the approaches of \cite{BorisovNoel,Spivak} via 
\cite[Corollary 2.2.10]{nuitenThesis}.
The primary  motivation for such objects comes from obstruction theory (notably Kuranishi spaces), and they also allow for phenomena such as well-behaved non-transverse derived intersections.

\begin{definition}\label{derivedmfd}
 Define a ($\cC^{\infty}$) dg manifold $X$ to be a pair $(X^0, \O_X)$ where $X^0$ is a real differentiable  manifold and 
\[
 \O_X= (\O_{X,0} \xla{\delta} \O_{X,1}  \xla{\delta} \O_{X,2}  \xla{\delta} \ldots )
\]
is a graded-commutative $\R$-algebra in chain complexes of sheaves on $X^0$. We require that $\O_{X,0}$ be the sheaf of smooth functions on $X^0$  and that $\O_X$ be locally semi-free in the sense that the underlying graded-commutative sheaf $\O_{X, \#}$ is locally of the form $\O_{X^0}\ten \Symm_{\R}V$ for some finite-dimensional graded vector space $V= V_1 \oplus \ldots \oplus V_n$.

Write $\cC^{\infty}(X):= \Gamma(X^0,\O_X)$ for the chain complex of global sections of $\O_X$.
\end{definition}

The following notation follows \cite{stacks2}; the corresponding notation in \cite{Quot} is $\pi_0$, which we avoid because it usually denotes quotients rather than subspaces, and can lead to confusion when working with stacks or Lie groupoids.
\begin{definition}
 Given a dg manifold $X$, define the truncation $\pi^0X$ to be the $\cC^{\infty}$-subscheme (in the sense of \cite{joyceAGCinfty}) of $X^0$ defined by the ideal $\delta \O_{X,1}$. Thus the space underlying $\pi^0X$ is the vanishing locus of the vector field $\delta$.
\end{definition}

\begin{remarks}\label{derivedstackyrmk}
 Because a dg manifold enhances the structure sheaf in the chain direction, it behaves very differently from an NQ-manifold, where the sheaf is enhanced in the cochain direction. The former (derived enhancements) generalise subspaces of a manifold,  while the latter (stacky enhancements) generalise quotients; the chain and cochain structures are shadows of simplicial and cosimplicial structures, respectively.

For each dg manifold, the commutative dg algebra $\cC^{\infty}(X)$ naturally has the structure of a  dg $\cC^{\infty}$-ring in the sense of \cite{CarchediRoytenbergHomological} (i.e. $\cC^{\infty}(X)_0$ is a $\cC^{\infty}$-ring in the sense of \cite{dubuc,MoerdijkReyes}). This gives a full and faithful contravariant functor from dg manifolds to dg $\cC^{\infty}$-rings. The image includes all finitely generated cofibrant objects, but is much larger, essentially consisting of finitely generated dg $\cC^{\infty}$-rings whose K\"ahler differentials compute the cotangent complex without needing to pass to a cofibrant replacement.\footnote{There is, however, a Quillen equivalent model structure in which these are all cofibrant: see Corollary \ref{locmodelfpcor}.}

Whereas the right notion of equivalence between NQ-manifolds is quite subtle (isomorphism on $X_0$ and quasi-isomorphism on K\"ahler differentials), we can regard a morphism of dg manifolds as an equivalence if it induces a quasi-isomorphism (i.e. a homology isomorphism) between complexes of $\C^{\infty}$ functions.  Many of the  singular $\C^{\infty}$-schemes of \cite{joyceAGCinfty} are then equivalent in this sense to dg manifolds. 

Any homotopically meaningful construction then has to be formulated in such a way that it is invariant under equivalences of dg manifolds (see Appendix \ref{equivapp}). 
In particular, the manifold $X^0$ is not an invariant of the equivalence class of a dg manifold  $X=(X^0, \O_X)$, because for any submanifold $U$  of $X^0$ containing the vanishing locus $\pi^0X$, the inclusion map
\[
 (U, \O_X|_U) \to (X^0, \O_X)
\]
is an equivalence. On the other hand, the $\C^{\infty}$ scheme  $\pi^0X$ is an invariant of the equivalence class. Since most quasi-isomorphisms are not strictly invertible, this also means that we cannot usually perform constructions by just taking local charts and gluing.

\end{remarks}

\begin{examples}\label{DCritex}
Given a manifold $M$ and  a smooth section $s\co M \to E$ of a vector bundle, there is an associated dg manifold, the derived vanishing locus $\oR s^{-1}(0)$  of $s$,  given by $X^0=M$ and $\O_X$ the sheaf of sections of $\Symm_{\R}(E^*_{[-1]})$, with differential 
\begin{align*}
\cC^{\infty}(X)_{r+1} &\xra{\delta} \cC^{\infty}(X)_r\\
 \cC^{\infty}(X,\L^{r+1} E^*) &\xra{s^*}  \cC^{\infty}(X,\L^r E^*).
\end{align*}
Then $\pi^0X$ is the vanishing locus of $s$; a simple example of this form is given by the DGA $\cC^{\infty}(\R) \xra{x^2} \cC^{\infty}(\R)$ resolving the dual numbers $\R[x]/x^2$,   but in general $\O_X$ can have higher homology. Every dg manifold $X$ with $\sO_X$ generated in degrees $0,1$ arises as a derived vanishing locus.

Note that when the image of $s$ intersects the zero section transversely, the morphism $s^{-1}(\{0\}) \to \oR s^{-1}(0)$ is a quasi-isomorphism, in the sense that $\cC^{\infty}(X) \to \H_0\cC^{\infty}(X)= \cC^{\infty}(M)/(s)$ is so.

This class of examples includes the derived critical locus $\DCrit(M,f)$ of a function $f \in \cC^{\infty}(M)$, by considering the section $df$ of the cotangent bundle, as in \cite[\S 2.1]{gwilliamThesis}. The $\cC^{\infty}$-differential graded algebra  $\cC^{\infty}( \DCrit(M,f) )$ is given by the chain complex
\[
 \cC^{\infty}(M) \xla{ \lrcorner df} T_{\C^{\infty}(M)} \xla{ \lrcorner df} \L^2_{\cC^{\infty}(M)}T_{\C^{\infty}(M)} \xla{ \lrcorner df} \ldots,
\]
so $\H_0\cC^{\infty}( \DCrit(M,f) )$ consists of functions on the critical locus of $f$. Explicitly, if $M=\R^n$, then $\cC^{\infty}( \DCrit(M,f) )$ is generated by co-ordinates $x_i,\xi_i$ of chain degrees $0,1$, with $\delta \xi_i =  \frac{\pd f}{\pd x_i}$. As explained on \href{https://ncatlab.org/nlab/show/BV-BRST+formalism}{the nlab}, $\O_{\DCrit(M,f)}$ is the classical BV complex of the function $f$ on the manifold $M$.

Other examples are shifted cotangent bundles $T^*[-n]M$ of manifolds $M$ for $n\ge 0$, with $\cC^{\infty}(T^*[-n]M)$ given by the free graded-commutative algebra over $\cC^{\infty}(M)$ generated by smooth sections $T_{\cC^{\infty}(M)}:=\cC^{\infty}(M,TM)_M$ of the tangent bundle placed in chain degree $n$, and with trivial differential $\delta$. Note that $ \DCrit(M,0)= T^*[-1]M$.
\end{examples}

\begin{definition}\label{derivedsupermfd}
Define a ($\cC^{\infty}$) dg supermanifold $X$ to be a pair $(X^{0,=}, \O_X)$ where $X^{0,=}$ is a real differentiable  manifold and $\O_X= \O_{X, \ge 0}$ is a graded-commutative $\R$-algebra in super chain complexes of sheaves on $X^{0,=}$. We require that   $\O_X$ be locally semi-free in the sense that the underlying super graded-commutative sheaf $\O_{X, \#}$ is locally of the form $\O_{X^{0,=}}\ten \Symm_{\R}V$   for some finite-dimensional super  graded vector bundle $V= V_0^{\ne} \oplus V_1 \oplus \ldots \oplus V_n$ (with $V_i= V_i^{=}\oplus V_i^{\ne}$). In particular, this implies that $\O_{X,0}/(\O_{X,0}^{\ne})$ is isomorphic to the sheaf of smooth functions on $X^{0,=}$.

Write $\cC^{\infty}(X):= \Gamma(X^{0,=},\O_X)$ for the super chain complex of global sections of $\O_X$ and  $\pi^0X^= \subset X^{0,=}$ to be the $\cC^{\infty}$-subscheme  defined by the ideal  given by the image of $\delta \O_{X,1}$. 
\end{definition}

\begin{examples}\label{DCritexodd}
An example of a non-trivial dg supermanifold is given by the  shifted cotangent bundle $T^*[-n]M$ of a supermanifold $M$, with $\cC^{\infty}(T^*[-n]M)$ given by the free graded-commutative algebra over $\cC^{\infty}(M)$ generated by smooth sections $T_{\C^{\infty}(M)}$
of the tangent bundle placed in chain degree $n$, and with trivial differential $\delta$. Another example is the parity-reversed shifted cotangent bundle $\Pi T^*[-n]M$, with $\cC^{\infty}(\Pi T^*[-n]M)$ given by the free graded-commutative algebra over $\cC^{\infty}(M)$ generated by  the tangent bundle $T_{\C^{\infty}(M)}$ placed in chain degree $n$ and unequal parity, and with trivial differential $\delta$.

A more interesting example is given by the derived critical locus $\DCrit(M,f)$ of an odd function $f \in \cC^{\infty}(M)^{\ne}$ on a supermanifold $M$. The $\cC^{\infty}$-differential graded algebra  $\cC^{\infty}( \DCrit(M,f) )$ is given by the chain complex
\[
 \cC^{\infty}(M) \xla{ \lrcorner df} \Pi T_{\C^{\infty}(M)} \xla{ \lrcorner df} \L^2_{\cC^{\infty}(M)}\Pi T_{\C^{\infty}(M)} \xla{ \lrcorner df} \ldots .
\]
Explicitly, if $M=\R^{m|n}$, then $\cC^{\infty}( \DCrit(M,f) )$ is generated by co-ordinates $\{x_i,\xi_i,y_j, \eta_j\}_{1 \le i \le m, 1 \le j \le n}$ with $x_i, y_j$ of chain degree $0$, $\xi_i, \eta_j$ of chain degree $1$, $x_i, \eta_j$ of even parity and $y_i, \xi_i$ of odd parity, with $\delta \xi_i =  \frac{\pd f}{\pd x_i}$ and $\delta \eta_j =  \frac{\pd f}{\pd y_j}$.

More generally, dg supermanifolds $X$ with with $\sO_X$ generated in chain degrees $0,1$ correspond to derived vanishing loci of sections $s \co M \to E$ of supervector bundles $E$ on supermanifolds $M$. 
\end{examples}

\begin{remark}\label{htpyfdrmkd1}
The finiteness conditions for dg manifolds and supermanifolds in Definitions \ref{derivedmfd}, \ref{derivedsupermfd} have been chosen to be stricter than necessary to ease exposition. In fact, our main constructions and results will apply if we only impose finiteness conditions up to homotopy, as in \cite{poisson}. 

Instead of looking at $\cC^{\infty}(X)$ for dg supermanifolds $X$, we could take super chain commutative algebras $A$ with $A_0^{=} = \LLim_{\alpha}\cC^{\infty}(Y_{\alpha})$ for inverse systems $\{Y_{\alpha}\}_{\alpha}$ of submersions of manifolds, and $A_{\#}$ freely generated as a  $(\Z \by \Z/2)$-graded-commutative algebra over $A_0^{=}$ by projective modules $P_0^{\ne},P_1, P_2, \ldots$, possibly of infinite rank. The finiteness condition we would then impose is that the module $\Omega^1_A$ of $\cC^{\infty}$-differential forms (cf. Remark \ref{abstractOmegarmk} or \cite[\S 5.3]{joyceAGCinfty}) of $A$ is perfect, meaning that $\Omega^1_{A}\ten_A\H_0A$ is quasi-isomorphic to a finite complex of projective $\H_0A$-modules of finite rank.

This relaxation of finiteness conditions slightly complicates several formulae, giving them the form occurring in \cite{poisson}. For instance, we would have  to replace symmetric powers of tangent complexes with duals of cosymmetric powers of cotangent complexes in the definition of polyvectors.

For the corresponding algebraic notions, see \cite[Theorem 6.42]{stacks2}, where a slight generalisation of the dg schemes from \cite{Quot} is shown to be equivalent to the notion of derived schemes emerging from higher topoi in \cite{hag2}. 
\end{remark}

\begin{definition}
Given a dg supermanifold $X$, define $\Omega^1_X$ to be the sheaf of smooth $1$-forms of $\O_X$.
 This is a sheaf of super chain complexes, and we write $\Omega^p_X:= \L^p_{\O_X}\Omega^1_X$. We also write
$\Omega^p_{\cC^{\infty}(X)}$ for the super chain complex $\Gamma(X, \Omega^p_X)$ of global sections of $\Omega^p_X$.  

Define $T_X$ to be the sheaf $\hom_{\O_X}(\Omega^1_X,\O_X)$  of smooth $1$-vectors of $\O_X$, and write $T_{\cC^{\infty}(X)}$ for the super chain complex $\Gamma(X, T_X)$ of global sections; equivalently, this is the internal $\Hom$ complex $\HHom_{\cC^{\infty}(X)}(\Omega^1_{\cC^{\infty}(X)}, \cC^{\infty}(X))$.
\end{definition}
Beware that $\Omega^1_{\cC^{\infty}(X)}$ is not the module of K\"ahler differentials of the abstract super dg algebra $\cC^{\infty}(X)$, since there are constraints requiring that the derivation $d \co \cC^{\infty}(X) \to \Omega^1_{\cC^{\infty}(X)}$ restrict to a $\cC^{\infty}$-derivation on $ \cC^{\infty}(X^{0,=})$, similarly to Remark \ref{abstractOmegarmk}.

\begin{definition}\label{DRdefdg}
Given a dg supermanifold $X$,  define the de Rham complex $\DR(X)$ to be the product total super cochain complex $\Tot^{\sqcap}  \Omega^{\bt}_{\cC^{\infty}(X)}$ of the double complex
\[
 \cC^{\infty}(X) \xra{d} \Omega^1_{\cC^{\infty}(X)} \xra{d} \Omega^2_{\cC^{\infty}(X)}\xra{d} \ldots,
\]
so $\DR(X)^m = \prod_j(\Omega^{m+j}_{\cC^{\infty}(X)})_{j}$ with  total differential  $d \pm \delta$.

We define a filtration $\Fil$ on  $\DR(X)$ by setting $\Fil^p\DR(X) \subset \DR(X)$ to consist of terms $\Omega^i_{X}$ with $i \ge p$.
\end{definition}
The complex $\DR(X)$ has the natural structure of a filtered commutative DG super algebra.

\begin{remark}\label{DRrmkdg}
 Note that the quasi-isomorphism class of  $\DR(X)$ only depends on the $\C^{\infty}$-ringed space $\pi^0X^=$ (by \cite[Theorem 6.8]{taroyanDRDerivedDG}, after \cite{FeiginTsygan,bhattDerivedDR}, and noting that the ideal of $\sH_0\sO_X$ generated by $\sH_0\sO_X^{\ne}$ is nilpotent),
so that in particular the map $\DR(X) \to \DR(\pi^0X^=)$ is a quasi-isomorphism whenever $\pi^0X^=$ is a supermanifold. By  \cite[Proposition 4.11 and Theorem 5.12]{taroyanDRDerivedDG}, this is quasi-isomorphic to $\oR\Gamma(\pi^0X^=,\R)$ whenever $\pi^0X^=$ is locally smoothly contractible, and also  for certain zero sets of subanalytic equations. However, beware that the quasi-isomorphism classes of the filtrations $\Fil^p\DR(X)$ depend on the full derived structure of $X$ for $p>0$. 
\end{remark}

It is very important to note that we are taking the product, not direct sum, total complex, to ensure quasi-isomorphism invariance, but that as a consequence we do not have the local quasi-isomorphism with strictly closed forms which occurred in the NQ-manifold case, unless the complex is bounded in the de Rham direction, which only happens if  $\sO_X$ has purely even parity.

\begin{definition}\label{presymplecticdef}
Define an $n$-shifted pre-symplectic structure $\omega$ on a dg supermanifold  $X$ to be an element
\[
 \omega \in \z^{n+2}\Fil^2\DR(X)^{=}.
\]
Define a parity-reversed  $n$-shifted pre-symplectic structure to be an element
\[
 \omega \in \z^{n+2}\Fil^2\DR(X)^{\ne}.
\]
Two pre-symplectic structures are regarded as equivalent if they induce the same cohomology class in $\H^{n+2}\Fil^2\DR(X)^{=}$ (resp. $\H^{n+2}\Fil^2\DR(X)^{\ne}$).
\end{definition}

Explicitly, this means that $\omega$ is given by an infinite  sum $\omega = \sum_{i \ge 2} \omega_i$, with $\omega_i \in (\Omega^i_{X})_{i-n-2}$ (of equal or unequal parity, respectively)  and with  $d\omega_i = \delta \omega_{i+1}$.

\begin{definition}\label{derivedsympdef}
 Define a (parity-reversed) $n$-shifted symplectic structure $\omega$ on  a dg supermanifold $X$ to be a (parity-reversed) $n$-shifted pre-symplectic structure $\omega$ for which  contraction with the component $\omega_2 \in \z^n\Omega^2_{X}$ induces a quasi-isomorphism
\begin{align*}
 \omega_2^{\flat} \co T_X &\to (\Omega^1_X)_{[-n]}\quad \text{resp.}\\
\omega_2^{\flat} \co T_X &\to \Pi(\Omega^1_X)_{[-n]}.
\end{align*}
\end{definition}

Note that for dg supermanifolds, bounds on the generators  mean that  $n$-shifted symplectic structures can only exist for $n \le 0$.

\begin{examples} 
The derived critical loci of  Examples \ref{DCritex} and Examples \ref{DCritexodd} all carry natural shifted symplectic structures.
For $n\ge 0$, the prototypical example of a $(-n)$-shifted symplectic structure is  given by the  shifted cotangent bundle $T^*[-n]M$, with symplectic form $\omega$ given by the canonical closed  $2$-form $\omega_2 \in T_{\cC^{\infty}(M)}\ten_{\cC^{\infty}(M)}\Omega^1_M \subset  (\Omega^2_{T^*M[n]})_{[n]}$. 
A similar expression defines a $(-1)$-shifted symplectic structure on the derived critical locus $\DCrit(M,f)$ of an even  function. 

The prototypical example of a parity-reversed $(-n)$-shifted symplectic structure is  given by the parity-reversed shifted cotangent bundle $\Pi T^*[-n]M$, and the derived critical locus $\DCrit(M,f)$ of an odd  function carries a  parity-reversed $(-1)$-shifted symplectic structure.
\end{examples}

\subsection{Derived NQ-manifolds}\label{dgLiesn}

We now consider supermanifolds enhanced in both stacky and derived directions. For technical reasons, we do not combine the structures in a single grading, but instead work with bigraded objects. The objects $\cC^{\infty}(X)$ we consider are smooth analogues of special cases of the stacky CDGAs of \cite{poisson} or of the graded mixed cdgas of \cite{CPTVV}. For simplicity of exposition, we consider more restrictive objects, which broadly correspond to the derived $\infty$-Lie algebroids of 
\href{https://ncatlab.org/nlab/show/derived+%E2%88%9E-Lie+algebroid}{the nlab},
and are global versions of the quasi-smooth functors of \cite[Definition \ref{ddt1-scspqsdef} and Proposition \ref{ddt1-ctodtnew}]{ddt1}.\footnote{The latter also incorporate a $Q$-acyclicity condition, corresponding to the cofibrant objects of $DG^+dg_+\C^{\infty}$ as in Remark \ref{locmodelNQrmk}, but every equivalence  class under the equivalences of  Remarks \ref{weakerequivrmks} contains such objects. 
}

\begin{definition}\label{derivedlie} 
 Define a dg NQ-manifold $X$ to be a pair $(X_0^0, \O_X)$ where $X_0^0$ is a real differentiable manifold and $\O_X$ is a graded-commutative $\R$-algebra
\[
 \xymatrix{ 
\vdots \ar[d]^{\delta}  & \vdots \ar[d]^{\delta}& \\
\O_{X,1}^0 \ar[r]^{Q}\ar[d]^{\delta} & \O_{X,1}^1 \ar[r]^{Q}\ar[d]^{\delta} & \ldots\\
 \O_{X,0}^0 \ar[r]^{Q} & \O_{X,0}^1 \ar[r]^{Q} &  \ldots } 
\]
  in chain cochain complexes of sheaves on $X_0^0$. We require that $\O_{X,0}^0$ be the sheaf of smooth functions on $X_0^0$, that the cochain differential $Q \co \O_{X,0}^0 \to \O_{X,0}^1$ be a $\cC^{\infty}$-derivation  and that $\O_X$ be locally semi-free in the sense that the underlying bigraded-commutative sheaf $\O_{X,\#}^{\#}$ is locally of the form $\O_{X_0}\ten \Symm_{\R}(V)$ for some finite-dimensional bigraded vector space $V= \bigoplus_{\substack{0 \le i \le m \\ 0 \le j \le n}}V^i_j$ with $V^0_0=0$ (and $\Symm$ defined as in \S \ref{notnsn}, so parity is the mod $2$ sum of both indices).
  
  We impose the additional conditions, which we call the Artin conditions,  that the graded-commutative $\sH^{\delta}_0\sO_X^0$-algebra $\sH_0^{\delta}(\sO^{\#}_{X,\bt}) $ should be locally freely generated,
  and that $\sH_*^{\delta}(\sO_{X,\bt})$ should be Cartesian over $\sH_0^{\delta}(\sO_{X,\bt}) $ in the sense that
  the multiplication maps $\sH_j^{\delta}(\sO^0_{X,\bt})\ten_{\sH_0^{\delta}(\sO^0_{X,\bt})}  \sH_0^{\delta}(\sO^i_{X,\bt}) \to \sH_j^{\delta}(\sO_{X,\bt}^i)$ should be isomorphisms.

Write $\cC^{\infty}(X):= \Gamma(X_0^0,\O_X)$ for the chain cochain complex of global sections of $\O_X$.
\end{definition}

\begin{examples}\label{Rstex}
The most easily constructed examples of dg NQ-manifolds arise when  $\sO_{X,\#}^{\#}$ is of the form $\O_{X_0}\ten \Symm_{\R}(U\oplus V)$ for some finite-dimensional bigraded vector spaces $U= U^1_0 \oplus \ldots \oplus U^N_0$ and $V=V^0_1 \oplus \ldots \oplus V^0_m$. The Artin conditions then automatically follow because  we have $\sO^{\#}_{X,\bt} \cong \sO_X^0\ten_{\R}\Symm_{\R}(U)$  and hence   $\sH_j^{\delta}(\sO^{\#}_{X,\bt}) \cong \sH_j^{\delta}(\sO^0_{X,\bt})\ten_{\R}\Symm_{\R}(U)$.

 For a simple example, consider the case where $X^0_0$ is a point, and $\O_X$ is freely generated by $s \in \O_{X,0}^2$ and $t \in \O_{X,2}^0$. Then $\O_X=\R[s,t]$, with bigrading given by $s^it^j \in \O_{X,2j}^{2i}$.
 \end{examples}

\begin{definition}
 Given a  dg NQ-manifold $X$, define the truncation $\pi^0X\subset X^0$ to be the dg-ringed space 
$(\pi^0(X_0), \O_{X,0}/\delta \O_{X,1})$. 
\end{definition}

\begin{remarks}\label{Artincdnrmk}
 The Artin conditions of Definition \ref{derivedlie} are motivated by the structures naturally occurring on charts of derived Artin stacks defined via Definition \ref{Dstardef}. They can be interpreted as saying that $\pi^0X$ should be a singular NQ-manifold 
 (i.e. having functions locally freely generated over the $\C^{\infty}$-scheme $\pi^0X_0$) 
 and that the sheaves $\sH_n^{\delta}\sO_X$ should be Cartesian coherent $\pi^0X$-modules. As in  \cite[Examples \ref{smallet2-BGLnex} and \ref{smallet2-QCohex}]{smallet2} (cf. Example \ref{vbundleex}), the latter are the natural analogue of coherent modules for an  NQ-manifold.\footnote{Specifically, they correspond to morphisms $\pi^0X \to D_*\mathrm{Coh}$, where  $\mathrm{Coh}$  is the stack classifying coherent sheaves and $D_*$  a canonical  fully faithful functor from stacks on $\C^{\infty}$-spaces to stacks on  singular NQ-manifolds.}
   For the natural notion of equivalence for dg NQ-manifolds (see Remarks \ref{weakerequivrmks}), relaxing the Artin conditions would not yield any new equivalence classes of objects. 
 
 Those conditions have an alternative characterisation in terms of the cotangent module, corresponding to the requirement that $\bigoplus_{i>0} \H_j(\Omega^1_{A}\ten_{A}\H_0(A^0))^i$ be 
be a  projective $\H_0(A^0)$-module of finite rank when $j=0$, and vanish for $j>0$. The equivalence of conditions follows because $A^{>0}/(A^{>0}\cdot A^{>0})$ is the conormal module $\ker(\Omega^1_A\ten_AA^0 \to \Omega^1_{A^0})$.
 \end{remarks}

\begin{definition}\label{derivedsuperlie}
Define a super dg NQ-manifold  $X$ to be a pair $(X_0^{0,=}, \O_X)$ where $X_0^{0,=}$ is a real differentiable  manifold and 
$\O_X= \O_{X, \ge 0}^{\ge 0}$
is a graded-commutative $\R$-algebra in super chain cochain complexes of sheaves on $X_0^{0,=}$. We require  that the cochain differential $Q \co \O_{X,0}^{0,=} \to \O_{X,0}^{1,=}$ be a $\cC^{\infty}$-derivation  
 and that  $\O_X$ be locally semi-free in the sense that the underlying super bigraded-commutative sheaf $\O_{X,\#}^{\#}$ is locally of the form $\O_{X_0^{0,=}}\ten \Symm_{\R}(V)$   for some finite-dimensional super  bigraded vector space $V= \bigoplus_{\substack{0 \le i \le m \\ 0 \le j \le n}}(V^{i,=}_j \oplus V^{i,\ne}_j)$ with $V^{0,=}_0=0$ (and $\Symm$ defined as in \S \ref{notnsn}, so parity is the mod $2$ sum of all three indices).
  In particular this implies that  $\O_{X,0}^{0}/(\O_{X,0}^{0, \ne})$ is isomorphic to the sheaf of smooth functions on $X_0^{0,=}$,

  We impose the additional conditions (the Artin conditions) that the graded commutative $\H^{\delta}_0\sO_X^0$-superalgebra $\sH_0^{\delta}(\sO^{\#}_{X,\bt}) $ should  be locally freely generated,  and that the multiplication maps $\sH_j^{\delta}(\sO^0_{X,\bt})\ten_{\sH_0^{\delta}(\sO^0_{X,\bt})}  \sH_0^{\delta}(\sO^i_{X,\bt}) \to \sH_j^{\delta}(\sO_{X,\bt}^i)$ should be isomorphisms.
  
Write $\cC^{\infty}(X):= \Gamma(X_0^{0,=},\O_X)$ for the super chain cochain complex of global sections of $\O_X$.
\end{definition}

\begin{example}\label{exYoverg}
 If $Y=(Y^{0,=},\O_Y)$ is a dg supermanifold, with a finite-dimensional Lie superalgebra $\g$ acting on $\O_Y$, in the sense of a map from  $\g$ to the Lie superalgebra of derivations of $\O_Y$ of degree $0$ which commute with the differential $\delta$. 
 
 There is then a dg super NQ-manifold  $[Y/\g]$ given by setting $[Y/\g]^{0,=}_0= Y^{0,=}_0$ and setting $\O_{[Y/\g]}$ to be the super cochain chain complex
\[
\O_{[Y/\g]}:=(  \O_Y \xra{Q} \O_Y\ten \g^* \xra{Q} \O_Y\ten \Lambda^2\g^*\xra{Q} \ldots), 
\]
the Chevalley--Eilenberg double complex of $\g$ with coefficients in the super chain complex $\O_Y$.

In particular, if $\O_Y$ is generated in chain degrees $0,1$ then $Y$  is a derived vanishing locus of the section $s$ of a vector bundle on $ Y^0$ as in Example \ref{DCritex}, and $s$ is necessarily $\g$-equivariant. More generally, giving a dg NQ supermanifold $X$ for which $\sO_X$ is generated by $\sO_{X,0}^0$, $\sO_{X,0}^1$ and $\sO_{X,1}^0$ can be thought of as the derived vanishing locus of a section of a vector bundle $E$ on an NQ-1 supermanifold  (equivalently,  Lie superalgebroid) $X^0$, where $\O_{X,1}$ is the sheaf of sections of $E^*$ on $X^0$.
\end{example}

\begin{remark}\label{htpyfdrmkd2}
As we saw in Remark \ref{htpyfdrmkd1} for  dg supermanifolds,
the finiteness conditions for dg NQ-manifolds and super dg NQ-manifolds  in Definitions \ref{derivedlie}, \ref{derivedsuperlie} can be relaxed to hold only up to homotopy, as in \cite{poisson}. 

Instead of super  differential bigraded algebras  $\cC^{\infty}(X)$ for super dg NQ-manifolds $X$, we could take super chain cochain commutative algebras $A$, with $A^0$ as in Remark \ref{htpyfdrmkd1}, and $A^{\#}_{\#}$ freely generated as a graded algebra over $A^0_0$ by projective modules $M^i_j$ (possibly of infinite rank) with $M^{0,=}_0=0$. In addition to the finiteness condition on smooth differentials $\Omega^1_{A^0}$ from Remark \ref{htpyfdrmkd1}, we would have to impose the Artin conditions from Definitions \ref{derivedlie}, \ref{derivedsuperlie}, also requiring that $\H_0A^{\#}$ be locally finitely generated as an algebra over $\H_0A^0$.

This gives a much more flexible category of objects with which to work
(essentially corresponding to the derived Lie algebroids locally of finite presentation in \cite{nuitenThesis}), but several constructions involving duals would have to be rewritten, along the lines of the expressions in \cite{poisson}.
\end{remark} 

\begin{remark}[Filtrations and $\Z$-graded manifolds]\label{filtrmk}
 The homotopy theory of double complexes up to  levelwise $\delta$-quasi-isomorphism is equivalent to the homotopy theory of complete filtered complexes, under the functor sending a double complex $A$ to the decreasing filtration $\{\Tot^{\sqcap} A^{\ge n}\}_{n \in \Z}$ given by product total complexes. All of our constructions and conditions  for a dg super NQ-manifold $X$ can thus be rephrased in terms of the super  $\C^{\infty}$-DGA $\Tot^{\sqcap} \C^{\infty}(X)$ equipped with the filtration  $\{\Tot^{\sqcap}(\C^{\infty}(X)^{\ge n})\}_n$, although that approach is never taken systematically in the derived literature for technical reasons. For details, see \S \ref{filtsn}.
 
 For the dg NQ manifold $X$ with $\O_{X,2j}^{2i}=\R\cdot s^it^j  $ from Examples \ref{Rstex}, we have $(\Tot^{\sqcap}\sO_X)^{2n} = s^n \R \llb st \rrb$, for   $\R \llb st \rrb$ the power series ring, and $(\Tot^{\sqcap}\sO_X)^{-2n} = t^n \R \llb st \rrb$, with filtration $\Fil^m(\Tot^{\sqcap}\sO_X)= s^{\lfloor m/2\rfloor}(\Tot^{\sqcap}\sO_X)$.
 \end{remark}

\begin{definition}
Given a  super dg NQ-manifold  $X$, define $\Omega^1_X$ to be the sheaf of smooth $1$-forms of $\O_X$.
 This is a sheaf of super chain cochain complexes, and we write $\Omega^p_X:= \L^p_{\O_X}\Omega^1_X$. We also write
$\Omega^p_{\cC^{\infty}(X)}$ for the super chain cochain complex $\Gamma(X, \Omega^p_X)$ of global sections of $\Omega^p_X$.  

Define $T_X$ to be the sheaf $\hom_{\O_X}(\Omega^1_X,\O_X)$  of smooth $1$-vectors of $\O_X$, and write $T_{\cC^{\infty}(X)}$ for the super chain cochain complex $\Gamma(X, T_X)$ of global sections; equivalently, this is the internal $\Hom$ space $\cHom_{\cC^{\infty}(X)}(\Omega^1_{\cC^{\infty}(X)}, \cC^{\infty}(X))$.
\end{definition}

\begin{definition}\label{DRdefderivedLie}
Given a super dg NQ-manifold  $X$,  define the de Rham complex $\DR(X)$ to be the  product total super cochain complex $\Tot^{\sqcap}  \Omega^{\bt}_{\cC^{\infty}(X)}$ of the triple complex
\[
 \cC^{\infty}(X) \xra{d} \Omega^1_{\cC^{\infty}(X)} \xra{d} \Omega^2_{\cC^{\infty}(X)}\xra{d} \ldots,
\]
so $\DR(X)^m = \prod_{i+j-k=m}(\Omega^i_{\cC^{\infty}(X)})^j_k$ with  total differential  $d \pm Q\pm \delta$.

We define a filtration $\Fil$ on  $\DR(X)$ by setting $\Fil^p\DR(X) \subset \DR(X)$ to consist of terms $\Omega^i_{X}$ with $i \ge p$.
\end{definition}
The complex $\DR(X)$ has the natural structure of a filtered commutative DG super algebra. 

\begin{remark}\label{DRrmkdgNQ}
The filtration of $\Omega^{\bt}_X$  by powers of $\ker(\Omega^{\bt}_X \to \Omega^{\bt}_{X_0})$ has  associated graded pieces which are bounded in the cochain direction and $\Tot$-acyclic, so we have  a quasi-isomorphism $\DR(X) \to \DR(X_0)$ (but not a quasi-isomorphism on $\Fil^p$). 

 As in Remark \ref{DRrmkdg}, the quasi-isomorphism class of   $\DR(X)$ (ignoring $\Fil$) thus only depends on the $\C^{\infty}$-ringed space $\pi^0X_0^{=}$, and often corresponds to $\oR\Gamma(\pi^0X^{=},\R)$.
\end{remark}

\begin{definition}\label{presymplecticdefdLie}
Define an $n$-shifted pre-symplectic structure $\omega$ on a super dg NQ-manifold  $X$ to be an element
\[
 \omega \in \z^{n+2}\Fil^2\DR(X)^{=}.
\]
Define a parity-reversed  $n$-shifted pre-symplectic structure to be an element
\[
 \omega \in \z^{n+2}\Fil^2\DR(X)^{\ne}.
\]
Two pre-symplectic structures are regarded as equivalent if they induce the same cohomology class in $\H^{n+2}\Fil^2\DR(X)^{=}$ (resp. $\H^{n+2}\Fil^2\DR(X)^{\ne}$).
\end{definition}

Explicitly, this means that $\omega$ is given by an infinite  sum $\omega = \sum_{i \ge 2,j \ge 0} \omega_{i,j}$, with $\omega_{i,j} \in (\Omega^i_{X})^{n+2-i+j}_j$ (of equal or unequal parity, respectively)  and with  $d\omega_{ij} = \delta \omega_{i+1,j+1}\pm Q \omega_{i+1,j}$, for the de Rham differential $d$.

\begin{definition}\label{symplecticdefdgLie}
 Define a (parity-reversed) $n$-shifted symplectic structure $\omega$ on $X$ to be a (parity-reversed) $n$-shifted pre-symplectic structure $\omega$ for which  contraction with the component $\omega_{2,\bt} \in \z^n\Omega^2_{X}$ induces a quasi-isomorphism
\begin{align*}
 \omega_{2,\bt}^{\flat} \co \Tot (T_X\ten_{\O_X}\O_X^0) &\to \Tot (\Omega^1_X\ten_{\O_X}\O_X^0)^{[n]}\quad \text{resp.}\\
\omega_{2,\bt}^{\flat} \co \Tot (T_X\ten_{\O_X}\O_X^0) &\to \Pi\Tot (\Omega^1_X\ten_{\O_X}\O_X^0)^{[n]}.
\end{align*}
on total complexes.
\end{definition}

Note that if $X$ carries a $0$-shifted symplectic structure, then either it is a supermanifold  or it has non-trivial enhancements in both stacky and derived directions, since we need the chain and cochain generators of $\cC^{\infty}(X)$ to balance each other. The motivation for this formulation of non-degeneracy is the same as in Remarks \ref{NQnondegrmks}, with  \cite[Lemma \ref{poisson-binondeglemma}]{poisson} again guaranteeing good behaviour under tensor operations.

\begin{examples}
Shifted cotangent bundles of NQ-manifolds give rise to dg NQ-manifolds with natural shifted symplectic structures, similarly to \cite{calaqueShiftedCotangent}. For instance, the NQ-manifold $B\g$ (i.e. $[*/\g]$) associated to a Lie algebra $\g$ has derived cotangent bundle $T^*B\g$ given by $[(\g^*[-1])/\g]$, which carries a canonical $0$-shifted symplectic structure. Explicitly, the manifold underlying $ T^*B\g$ is a point, with $\O_{T^*B\g}$ freely generated as a bigraded algebra by $\g$ in chain degree $1$ and $\g^*$ in cochain degree $1$, with chain differential $\delta=0$ and cochain differential $Q$ given by the Chevalley--Eilenberg complex of $\g$ acting on $\L^*\g$. The symplectic structure consists of the single element $\omega_2 \in (\Omega^2_{T^*B\g})^1_1$ corresponding to the image under $d \ten d$ of the  canonical element of $\g\ten \g^*$. 


A more interesting example of a dg NQ-manifold with a $0$-shifted symplectic structure is given by the infinitesimal analogue of the derived Hamiltonian reduction from \cite{safronovPoissonRednCoisotropic}. Given a symplectic manifold $Y$ equipped with an infinitesimal action of a Lie group $\g$ preserving the symplectic form, and a $\g$-equivariant map $\mu \co Y \to \g^*$ which is Hamiltonian for the $\g$-action, we can first form the derived vanishing locus $\oR \mu^{-1}(0)$ as in Example \ref{DCritex} (a dg manifold), then form the dg NQ-manifold $[\oR \mu^{-1}(0)/\g]$ as in Example \ref{exYoverg}.
Functions on  $[\oR \mu^{-1}(0)/\g]$ look like
%
\[
 \xymatrix{
 \vdots \ar[d] & \vdots \ar[d] & \vdots \ar[d] \\
 \O_Y \ten \L^2\g  \ar[r]_-{Q} \ar[d]_{\lrcorner \mu} & \O_Y\ten\L^2\g\ten  \g^* \ar[r]_-{Q}\ar[d]_{\lrcorner \mu} & \O_Y\ten \L^2\g \ten \L^2\g^* \ar[r]_-{Q} \ar[d]_{\lrcorner \mu}& \ldots\\
 \O_Y \ten \g \ar[r]_-{Q} \ar[d]_{ \mu} & \O_Y\ten\g\ten  \g^*\ar[r]_-{Q} \ar[d]_{ \mu} & \O_Y\ten \g \ten \L^2\g^* \ar[r]_-{Q} \ar[d]_{ \mu} & \ldots \\
   \O_Y \ar[r]_-{Q} & \O_Y\ten \g^* \ar[r]_-{Q} & \O_Y\ten \L^2\g^* \ar[r]_-{Q} & \ldots,
}
\]
and it carries a $0$-shifted symplectic structure given by the $2$-form which combines the symplectic form on $Y$  with the pairing of $\g$ and $\g^*$. The Hamiltonian condition reduces to the condition that the form is closed under $Q \pm \delta$. Experts may recognise similarities between this construction and \cite{stasheffHomologicalRednPoisson}.
\end{examples}

\begin{definition}\label{productdef}
Given dg NQ-manifolds $X,Y$, define the product $X \by Y$ to be the NQ-manifold with underlying manifold $(X \by Y)_0^{0,=}= X_0^{0,=} \by Y_0^{0,=}$, and with structure sheaf given by pullback
\[
\O_{X \by Y}:= ((\pr_1^{-1}\O_X)\ten (\pr_2^{-1}\O_Y))\ten_{((\pr_1^{-1}\O_{X_0^{0,=}})\ten_{\R}(\pr_2^{-1}\O_{Y_0^{0,=}}))}  \O_{(X \by Y)_0^{0,=}},  
\]
for the projection maps $\pr_1 \co X_0^{0,=} \by Y_0^{0,=} \to X_0^{0,=}$ and $\pr_2\co X_0^{0,=} \by Y_0^{0,=} \to Y_0^{0,=}$.
\end{definition}

\begin{example}[Transgression]\label{transgressex}
A major source of shifted symplectic structures is the \cite{AKSZ}-inspired transgression procedure of \cite[Theorem 2.5]{PTVV}. Although the  finiteness it requires on the source $X$ is more common in algebraic settings, being satisfied by proper schemes, and is very rare for differentiable manifolds, it is relatively common for NQ manifolds. The procedure translates into our setting as follows.

Take a dg NQ manifold $X$ equipped with a trace element $\tau \in \z^{-d}\Tot \cD_c(X)$, a closed element of degree $-d$ in the total complex of the double complex $\cD_c(X):=\cHom_{\C^{\infty}(X^0_0)}(\C^{\infty}(X), \cD_c(X^0_0))$ of compactly supported distributions on $X$ (i.e. the continuous dual of $\C^{\infty}(X)$). For example, if $X$ is the NQ manifold associated to the tangent Lie algebroid of a compact orientable $d$-manifold $N$, then we can take $\tau$  to be  $\int_N \co  \Omega^d(N) \to \R$.

For any vector bundle $E$ on any dg NQ manifold $M$, we then have an induced map $\tau \co \Tot^{\sqcap} \C^{\infty}(M \by X, \pr_1^{-1}E)\to \Tot^{\sqcap} \C^{\infty}(M, E)[-d]$.  Consequently, for each morphism $f \co M \by X \to Y$, we have an induced chain map 
\[
 \Fil^2\DR(Y) \xra{f^{\sharp}} \Fil^2\DR(M \by X) \to  \Tot^{\sqcap}(\Omega^{\ge 2}_{\cC^{\infty}(M \by X)/\cC^{\infty}(X)}) \xra{\tau}  \Fil^2\DR(M)[-d],
\]
associating an $(n-d)$-shifted pre-symplectic structure on $M$ to each $n$-shifted pre-symplectic structure on $Y$, where $\Omega^{1}_{B/A}:= \coker( \Omega^1_A\ten_AB \to \Omega^1_B)$, and $\Omega^p_{B/A}$ is its $p$th alternating power.

Comparison of tangent spaces then gives conditions for symplectic structures on $Y$ to yield symplectic structures on $M$, which are roughly speaking satisfied if $\tau$ induces a perfect pairing on cohomology (Poincar\'e duality in our example above) and  $M$ is locally diffeomorphic to the functor of maps from $X$ to $Y$.
\end{example}

\section{Shifted Poisson structures on super derived NQ-manifolds}\label{poisssn}

This section is adapted from \cite{poisson}, transferring results from the algebraic to the smooth setting. Whereas \cite{poisson} begins with derived affines (analogous to dg manifolds), we begin with super NQ-manifolds, since these are more likely to be familiar to readers. The two cases behave similarly, both having simplifications compared with the general case (super dg NQ-manifolds), since they do not require us to deal with chain and cochain structures simultaneously. There are however slight differences in the notion of non-degeneracy, which we have to adapt from \cite[\S \ref{poisson-Artinsn}]{poisson} (rather than \cite[\S \ref{poisson-affinesn}]{poisson} as we would for dg manifolds).

\subsection{Shifted Poisson structures on super NQ-manifolds}\label{poisssnLie}

For now, we fix a super NQ-manifold  $X=(X_0^{=}, \O_X)$.

\subsubsection{Polyvector fields}\label{polsnLie}

\begin{definition}\label{poldefLie}   
We  define the super cochain complex of $n$-shifted polyvector fields on $X$ by
\[
 \widehat{\Pol}(X,n):= \prod_i \Symm_{\cC^{\infty}(X)}^i(T_{\cC^{\infty}(X)}^{[-n-1]}), 
\]
with graded-commutative  multiplication following the usual conventions for symmetric powers.  
Similarly, define the super cochain complex of $n$-shifted polyvector fields of reversed parity  on $X$ by
\[
 \widehat{\Pol}(X,\Pi n):= \prod_i\Symm_{\cC^{\infty}(X)}^i(\Pi T_{\cC^{\infty}(X)}^{[-n-1]})
\]

The Lie bracket on $T_{\cC^{\infty}(X)}$ then extends to give a bracket (the Schouten--Nijenhuis bracket)
\[
[-,-] \co \widehat{\Pol}(X,n)\by \widehat{\Pol}(X,n)\to \widehat{\Pol}(X,n)^{[-n-1]},
\]
of equal parity
determined by the property that it is a bi-derivation with respect to the multiplication operation. 
Similarly, $\widehat{\Pol}(X,\Pi n)$ has a Lie bracket bi-derivation of cochain degree $-n-1$ and unequal parity.

Thus  $\widehat{\Pol}(X,n)$ has the natural structure of a super $P_{n+2}$-algebra  (i.e. an $(n+1)$-shifted Poisson algebra in super cochain complexes), so $\widehat{\Pol}(X,n)^{[n+1]}$  is a super differential graded Lie algebra (DGLA) over $\R$. In particular, the subcomplexes   $\widehat{\Pol}(X,n)^{[n+1],=}$ and  $\widehat{\Pol}(X,\Pi n)^{[n+1],\ne}$ are DGLAs over $\R$. 

\begin{remark}
 Here, we follow the conventions of \cite{melaniPoisson} for $P_k$-algebras, so they carry a graded-commutative multiplication of degree $0$, and a graded Lie bracket of degree $1-k$, both of equal parity. In particular, this means that if we commute the $\Z$-action to a $\Z/2$-action, then a $P_k$-algebra is a Poisson algebra for odd $k$, and a Gerstenhaber algebra for even $k$. 
\end{remark}

Note that the cochain differential $Q$ on $\widehat{\Pol}(X,n)$ (resp. $\widehat{\Pol}(X,\Pi n)$)  can be written as $[Q,-]$, where $Q \in \widehat{\Pol}(X,n)^{n+2,=}$ (resp. $\widehat{\Pol}(X,\Pi n)^{n+2,\ne} $) is the element corresponding to the derivation $Q \in (T_{\cC^{\infty}(X)})^1$.
\end{definition}

\begin{definition}\label{Fdef}
Define  decreasing filtrations $\Fil$ on  $\widehat{\Pol}(X,n)$  and $\widehat{\Pol}(X,\Pi n)$ by 
\begin{align*}
 \Fil^i\widehat{\Pol}(X,n)&:= \prod_{j \ge i} \Symm^j_{\cC^{\infty}(X)}(T_{\cC^{\infty}(X)}^{[-n-1]});\\
\Fil^i\widehat{\Pol}(X,\Pi n)&:= \prod_{j \ge i}  \Symm^j_{\cC^{\infty}(X)}(\Pi T_{\cC^{\infty}(X)}^{[-n-1]}) ;
\end{align*}
this has the properties that $\widehat{\Pol}(X,n)= \Lim_i \widehat{\Pol}(X,n)/\Fil^i$, with $[\Fil^i,\Fil^j] \subset \Fil^{i+j-1}$, $Q \Fil^i \subset \Fil^i$, and $\Fil^i \Fil^j \subset \Fil^{i+j}$, and similarly for $\widehat{\Pol}(X,\Pi n)$.
\end{definition}

Observe that this filtration makes $\Fil^2\widehat{\Pol}(X,n)^{[n+1],=}$ and $\Fil^2\widehat{\Pol}(X,\Pi n)^{[n+1],\ne}$ into pro-nilpotent DGLAs.

\subsubsection{Poisson structures}

\begin{definition}\label{mcdef}
 Given a   DGLA $(L,d)$, define the the Maurer--Cartan set by 
\[
\mc(L):= \{\omega \in  L^{1}\ \,|\, d\omega + \half[\omega,\omega]=0 \in  \bigoplus_n L^{2}\}.
\]
\end{definition}

\begin{definition}\label{poissdef0Lie}
Define an   $n$-shifted Poisson structure on $X$ to be an element of
\[
 \mc(\Fil^2 \widehat{\Pol}(X,n)^{[n+1],=}),
\]
and an $n$-shifted Poisson structure of reversed parity on $X$ to be an element of
\[
 \mc(\Fil^2 \widehat{\Pol}(X,\Pi n)^{[n+1],\ne}).
\]

Regard two  $n$-shifted  Poisson structures  as equivalent if they are gauge equivalent as Maurer--Cartan elements (cf. \cite{Man}), i.e. if they lie in the same orbit for the gauge action on the Maurer--Cartan set of   the formal group  $\exp(\Fil^2 \widehat{\Pol}(X,n)^{n+1})$ corresponding to the pro-nilpotent  Lie algebra $\Fil^2 \widehat{\Pol}(X,n)^{n+1}$.
\end{definition}

\begin{remark}
Observe that $n$-shifted  Poisson structures consist of infinite sums $\pi = \sum_{i \ge 2}\pi_i$ with polyvectors
\[
 \pi_i \in \Symm^i_{\cC^{\infty}(X)}(T_{\cC^{\infty}(X)}^{[-n-1]})^{n+2}
\]
 satisfying $Q(\pi_i) + \half \sum_{j+k=i+1} [\pi_j,\pi_k]=0$. This is precisely the condition which ensures that $\pi$ defines an $L_{\infty}$-algebra structure on $\cC^{\infty}(X)^{[n]}$. It then makes $\cC^{\infty}(X)$ into a $\hat{P}_{n+1}$-algebra in the sense of  \cite[Definition 2.9]{melaniPoisson}; in \cite{CattaneoFelder} these are referred to as $P_{\infty}$-algebras in the case $n=0$. 
In our setting, however, we have more than just an abstract $\hat{P}_{n+1}$-algebra structure, since we have a $\cC^{\infty}$-ring and $\cC^{\infty}$ derivations.

Pre-existing names for $0$-shifted Poisson NQ-manifolds include flat homotopy Poisson manifolds, $P_{\infty}$-manifolds and homotopy Poisson manifolds. 
\end{remark}

\begin{example}
For $n=0$ and $X$ a manifold, Definition \ref{poissdef0Lie} recovers the usual notion of a Poisson structure, as we necessarily have $\pi=\pi_2$ for degree reasons, and the Maurer--Cartan equation reduces to the Jacobi identity.
\end{example}

\begin{example}\label{exCasimir}
As in \cite[Examples \ref{poisson-2PBG}]{poisson}, for any manifold $M$ equipped with an action of a finite-dimensional real Lie algebra $\g$,  we may consider the NQ-manifold $[M/\g]=(M, \mathrm{CE}(\g, \C^{\infty}(M)))$, where $\mathrm{CE}(\g,-)$ denotes the Chevalley--Eilenberg complex of a $\g$-representation. Its tangent complex is then given by
\[
T_{\cC^{\infty}([M/\g])}:= \mathrm{CE}(\g, \cone( \g \ten \C^{\infty}(M)) \to T_{\cC^{\infty}(M)}), 
\]
which is concentrated in cohomological degrees $\ge -1$.
Degree restrictions thus show that the set of $2$-shifted Poisson structures is given by 
\[
 \{ \pi \in (S^2\g \ten \C^{\infty}(M))^{\g} ~:~ [\pi, a]=0 \in \g \ten \C^{\infty}(M) ~\forall a \in \C^{\infty}(M)\}.
\]
In fact, this set is a model for the space $\cP([M/\g],2)$ of Poisson structures in Definition \ref{poissdefLie} below, there being no automorphisms or higher automorphisms since the DGLA has no terms in non-positive degrees.

A generalisation of this example to Lie pairs is given in \cite{BandieraChenStienonXu}.
Similar expressions hold for Lie algebroids $A$ on $M$, replacing $S^k\g \ten \C^{\infty}(M) $ with $\Gamma(M,S^kA)$.
Specialising to the case where $M$ is a point, the expression above says that $2$-shifted Poisson structures on the NQ-manifold $B\g$ of Examples \ref{exBg} are given by  $ (S^2\g)^{\g}$, the set of quadratic Casimir elements. 
\end{example}

\begin{example}\label{exQuasiLie}
Meanwhile, $1$-shifted Poisson structures on $[M/\g]$ are given by pairs $(\varpi,\phi)$ with 
\[
 \varpi \in (\g \ten T_{\cC^{\infty}(M)}) \oplus   (\g^*\ten \L^2\g \ten \C^{\infty}(M)), \quad \quad \phi \in \L^3\g \ten \C^{\infty}(M)
\]
 satisfying,  for $Q\in \g^*\ten T_{\cC^{\infty}(M)}$ corresponding to the Chevalley--Eilenberg derivative, 
\[
\{Q, \varpi\}=0, \quad \quad \half \{\varpi,\varpi\} + \{Q, \phi\}=0, \quad\quad [\varpi, \phi]=0,
\]
where  $\{-,-\}$ is the shifted Schouten--Nijenhuis bracket. This characterisation also generalises to Lie algebroids. As explained in \cite[Theorem 3.15]{safronovPoissonLie} (which uses the more involved formulation of shifted Poisson structures from \cite{CPTVV}), this structure  is just the same as a quasi-Lie bialgebroid in the sense of \cite[Definition 4.6]{IPLGX} (after \cite[\S 3]{roytenbergQuasiLie}), with $\varpi$ the $2$-differential and $\phi$ its curvature. 

Since the DGLA of $1$-shifted polyvectors is concentrated in non-negative cohomological degrees, the space $\cP([M/\g],1)$ has no higher homotopy groups, but it does have non-trivial fundamental groups at each point.
As in \cite[Definition 3.4 and Theorem 3.15]{safronovPoissonLie}, we can calculate morphisms in the fundamental groupoid via gauge transformations in the DGLA of polyvectors; these are given by elements of degree $0$ in the DGLA, i.e. by twists $\lambda \in (\L^2\g \ten \C^{\infty}(M))$, 
which  send $(\varpi, \phi)$ to 
$(\varpi +\{Q,\lambda\}, \phi + \{\lambda, \varpi\} + \half\{\lambda, \{Q,\lambda\}\})$. 
\end{example}

\begin{example}
A special case of the previous example is given by taking $M$ to be a point. As in  \cite[Theorem 2.8]{safronovPoissonLie}, a $1$-shifted Poisson structure on $B\g$  is then the same as a quasi-Lie bialgebra structure on $\g$, with the condition $Q(\varpi)=0$ giving compatibility of bracket and cobracket while the condition $\half \{\varpi,\varpi\} + Q(\phi)=0$ says that the trivector $\phi$ measures the failure of the cobracket $\varpi$ to satisfy the Jacobi identity. 
\end{example}

For constructions involving functoriality, gluing or descent in \S \ref{descentsn}, we will need to keep track of automorphisms of Poisson structures, including higher automorphisms. To this end, we now define a whole simplicial set (or equivalently, a topological space or $\infty$-groupoid) of Poisson structures.
 As for instance in \cite[\S 8.1]{W}, a simplicial set is a diagram
\[
\xymatrix@1{ Z_0 \ar@{.>}[r]|{\sigma_0}& \ar@<1ex>[l]^{\pd_0} \ar@<-1ex>[l]_{\pd_1} Z_1 \ar@{.>}@<0.75ex>[r] \ar@{.>}@<-0.75ex>[r]  & \ar[l] \ar@/^/@<0.5ex>[l]^{\pd_0} \ar@/_/@<-0.5ex>[l]_{\pd_2} 
Z_2 &  &\ar@/^1pc/[ll]^{\pd_0} \ar@/_1pc/[ll]_{\pd_3} \ar@{}[ll]|{\cdot} \ar@{}@<1ex>[ll]|{\cdot} \ar@{}@<-1ex>[ll]|{\cdot}  Z_3 & \ldots&\ldots,}
\]
of sets, with various relations between the face maps $\pd_i$ and the degeneracy maps $\sigma_i$ such as $\sigma_i\pd_i=\id$. The main motivating examples are given by taking $Z_n$ to be the set of continuous maps to a topological space from the geometric $n$-simplex
\[
|\Delta^n|:=\{ x \in \R_+^{n+1}\,:\, \sum_{i=0}^nx_i=1\}.
\]
Path components of the simplicial set will correspond to equivalence classes of Poisson structures, fundamental groups to automorphisms of Poisson structures up to homotopy equivalence, and higher homotopy groups to  higher automorphisms. Such homotopy groups come from negative degree cohomology in the DGLA of shifted polyvectors, and it is the possible non-triviality of these higher automorphism groups which complicates gluing arguments.

\begin{definition}\label{mcPLdef}
Following \cite{hinstack}, define the Maurer--Cartan space $\mmc(L)$ (a simplicial set) of a nilpotent  DGLA $L$ by
\[
 \mmc(L)_n:= \mc(L\ten_{\Q} \Omega^{\bt}(\Delta^n)),
\]
where 
\[
\Omega^{\bt}(\Delta^n)=\Q[t_0, t_1, \ldots, t_n,d t_0, d t_1, \ldots, d t_n ]/(\sum t_i -1, \sum d t_i)
\]
is the commutative dg algebra of de Rham polynomial forms on the $n$-simplex, with the $t_i$ of degree $0$.
\end{definition}

\begin{definition}
Given an inverse system $L=\{L_{\alpha}\}_{\alpha}$ of nilpotent DGLAs, define
\[
 \mc(L):= \Lim_{\alpha} \mc(L_{\alpha}) \quad  \mmc(L):= \Lim_{\alpha} \mmc(L_{\alpha}).
\]
Note that  $\mc(L)= \mc(\Lim_{\alpha}L_{\alpha})$, but $\mmc(L)\ne \mmc(\Lim_{\alpha}L_{\alpha}) $. 
\end{definition}

The homotopy groups  of this Maurer--Cartan space are closely related to the cohomology of $L$; for details, see Appendix \ref{towersn}. 

\begin{definition}\label{poissdefLie}
Define the space $\cP(X,n)$ of  $n$-shifted Poisson structures on $X$ to be given by the simplicial 
set
\[
 \cP(X,n):= \mmc( \{\Fil^2 \widehat{\Pol}(X,n)^{[n+1]}/\Fil^{i+2}\}_i).
\]
%

For the space $\cP(X,\Pi n)$ of $n$-shifted Poisson structures with reversed parity, replace $\widehat{\Pol}(X,n)^{=}$ with $\widehat{\Pol}(X,\Pi n)^{\ne}$.
\end{definition}

Thus observe that  $n$-shifted Poisson structures are elements of $\cP_0(X,n)$. Since gauge transformations and polynomial de Rham homotopies both give rise to path objects in a suitable model category of pro-nilpotent DGLAs, two Poisson structures define the same class in the  set $\pi_0\cP(X,n)$ of path components if and only if they are equivalent in the sense of Definition \ref{poissdef0Lie}.

\begin{definition}
We say that an  $n$-shifted Poisson structure (resp. $n$-shifted Poisson structure of reversed parity) $\pi = \sum_{i \ge 2}\pi_i $ 
is non-degenerate if  contraction with  $\pi_2 \in \Symm^2_{\cC^{\infty}(X)}(T_{\cC^{\infty}(X)}^{[-n-1]})^{[n+2],=} $ (resp. $\Symm^2_{\cC^{\infty}(X)}(\Pi T_{\cC^{\infty}(X)}^{[-n-1]})^{[n+2],\ne} $)  
induces a quasi-isomorphism 
\begin{align*}
 \pi_2^{\sharp}\co  (\Omega^1_X\ten_{\O_X}\O_X^0)^{[n]} &\to T_X\ten_{\O_X}\O_X^0, \quad \text{resp.}\\
\pi_2^{\sharp}\co  (\Omega^1_X\ten_{\O_X}\O_X^0)^{[n]} &\to \Pi T_X\ten_{\O_X}\O_X^0;
\end{align*}
 since $Q(\pi_2)=0$, these are automatically chain maps by the tensor-Hom adjunction for chain complexes.

Define $\cP(X,n)^{\nondeg}\subset \cP(X,n)$ and  $\cP(X,\Pi n)^{\nondeg}\subset \cP(X,\Pi n)$ to consist of non-degenerate elements --- these are unions of path-components.
\end{definition}

\begin{example}
 In this sense, \cite[Theorem 2.2]{CattaneoFelder} gives a $0$-shifted Poisson structure on the NQ-manifold corresponding to the Lie algebroid given by the conormal bundle on a co-isotropic submanifold. More examples of this form are given in \cite[\S 7.2]{PymSafronov}, since $(n+1)$-shifted Lagrangians (see \S \ref{Lagsn}) are $n$-shifted Poisson, by \cite[Theorem 4.23]{MelaniSafronovII}.
\end{example}


\subsubsection{Equivalence of non-degenerate Poisson and symplectic structures}\label{compsnLie}

We can regard $\Fil^2\DR(X)^{[n+1],=}$ and $\Fil^2\DR(X)^{[n+1],\ne}$  as  filtered  DGLAs with trivial bracket. Such abelian DGLAs $A$ have the property that $\mc(A)= \z^{1}A$.   The following definition thus generalises Definition \ref{presymplecticdefLie}

\begin{definition}\label{PreSpdefLie}
Define the spaces of $n$-shifted pre-symplectic structures  and of parity-reversed  $n$-shifted pre-symplectic structures
 on a super NQ-manifold  $X$ by
\begin{align*}
\PreSp(X,n)&:= \mmc( \{\Fil^2\DR(X)^{[n+1],=}/\Fil^{i+2}\}_i) \\
\PreSp(X,\Pi n)&:= \mmc( \{\Fil^2\DR(X)^{[n+1],\ne}/\Fil^{i+2}\}_i). 
\end{align*}
Set $\Sp(X,n) \subset \PreSp(X,n)$ and $\Sp(X,\Pi n) \subset \PreSp(X,\Pi n)$ to consist of the symplectic structures --- these subspaces are  unions of path-components.
\end{definition}

Note that the spaces $\PreSp(X,n)$ and $\PreSp(X,\Pi n)$ are canonically weakly  equivalent to Dold--Kan denormalisations of good truncations of the equal and unequal parity summands of $\Fil^2\DR$, so their homotopy groups are  just given by 
\begin{align*}
\pi_i\PreSp(X,n)&\cong   \H^{n+2-i}\Fil^2\DR(X)^{=}\\ 
\pi_i\PreSp(X,\Pi n)&\cong \H^{n+2-i}\Fil^2\DR(X)^{\ne}. 
\end{align*}


\begin{theorem}\label{compatthmLie}
For a super NQ-manifold $X$, there are canonical weak equivalences
\[
 \Sp(X,n) \simeq \cP(X,n)^{\nondeg}\quad \quad \Sp(X,\Pi n) \simeq \cP(X,\Pi n)^{\nondeg}
\]
of simplicial sets.

In particular,  the sets of equivalence classes of (parity-reversed) $n$-shifted symplectic structures and of (parity-reversed) non-degenerate $n$-shifted Poisson structures on $X$ are isomorphic. 
 \end{theorem}
\begin{proof}
The proof of \cite[Corollary \ref{poisson-compatcor2}]{poisson} adapts, \emph{mutatis mutandis}. We now outline the main steps. The passage from non-degenerate Poisson structures to symplectic structures proceeds along the lines of the more specific cases considered in \cite[Proposition 6.4]{KosmannSchwarzbachMagriPN}, \cite[Proposition 2 and Theorem 2]{KhudaverdianVoronov} and \cite[Proposition 5.2.2]{bruceGeomObjects}. However, in those cases, the differential is $0$, so homotopical considerations do not arise. Other related constructions can be found in \cite{KosmannSchwarzbachLaurentGengoux,KosmannSchwarzbachDerivedBrackets}.

Each Poisson structure $\pi \in \cP(X,n)$ (resp. $\pi \in \cP(X,\Pi n)$) gives rise to a Poisson cohomology complex
\[
 \widehat{\Pol}_{\pi}(X,n) \quad\text{ (resp. }\quad \widehat{\Pol}_{\pi}(X,\Pi n)\text{)},
\]
defined as the super cochain complex  given by the derivation  $Q+[\pi,-]$ acting on $\widehat{\Pol}(X,n) $ (resp. $\widehat{\Pol}(X,\Pi n) $). There is also a canonical element $\sigma(\pi) \in \z^{n+2}\widehat{\Pol}_{\pi}(X,n)^{=}$ (resp. $\z^{n+2}\widehat{\Pol}_{\pi}(X,\Pi n)^{\ne}$) given by 
\[
 \sigma(\pi)= \sum_{i \ge 2}  (i-1)\pi_i,
\]
for $\pi_i \in \Symm^iT$.

The key construction is then given by the ``compatibility map''  
\begin{align*}
 \mu(-,\pi) \co \DR(X) &\to \widehat{\Pol}_{\pi}(X,n) \quad\text{ (resp. }\quad \widehat{\Pol}_{\pi}(X,\Pi n)\text{)}\\
a df_1 \wedge \ldots df_p &\mapsto a[\pi,f_1]\ldots [\pi,f_p]
\end{align*}
of filtered super cochain complexes. When $\pi$ is non-degenerate, this map is necessarily a quasi-isomorphism, and the symplectic structure associated to $\pi$ is given by 
\[
 \mu(-,\pi)^{-1}\sigma(\pi) \in \H^{n+2}\Fil^2\DR(X)^{=} \quad\text{ (resp. }\quad \H^{n+2}\Fil^2\DR(X)^{\ne}\text{)}. 
\]
In fact, \cite{KhudaverdianVoronov} observe that the inverse map $\mu(\pi,-)^{-1}$ is a Legendre transform. 

Establishing that this gives an equivalence between symplectic and Poisson structures relies on obstruction theory associated to filtered DGLAs, building the Poisson form $\pi=\pi_2+\pi_3 + \ldots$ inductively from the  symplectic form $\omega=\omega_2+\omega_3 + \ldots$ by solving the equation $\mu(\omega, \pi)\simeq \sigma(\pi)$ uniquely up to coherent homotopy. Formally, the construction gives weak equivalences $\Sp(X,n) \la \Comp(X,n)^{\nondeg} \to \cP(X,n)^{\nondeg}$, where  $\Comp(X,n)^{\nondeg}$ is the space consisting of triples $(\omega, \pi,h)$, with $\omega \in \Sp(X,n)$, $\pi \in \cP(X,n)^{\nondeg}$ and $h$ a homotopy between $\mu(\omega,\pi)$ and $\sigma(\pi)$; for a readable summary of the argument from \cite{poisson}, see \cite[\S 2.5]{safronovPoissonLectures}.
%
\end{proof}
 
\begin{examples}\label{strictcompex}
 Note that if an $n$-shifted symplectic structure $\omega$ consists of just a $2$-form (i.e.  $\omega = \omega_2$) and if an $n$-shifted Poisson structure $\pi$ consists of just a bivector (i.e.  $\pi = \pi_2$), then the equation $\mu(\omega, \pi)\simeq \sigma(\pi)$ is equivalent to $\pi^{\sharp} \circ \omega^{\flat} \circ \pi^{\sharp} \simeq \pi^{\sharp}$. 
 
 In particular, if $\omega = \omega_2$ and $\omega_2^{\flat}$ is an isomorphism (not just a quasi-isomorphism), then the corresponding non-degenerate $n$-shifted Poisson structure $\pi=\pi_2$ is determined by setting $\pi_2^{\sharp}:=(\omega^{\flat})^{-1}$. The theorem implies that any other  solution $\pi'$ of the equation $\mu(\omega, \pi')\simeq \sigma(\pi')$  is homotopic to $\pi$, with that homotopy being unique up to a contractible space of choices.

 At the other extreme, the shifted symplectic structure $0$ on a tangent Lie algebroid as in Examples \ref{exBg} corresponds to the  shifted Poisson structure $0$, since $\mu(0,0)=0=\sigma(0)$, which is also non-degenerate because the cotangent complex is acyclic.
 \end{examples}

\subsection{Shifted Poisson structures on super dg NQ-manifolds}\label{poisssndLie}

The formulation of shifted Poisson structures for super dg NQ-manifolds follows along the same lines as the construction for stacky CDGAs in \cite[\S \ref{poisson-Artinsn}]{poisson}. The main subtlety is to combine the two gradings in an effective way. In this section, we fix a super dg NQ-manifold $X=(X^{0,=}_0,\O_X)$.

The definition of an $n$-shifted Poisson structure on $X$ is fairly obvious: it is a Lie bracket of total cochain degree $-n$ on $\O_X$,
 or rather an $L_{\infty}$-structure in the form of a  sequence $[-]_m$ of $m$-ary operations of cochain degree $1-(n+1)(m-1)$. However, the precise formulation  (Definition \ref{bipoldef}) is quite subtle, involving  lower bounds on the cochain degrees of the operations. See \S \ref{filtsn} for a little more detail on that perspective.

\begin{definition}\label{hatTotdef}
 Given a chain cochain complex $V$, define the sum-product cochain complex $\hatTot V \subset \Tot^{\sqcap}V$ by
\[
(\hatTot V)^m := (\bigoplus_{i < 0} V^i_{i-m}) \oplus (\prod_{i\ge 0}   V^i_{i-m})
\]
with differential $Q \pm \delta$.
\end{definition}

An alternative description of $\hatTot V$ is as the completion of $\Tot V$ with respect to the filtration $ \{\Tot \sigma^{\ge m }V\}_m$, where $\sigma^{\ge m}$ denotes brutal truncation in the cochain direction. In fact, we can write
\[
 \Lim_m \LLim_n \Tot( (\sigma^{\ge - n }V)/(\sigma^{\ge m }V)) = \hatTot V =  \LLim_n \Lim_m\Tot( (\sigma^{\ge -n }V)/(\sigma^{\ge m }V)).
\]
 The latter description also shows that there is a canonical map $(\hatTot U)\ten( \hatTot V) \to \hatTot (U \ten V)$ --- the same is not true of the product total complex $\Tot^{\sqcap}$ in general.

Unlike $\Tot$ and $\Tot^{\sqcap}$, the sum-product total complex always sends levelwise $\delta$-quasi-isomorphisms to quasi-isomorphisms.
 
%


\begin{definition}\label{bipoldef}
Given a  super dg NQ-manifold $X$, define the super cochain complex of $n$-shifted polyvector fields (resp. $n$-shifted polyvector fields of reversed parity) on $X$ respectively by
\begin{align*}
 \widehat{\Pol}(X,n)&:= \prod_{j \ge 0}  \hatTot  \Symm_{\cC^{\infty}(X)}^j(T_{\cC^{\infty}(X)}^{[-n-1]}), 
 \\
\widehat{\Pol}(X,\Pi n)&:= \prod_{j \ge 0}  \hatTot  \Symm_{\cC^{\infty}(X)}^j(\Pi T_{\cC^{\infty}(X)}^{[-n-1]}).  
\end{align*}
These have  filtrations  by super cochain complexes
\begin{align*}
\Fil^p\widehat{\Pol}(X,n)&:= \prod_{j \ge p} \hatTot  \Symm_{\cC^{\infty}(X)}^j(T_{\cC^{\infty}(X)}^{[-n-1]}),
 \\
\Fil^p\widehat{\Pol}(X,\Pi n)&:= \prod_{j \ge p} \hatTot  \Symm_{\cC^{\infty}(X)}^j(\Pi T_{\cC^{\infty}(X)}^{[-n-1]})
\end{align*}
respectively,
with $[\Fil^i,\Fil^j] \subset \Fil^{i+j-1}$ and $\Fil^i \Fil^j \subset \Fil^{i+j}$, where the commutative product and Schouten--Nijenhuis bracket are defined as before.
\end{definition}
Note that for our hypotheses on dg NQ-manifolds, the double complex $T_{\cC^{\infty}}(X)$ is bounded below in the cochain direction, so we could just write the total complexes $\hatTot$ as product total complexes $\Tot^{\sqcap}$, which would not work for the more general objects of Remark \ref{htpyfdrmkd2}.

We now define the spaces $\cP(X,n), \cP(X, \Pi n)$ of Poisson structures  by the formulae of Definition \ref{poissdefLie}.

\begin{definition}\label{binondegdef}
We say that a Poisson structure $\pi \in  \cP(X,n)$ 
is non-degenerate if   the map
\[
 \pi_2^{\sharp}\co \Tot^{\sqcap} (\Omega_X^1\ten_{\O_X}\O_X^0)^{[n]} \to \Tot^{\sqcap} (T_X\ten_{\O_X}\O_X^0)
\]
defined by contraction is a quasi-isomorphism. 
\end{definition}

The definitions of shifted symplectic structures from \S \ref{poisssnLie} now carry over:

\begin{definition}\label{PreSpdefb}
 Define the space $\PreSp(X,n)$ of $n$-shifted pre-symplectic structures on $X$ by regarding the de Rham complex of Definition \ref{DRdefderivedLie} as an abelian filtered DGLA, and
 writing 
\begin{align*}
\PreSp(X,n)&:= \Lim_i\mmc( \Fil^2\DR(X)^{[n+1],=}/\Fil^{i+2}) \\
\PreSp(X,\Pi n)&:= \Lim_i\mmc( \Fil^2\DR(X)^{[n+1],\ne}/\Fil^{i+2}). 
\end{align*}

Let $\Sp(X,n) \subset \PreSp(X,n)$ (resp.  $\Sp(X,\Pi n) \subset \PreSp(X,\Pi n)$)  consist of the symplectic structures in the sense of Definition \ref{symplecticdefdgLie} --- this is a union of path-components.
\end{definition}

\begin{theorem}\label{compatthmdgLie}
For a super dg NQ-manifold $X$, there are canonical weak equivalences
\[
 \Sp(X,n) \simeq \cP(X,n)^{\nondeg}\quad \quad \Sp(X,\Pi n) \simeq \cP(X,\Pi n)^{\nondeg}
\]
of simplicial sets.

In particular,  the sets of equivalence classes of (parity-reversed) $n$-shifted symplectic structures and of (parity-reversed) non-degenerate $n$-shifted Poisson structures on $X$ are isomorphic. 
 \end{theorem}
\begin{proof}
 This proof follows along the lines of Theorem \ref{compatthmLie},  with constructions adapted from \cite[\S \ref{poisson-Artinsn}]{poisson}. The space of Poisson structures still has a canonical tangent vector $\sigma$, and good properties of $\hatTot$ with respect to tensor products ensure that each Poisson structure $\pi$ gives  a compatibility map $\mu(-,\pi)$ from de Rham cohomology to Poisson cohomology. The non-degeneracy condition is formulated to ensure that this is a quasi-isomorphism, and the proof of Theorem \ref{compatthmLie} then adapts verbatim.
\end{proof}

\begin{example}
 As in Examples
\ref{strictcompex}, for shifted symplectic structures  with $\omega=\omega_2$ and $\omega_2^{\flat}$ an isomorphism, the theorem gives the corresponding non-degenerate shifted Poisson structure as $\pi=\pi_2$ with $\pi_2^{\sharp}= (\omega_2^{\flat})^{-1}$. 

Beware that such strict non-degeneracy conditions for shifted symplectic and Poisson structures are not preserved by the functoriality of  \S \ref{descentsn}, and in particular not preserved by levelwise $\delta$-quasi-isomorphisms.
\end{example}

\section{Deformation quantisation}\label{quantnsn}

We now consider quantisation for $n$-shifted symplectic and Poisson structures,  which for us means  an algebraic structure over the power series ring with formal variable $\hbar$ (of bidegree $0$ and even parity) which recovers the algebra of functions on setting $\hbar \mapsto 0$, and endows it with a Poisson structure.\footnote{Most of the results in this section have analogues with $\hbar$ of some other chain degree $r$, by inserting appropriate shifts into the relevant constructions.  In that case, analogues of $BD_k$ quantisation arise whose classical limits are $(k-1-r)$-shifted Poisson structures.} In this section, we will not consider parity-reversed Poisson structures, since they have no non-trivial notion of quantisation in that sense. 

For $k \ge 1$, the natural notion of quantisation for $P_k$-algebras   is given by $E_k$-algebras. An $E_k$-algebra can be thought of as a cochain complex with $k$ associative multiplications which commute with each other up to homotopy, and the commutators then give rise to a Lie bracket of cochain degree $1-k$. For instance, $k=2$ can  be modelled by brace algebras, since by \cite[Theorem 1.1]{McClureSmithDeligneConj}, the brace operad $\Br$ is naturally  $\Q$-linearly quasi-isomorphic to rational chains $\CC_{\bt}(E_k,\Q)$ on the $E_2$ operad.

The $k=1$ case was formulated  first, and corresponds to  associative deformations of a commutative algebra, recovering the Poisson bracket from the commutator. The motivation for the other cases come from Dunn additivity, because homotopically an $E_k$-algebra can be regarded as an $E_{k-1}$-algebra in $E_1$-algebras, and the additivity property for Poisson operads of \cite[Theorem 3.4.1]{CPTVV} (due to Rozenblyum and stated there without proof) or \cite{safronovBraces} allows us to regard a $P_k$-algebra as an $E_{k-1}$-algebra in $P_1$-algebras.

More specifically, deformation quantisations are formulated in terms of the $BD_k$ operads. The $BD_1$ operad  of \cite[Definition 13.2.2]{costelloNotesSupersymm} (due to Ed Segal) is given by the  Rees construction of the smallest filtration on the associative operad $\Ass$ for which the multiplication has weight $0$ and the commutator bracket weight $-1$, so $BD_1/\hbar \cong \Com$ and $BD_1[\hbar^{-1}] \cong \Ass(\!(\hbar)\!)$.  A $BD_1$-algebra structure on a complete flat $\R\brh$-module concentrated in degree $0$ is precisely an associative algebra structure which is commutative modulo $\hbar$. The $BD_k$ operads of \cite[\S 3.5.1]{CPTVV} are defined similarly, in terms of the good truncation filtration on $\CC_{\bt}(E_k,\Q)$.

However, for $k \ge 2$ Kontsevich formality \cite[Theorem 2]{KontsevichOperads} gives an equivalence (for any choice of Drinfeld associator) between $E_k$-algebras and $P_k$-algebras, and hence between $BD_k$-algebras and $P_k\brh$-algebras, so quantisations of $n$-shifted Poisson structures automatically exist for all $n\ge 1$ and are in one-to-one correspondence with deformations of the Poisson structure over $\R\brh$.

We will focus here  on the non-trivial (given an operadic formality equivalence) cases of $0$-shifted and $(-1)$-shifted symplectic structures, with  quantisations of $P_1$-algebras being given by $E_1$ (i.e. $A_{\infty}$)-algebras, and quantisations  of $P_0$-algebras being given by $BV_{\infty}$-algebras. These will satisfy constraints making them quasi-isomorphic respectively to $BD_1$-algebras and the $BD_0$-algebras of \cite[Definition 2.2.5]{gwilliamThesis};  $BD_k$-algebras correspond to $E_k$-algebras in $BD_0$-algebras \cite[Theorem 5.17]{KarlssonKellerMuellerPulmann}. The situation for $(-2)$-shifted structures is even more subtle, deforming just Maurer--Cartan elements of a BV algebra (i.e. QME solutions) rather than algebraic structures.

\subsection{Quantisation of \texorpdfstring{$0$}{0}-shifted Poisson structures}

\subsubsection{Polydifferential operators}

\begin{definition}\label{diffopsdef}
Given a super dg NQ-manifold $X$, 
we write $\cD_{\C^{\infty}(X)} \subset \HHom_{\R}(\C^{\infty}(X),\C^{\infty}(X))$ for the super chain cochain complex of $\C^{\infty}$ differential operators. This consists of homomorphisms which can be written locally as the $\O_X$-linear span of
\[
\pd_{i_1, \ldots, i_m}:= \pd_{x_{i_1}}\ldots \pd_{x_{i_m}}
\]
for homogeneous co-ordinates $x_i \in \O_X$. We denote by $F_k\cD_{\C^{\infty}(X)} \subset \cD_{\C^{\infty}(X)} $ the space of differential operators of order $\le k$, i.e. the span of $ \{\pd_{i_1, \ldots, i_m}~:~\sum i_r\le k\}$. 

Given  $\C^{\infty}$-modules $M,N$ in  super chain cochain complexes, we write $F_k\cDiff_{\C^{\infty}(X)}(M,N)$ for the space of $\C^{\infty}$ differential operators from $M$ to $N$ of order $\le k$, i.e.
\[
 F_k\cDiff_{\C^{\infty}(X)}(M,N):= \HHom_{\C^{\infty}(X)}(M, N\ten_{\C^{\infty}(X)}F_k\cD_{\C^{\infty}(X)}),
\]
where $\HHom$ is taken with respect to the right $\C^{\infty}(X)$-module structure on the target. We then write $\cDiff_{\C^{\infty}(X)}(M,N):=\LLim_k F_k\cDiff_{\C^{\infty}(X)}(M,N)$.
 \end{definition}

\begin{remark}\label{diffopsrmk}
For  co-ordinate-free descriptions of $\cD_{\C^{\infty}(X)}$ and $\cDiff_{\C^{\infty}(X)}(M,N)$, we may adapt the standard algebraic descriptions. 
 The super chain cochain algebra $\C^{\infty}(X \by X)$ of Definition \ref{productdef} has a left  $\C^{\infty}(X)$-module structure coming from the projection map $X^n \by X \to X$ to the first factor, and a morphism $\Delta^{\sharp} \co \C^{\infty}(X \by X) \to \C^{\infty}(X)$ coming from the diagonal embedding. If we write $I:=\ker (\Delta^{\sharp})$, then we just have
\begin{align*}
 F_k\cD_{\C^{\infty}(X)} &\cong  \HHom_{\C^{\infty}(X)}( \C^{\infty}(X \by X)/(\overbrace{I \cdots I}^{k+1}) ,\C^{\infty}(X)),\\
F_k\cDiff_{\C^{\infty}(X)}(M,N)&\cong  \HHom_{\C^{\infty}(X)}( M\ten_{\C^{\infty}(X)} \C^{\infty}(X \by X)/(\overbrace{I \cdots I}^{k+1}) ,N).
\end{align*}

Alternatively, we can describe $\C^{\infty}$-differential operators as algebraic differential operators with additional conditions. If for $a \in \C^{\infty}(X)$ and $\theta \in \HHom_{\R}(M,N)$, we write $\ad_a(\theta)$ for the commutator $ a \circ \theta \mp  \theta \circ a$, 
then algebraic differential operators from $M$ to $N$ of order $\le k$  are elements  $\theta \in \HHom_{\R}(M,N)$ satisfying 
\[
 \ad_{a_0}(\ad_{a_1}(\ldots (\ad_{a_k}(\theta))\ldots))=0
\]
for all $(k+1)$-tuples $(a_0, a_1, \ldots, a_k) \in \C^{\infty}(X)^{k+1}$. To be a $\C^{\infty}$-differential operator, $\theta$ must also satisfy a generalisation of the condition of  \cite[\S 5.3]{joyceAGCinfty} for $\C^{\infty}$ $1$-forms, namely that for all $v \in M$, $f \in \C^{\infty}(\R^m)$ and $a_1, \ldots, a_m \in \C^{\infty}(X)^{0,=}_0 $, we have a generalised Taylor series expansion 
\[
 \theta(f(a_1, \ldots,a_n)v)= \sum_{I=(i_1, \ldots,i_n)} \frac{1}{i_1!\ldots i_n!} \frac{\pd^I f}{\pd x^I}(a_1, \ldots,a_n) \ad_{a_1}^{i_1}(\ldots (\ad_{a_n}^{i_n}(\theta))\ldots)(v)
\]
(the sum is necessarily finite, since only terms with $|I|\le k$ contribute when $\theta$ has order $\le k$).
\end{remark}

We now generalise the complex of polydifferential operators from \cite[4.6.1]{kontsevichPoisson}.

\begin{definition}\label{HHdefa} 
Given a super  dg NQ-manifold $X$, define the  super chain cochain complex $\cD^{\poly}(X)$ of polydifferential operators in terms of the products $X^n:= \overbrace{X \by \ldots \by X}^n$ by
\[
\cD^{\poly}(X)_{\#}:= \prod_{n\ge 0} \cDiff_{\C^{\infty}(X^n)}(\C^{\infty}(X^n), \C^{\infty}(X))_{[n]}, 
\]
with cochain differential $Q$ and chain  differential $\delta \pm b$, for the  Hochschild differential $b$ determined by
\begin{align*}
 (b f)(a_1, \ldots , a_n) = &a_1 f(a_2, \ldots, a_n)\\
 &+ \sum_{i=1}^{n-1}(-1)^i f(a_1, \ldots, a_{i-1}, a_ia_{i+1}, a_{i+2}, \ldots, a_n)\\
&+ (-1)^n f(a_1, \ldots, a_{n-1})a_n.
\end{align*} 

We define an increasing filtration $\tau^{HH}$ on $\cD^{\poly}(X)$ 
 by good truncation in the Hochschild direction, so $\tau^{HH}_p \cD^{\poly}(X)   \subset \cD^{\poly}(X)$ is the subspace
\begin{align*}
&\prod_{n=0}^{p-1} \cDiff_{\C^{\infty}(X^n)}(\C^{\infty}(X^n), \C^{\infty}(X))_{[n]}  \\
&\by \ker(b \co \cDiff_{\C^{\infty}(X^p)}(\C^{\infty}(X^p), \C^{\infty}(X))\to \cDiff_{\C^{\infty}(X^{p+1})}(\C^{\infty}(X^{p+1}), \C^{\infty}(X))   )_{[p]}.
\end{align*}
\end{definition}

Writing $\Br$ for the brace operad, regarded as an operad in chain complexes,   $\cD^{\poly}(X)$ is then naturally a $\Br$-algebra (i.e. a homotopy G-algebra in the sense of \cite[Remark 8]{voronovHtpyGerstenhaber}), with operations defined as in \cite[\S 3.1]{voronovHtpyGerstenhaber}, but incorporating Koszul signs. In other words, it has
 a cup product in the form of
a  map
\[
\cD^{\poly}(X) \ten  \cD^{\poly}(X)\xra{\smile} \cD^{\poly}(X),
\]
 of super cochain chain complexes,  and braces in the form of  maps
\[
 \{-\}\{-,\ldots,-\}_r \co \cD^{\poly}(X) \ten \cD^{\poly}(X)^{\ten r}\to \cD^{\poly}(X)_{[r]}
\]
of super cochain complexes, satisfying the conditions of \cite[\S 3.2]{voronovHtpyGerstenhaber} with respect to the chain  differential $b$. The commutator of the brace $\{-\}\{-\}_1$ is a Lie bracket (the Gerstenhaber bracket), so  $\cD^{\poly}(X)_{[-1]}$ is naturally a Lie algebra in  super cochain chain complexes.

The associated graded piece $\gr^{\tau^{HH}}_p\cD^{\poly}(X)$ surjects quasi-isomorphically to cohomology $\H^p_b\cD^{\poly}(X)$ of $\cD^{\poly}(X)$ with respect to the Hochschild differential $b$. On $\H^*_b\cD^{\poly}(X)$, the cup product is graded-commutative and the higher braces vanish, so we just have a $P_2$-algebra. Moreover, the HKR map gives an isomorphism $\mathrm{HKR} \co \prod_p \H^p_b\cD^{\poly}(X)^{[-p]} \to \widehat{\Pol}(X,0)$ of $P_2$-algebras.

As in \cite[\S 2.1]{braunInvolutive}, there is also an involution
\[
 i(f)(a_1, \ldots, a_m) := - (-1)^{\sum_{i<j}  \bar{a}_i \bar{a}_j} (-1)^{m(m+1)/2}f(a_m, \ldots , a_1).
\]
of $\cD^{\poly}(X)$, with $-i$ reversing the cup product and changing the sign of the Gerstenhaber bracket, so $i$ itself is a morphism of DGLAs.

The following definitions are adapted from \cite{DQnonneg}, replacing Hochschild complexes with complexes of polydifferential operators:

\begin{definition}\label{qpoldef0}
Define the complex of quantised $0$-shifted polyvector fields  on $X$ by
\[
 Q\widehat{\Pol}(X,0):= \prod_{p \ge 0} \hatTot \tau^{HH}_p \cD^{\poly}(X)\hbar^{p-1}. 
\]

Define its subcomplex $Q\widehat{\Pol}(X,0)^{sd}$ to consist of elements $\Delta(\hbar)$ with  $i(\Delta)(\hbar)= (\Delta)(-\hbar)$.
\end{definition}

Properties of the filtration $\tau^{HH}$ as in \cite[Lemmas \ref{DQnonneg-gradedHH} and \ref{DQnonneg-involutiveHH}]{DQnonneg}  ensure that $Q\widehat{\Pol}(X,0)^{[1]}$ and $(Q\widehat{\Pol}(X,0)^{sd})^{[1]}$ are super DGLAs.

\begin{definition}\label{QFdef}
Define a decreasing filtration $\tilde{\tau}_{HH}$ on  $Q\widehat{\Pol}(X,0)$  by the subcomplexes 
\begin{align*}
 \tilde{\tau}_{HH}^iQ\widehat{\Pol}(X,0)&:= \prod_{j \ge i} \hatTot\tau^{HH}_j\cD^{\poly}(X)\hbar^{j-1}.
\end{align*}
 \end{definition}
This filtration is complete and Hausdorff,   with $[\tilde{\tau}_{HH}^i,\tilde{\tau}_{HH}^j] \subset \tilde{\tau}_{HH}^{i+j-1}$.
In particular,  this makes $\tilde{\tau}_{HH}^2Q\widehat{\Pol}(X,0)^{[1]}$ and  $(Q\widehat{\Pol}(X,0)^{sd})^{[1]}$ into pro-nilpotent filtered super DGLAs.

\begin{definition}\label{E1QPdef}
Define a $BD_1$ quantisation of $X$ to be a Maurer--Cartan element
\[
 \Delta \in \mc((\tilde{\tau}_{HH}^2Q\widehat{\Pol}(X,0)^{=})^{[1]}),
\]
and define the space of $BD_1$ quantisations of $X$ by
\[
 Q\cP(X,0):= \mmc((\tilde{\tau}_{HH}^2Q\widehat{\Pol}(X,0)^{=})^{[1]}).
\]

The morphism $Q\cP(X,0) \to \cP(X,0)$, sending a quantisation  to its underlying $0$-shifted Poisson structure, is given by the composite 
\[
  \prod_{j \ge 2} \tau^{HH}_j\cD^{\poly}(X)\hbar^{j-1}  \to 
  \prod_{j \ge 2} \gr^{\tau^{HH}}_j\cD^{\poly}(X)\hbar^{j-1}\xra[\sim]{\hbar^{1-j}\mathrm{HKR}} F^2\widehat{\Pol}(X,0).
\]

Define a self-dual $BD_1$ quantisation of $X$ to be a Maurer--Cartan element
\[
 \Delta \in \mc((\tilde{\tau}_{HH}^2Q\widehat{\Pol}(X,0)^{sd,=})^{[1]}),
\]
and define the space of $BD_1$ quantisations of $X$ by
\[
 Q\cP(X,0):= \mmc((\tilde{\tau}_{HH}^2Q\widehat{\Pol}(X,0)^{=})^{[1]}).
\]
\end{definition}
Thus each  element $\Delta \in  Q\cP(X,0)_0$ can be written as $\Delta = \sum_{n \ge 0} \Delta_n$, for 
\[
\Delta_n  \in \hbar^{\max{n-1,1}} \hatTot\cDiff_{\C^{\infty}(X^n)}(\C^{\infty}(X^n), \C^{\infty}(X))^{2-n}\brh
\]
satisfying the Maurer--Cartan equation, which forces restriction of the leading coefficients to $\tau^{HH}$.

\begin{remarks}\label{Q1quantrmks}
When $X$ is just an NQ supermanifold, $\hatTot$ acts trivially, so $(\Delta_0,Q +\Delta_1, m+\Delta_2, \Delta_{\ge 3})$ gives a curved $A_{\infty}$-algebra structure $\O_X'$ on $\O_X\llbracket\hbar\rrbracket$ with $\O_X'/\hbar=\O_X$,  because $\hbar\mid \Delta$. As in \cite[\S 1.1.2 and Remark \ref{DQLag-HKRrmk}]{DQLag}, this is in fact an algebra over a resolution of the $BD_1$ operad. 
For dg supermanifolds the same is true, and the curvature $\Delta_0$ is necessarily $0$ for degree reasons, but it still exhibits itself in the higher  structure of $Q\cP(X,0)$, with inner automorphisms (conjugation by $\exp(\hbar\sO_X\brh)$) giving rise to homotopies between isomorphisms. Self-dual quantisations correspond to such deformations equipped with an $(\hbar \mapsto -\hbar)$-semilinear involution $\O_X' \to (\sO_X')^{\op}$; when $X$ is a supermanifold, this just says $a\star_{-\hbar} b = b\star_{\hbar} a$.

For more general super dg NQ-manifolds, the stacky and derived structures interact in a non-trivial way for quantisations, and indeed for Poisson structures; a quantisation gives rise to a curved $A_{\infty}$-algebra structure on $(\hatTot \O_X) \brh$, but each coefficient of each component $\Delta_n$ of the $A_{\infty}$ structure must be bounded below in the cochain direction. See \S \ref{filtsn} for a little more detail on that perspective.
In general, as in \cite[Remark \ref{DQpoisson-strictquantnrmk}]{DQpoisson},  if  $\H^2\hatTot\C^{\infty}(X)=0$ then all quantisations of $X$ are equivalent to uncurved quantisations.

The boundedness constraints on the operators $\Delta_n$ make such quantisations significantly stronger structures than just a deformation of the ring of functions up to quasi-isomorphism. In particular, \cite[Proposition \ref{DQnonneg-Perprop}]{DQnonneg} ensures that each element of  $Q\cP(X,0)$ gives rise to a curved $A_{\infty}$ deformation of the dg category of perfect complexes on $X$, defined as in Example \ref{vbundleex}.
\end{remarks}

Examples of $BD_1$  quantisations  on NQ-manifolds and supermanifolds  are given by \cite[Theorem 3.2]{CattaneoFelder} and \cite{AlekseevMeinrenken}, respectively. The latter can also be regarded as a structure on an NQ-manifold satisfying the  analogue of Definition \ref{E1QPdef} with $\hbar$ in degree $2$, so the $\hbar^{j-1}$ terms in the comparison map make the underlying Poisson structure $2$-shifted rather than $0$-shifted.

\subsubsection{Quantisation of \texorpdfstring{$0$}{0}-shifted Poisson structures on dg NQ supermanifolds}\label{quant0sn}

We now explain how \cite{DQpoisson}  gives quantisations of $0$-shifted Poisson structures on dg NQ-manifolds. As in the Kontsevich--Tamarkin approach \cite{tamarkinAnotherKontsevichFormality,KontsevichOperads,YekutieliDQAG,yekutieliTwistedDQ,vdBerghGlobalDQ} to quantisation, we begin by making use of the $E_2$-algebra structure on polydifferential operators and formality of the $E_2$-operad. Where their  quantisation for manifolds hinges on invariance of the Hochschild complex under affine transformations, an argument which will not adapt to dg manifolds, we instead exploit
 the observation that the Hochschild complex  carries an anti-involution, and that such anti-involutive deformations of the complex of polyvectors are essentially unique.

\begin{theorem}\label{derived0quantthm}
 Given a dg NQ supermanifold  $X$, 
a choice of even associator gives an equivalence between
the space $Q\cP(X,0)$ of $BD_1$ quantisations of $X$
and the Maurer--Cartan space 
\begin{align*}
  \mmc( F^2\widehat{\Pol}(X,0)^=_{[-1]} \by\hbar F^1\widehat{\Pol}(X,0)^=_{[-1]} \by \hbar^2\widehat{\Pol}(X,0)_{[-1]}^=\brh),
 \end{align*}
with the space $Q\cP(X,0)^{sd} \subset Q\cP(X,0)$ of self-dual quantisations corresponding to the subspace
\[
 \mmc(F^2\widehat{\Pol}(X,0)^=_{[-1]} \by \hbar^2\widehat{\Pol}(X,0)^=_{[-1]}\brhh)
\]

In particular, every Poisson structure
\[
 \pi \in \cP(X,0)= \mmc( F^2\widehat{\Pol}(X,0)^=_{[-1]}) 
\]
admits at least one quantisation
in the form of a curved $A_{\infty}$-deformation of $\O_X$  carrying an $\hbar$-semilinear anti-involution. 
\end{theorem}
\begin{proof}
This is the $\C^{\infty}$ case of \cite[Theorem \ref{DQpoisson-fildefhochthm2}]{DQpoisson} (which also establishes the result in algebraic and analytic settings), slightly generalised to incorporate the additional superalgebra $\Z/2$-grading. We now describe the main steps of the proof.

For $\cD^{\poly}_{\bigoplus} \subset \cD^{\poly}$ defined using $\bigoplus$ instead of $\prod$,  the algebra $(\cD^{\poly}_{\bigoplus}, \tau^{HH},-i)$ is a quasi-involutive  almost commutative  brace algebra in super chain cochain complexes in the sense of \cite[Definition \ref{DQpoisson-acbracedef}]{DQpoisson}.
 As in \cite[Definition \ref{DQpoisson-pwinvdef}]{DQpoisson}, for  Levi decompositions $w$ of the Grothendieck--Teichm\"uller group corresponding to even associators, the Grothendieck--Teichm\"uller action on the brace operad $\Br$ gives an $\infty$-equivalence $p_w$  between almost commutative anti-involutive brace algebras and almost commutative anti-involutive $P_2$-algebras. The latter are filtered $P_2$-algebras equipped with an involution preserving the product and reversing the bracket, such that the involution acts as $(-1)^i$ on the $i$th graded piece, the product and bracket preserve the filtration, and the bracket vanishes on the associated graded algebra.
 
 The associated graded object $\gr^{\tau^{HH}}\cD^{\poly}$ is a graded super chain cochain complex, which via the HKR isomorphism is levelwise quasi-isomorphic to the graded  $\gr_{\tau}\Br \simeq P_2$-algebra 
\[
 \cPol(X,0) :=\bigoplus_m \Symm_{\cC^{\infty}(X)}^m((T_{\cC^{\infty}(X)})_{[1]}).
\]
As a graded commutative algebra, this is generated over its weight $0$ component $\cC^{\infty}(X)$ by its weight $1$ component $T_{\cC^{\infty}}(X)$. This severely constrains the possible almost commutative $P_2$-algebras with this associated graded object, and imposing an anti-involution completely rigidifies the problem. In particular,  on applying $\hatTot$ termwise this gives us an involutive filtered $P_{2, \infty}$-quasi-isomorphism between  $\cPol(X,0)$ and  $p_w\cD^{\poly}_{\oplus}$ (and hence   
 a filtered  $L_{\infty}$-quasi-isomorphism between $\cPol(X,0)$ and $\cD^{\poly}_{\oplus}$).

The desired expressions then follow by substitution, in particular clearing out factors of $\hbar^{m-1}$ multiplying $m$-vectors as in \cite[Corollary \ref{DQpoisson-affquantcor}]{DQpoisson}.
\end{proof}

\begin{remarks}\label{quantpropsd0}
Theorem \ref{derived0quantthm} gives a complete parametrisation of all deformation quantisations, so necessarily includes those of \cite{tamarkinAnotherKontsevichFormality,kontsevichPoisson} for manifolds. 
For manifolds, the formality quasi-isomorphism $(\cD^{\poly}(X),\tau^{HH}) \simeq (\widehat{\Pol}(X,0),\Fil)$ must be the same as that of \cite{tamarkinAnotherKontsevichFormality}, because the uniqueness property satisfied by our deformation guarantees local invariance under affine transformations, which is Tamarkin's uniqueness property. 

The objects of the parametrisations of $Q\cP(X,0)$ and $Q\cP(X,0)^{sd}$ in Theorem \ref{derived0quantthm} can be thought of as formal Poisson structures on $\sO_X\brh$ and $\sO_X\brhh$ respectively, except that the underlying differential algebra is allowed to deform (the term in $\hbar T_{\cC^{\infty}(X)}\brh$, resp. $\hbar^2T_{\cC^{\infty}(X)}\brhh$) and there can be curvature (the term in $\hbar^2\cC^{\infty}(X)\brh$, resp. $\hbar^2\cC^{\infty}(X)\brhh$).

For  non-degenerate $0$-shifted Poisson structures, we can parametrise the space of quantisations even more explicitly. The Legendre transformation from the proof of  Theorem \ref{compatthmdgLie}, using $\sigma + \pd_{\hbar}$ in place of $\sigma$, combines with Theorem \ref{derived0quantthm} to give a weak equivalence 
 \begin{align*}
   Q\cP(X,0)^{\nondeg,sd} &\simeq \cP(X,0)^{\nondeg} \by \mmc(\hbar^2 \DR(X)^{[1],=}\llbracket\hbar^2\rrbracket);
 \end{align*}
this
recovers the description of \cite[Theorem \ref{DQnonneg-quantpropsd}]{DQnonneg}, which has a more direct proof but still depends \emph{a priori} on an even associator. Independence from choices of associator in the non-degenerate setting is established  whenever $\pi^0X^=$ is locally smoothly contractible in \cite[Example \ref{DQexact-DWLex}.(\ref{DQexact-DWLex0shift})]{DQexact}, together with a description of the transformation between the respective parametrisations, and also  the relation to Fedosov's and De Wilde and Lecomte's parametrisations when $X$ is a manifold.
\end{remarks}

\subsection{Quantisation of \texorpdfstring{$(-1)$}{-1}-shifted symplectic structures on dg NQ-manifolds}\label{quantneg1sn}

 Fix a super dg NQ-manifold $X$.
 
\subsubsection{Formulation of \texorpdfstring{$(-1)$}{-1}-shifted quantisations}
The following definitions are adapted from  
\cite[Definitions \ref{DQvanish-bistrictlb}, \ref{DQvanish-qpoldef}, \ref{DQvanish-QFdef} and \ref{DQvanish-Qpoissdef}]{DQvanish}:

\begin{definition}\label{bistrictlb}
 Define a strict line bundle over $X$ to be a $\C^{\infty}(X)$-module $M$ in super chain cochain complexes such that $M^{\#}_{\#}$ is a projective module of rank $1$ over the super bigraded-commutative algebra $\C^{\infty}(X)^{\#}_{\#}$ underlying $\C^{\infty}(X)$. 
\end{definition}
What we are calling a line bundle should really be thought of as the module of global sections of a line bundle. For each such $M$, there is an associated sheaf $M\ten_{ \C^{\infty}(X)}\O_X$ of sections.

\begin{definition}\label{qpoldefneg1}
Given a strict line bundle $M$ over $X$, define the super cochain complex of quantised $(-1)$-shifted polyvector fields  on $M$ by
\[
 Q\widehat{\Pol}(X,M,-1):= \prod_{p \ge 0}\hatTot F_p\cDiff_{\C^{\infty}(X)}(M,M)\hbar^{p-1}, 
\]
for differential operators and the order filtration from Definition \ref{diffopsdef}.

Multiplication of differential operators gives us a product
\[
 Q\widehat{\Pol}(X,M,-1) \by Q\widehat{\Pol}(X,M,-1) \to \hbar^{-1} Q\widehat{\Pol}(X,M,-1),
\]
but because $M$ is a line bundle, the associated commutator $[-,-]$ takes values in $Q\widehat{\Pol}(X,M,-1)$, so $Q\widehat{\Pol}(X,M,-1)$ is a super DGLA.

Define a decreasing filtration $\tilde{F}$ on  $Q\widehat{\Pol}(X,M,-1)$ by 
\[
 \tilde{F}^iQ\widehat{\Pol}(X,M,-1):= \prod_{j \ge i} F_j\cDiff_{\C^{\infty}(X)}(M,M)\hbar^{j-1};
\]
this has the properties that $Q\widehat{\Pol}(X,M,-1)= \Lim_i Q\widehat{\Pol}(X,M,-1)/\tilde{F}^i$, with $[\tilde{F}^i,\tilde{F}^j] \subset \tilde{F}^{i+j-1}$, $\delta \tilde{F}^i \subset \tilde{F}^i$, and $ \tilde{F}^i \tilde{F}^j \subset \hbar^{-1} \tilde{F}^{i+j}$.
\end{definition}

\begin{definition}\label{Qpoissdef0}
Define  the space $Q\cP(X,M,-1)$ of $BD_0$ quantisations of a strict line bundle $M$  on $X$  to be given by the simplicial
set
\[
 Q\cP(X,M,-1):= \Lim_i \mmc( \tilde{F}^2 Q\widehat{\Pol}(X,M,-1)^{=}/\tilde{F}^{i+2}).
\]

The morphism $Q\cP(X,M,-1) \to \cP(X,-1)$, sending a quantisation  to its underlying $(-1)$-shifted Poisson structure, is given by the composite
 \[
  \prod_{j \ge 2} F_j\cDiff_{\C^{\infty}(X)}(M,M)\hbar^{j-1} \to \prod_{j \ge 2} \gr^F_j\cDiff_{\C^{\infty}(X)}(M,M)\hbar^{j-1} \cong F^2\widehat{\Pol}(X,-1),
 \]
 since $ \gr^F_j\cDiff_{\C^{\infty}(X)}(M,M)\cong \Symm_{\cC^{\infty}(X)}^j(T_{\cC^{\infty}(X)})$.
\end{definition}

Thus a $BD_0$ quantisation is a deformation of $M$ given by differential operators, with some constraints on their orders. As in \cite[Remark \ref{DQvanish-BVrmk}]{DQvanish},  $BD_0$ quantisations $\Delta$ of the trivial line bundle $M=\C^{\infty}(X) $ with  $\Delta(1)=0$ are commutative $BV_{\infty}$-algebras in the sense of \cite[Definition 9]{kravchenko}, though we have additional restrictions because being a $\C^{\infty}$ differential operator is more restrictive than being an algebraic differential operator. 
Similarly to  Remarks \ref{Q1quantrmks},   in general these operations are defined on  $\hatTot \C^{\infty}(X)$ with bounds on the cochain degrees as \S \ref{filtsn}.
The case where the only non-zero term lies in  $F_2\cD_X\hbar$ then corresponds to a $BD_0$-algebra structure on  
$\C^{\infty}(X)\brh $ (i.e. a $BD$-algebra in the sense of \cite[Definition 2.2.5]{gwilliamThesis}).

\subsubsection{Quantisation for spin structures}

The module $\C^{\infty}(X)$ naturally has the structure of a left $\cD_{\C^{\infty}(X)}$-module (via the embedding of $\cD_{\C^{\infty}(X)}$ in $\HHom_{\R}(\C^{\infty}(X), \C^{\infty}(X))$); the same is true for any vector bundle equipped with a flat connection. Right $\cD$-modules are more subtle to construct, but on a super NQ-manifold $X$, the orientation bundle  (i.e. the determinant, or Berezinian, of $\Omega^1_{\C^{\infty}(X)}$) is naturally a right $\cD_{\C^{\infty}(X)}$-module, via the identification
\[
 \det\Omega^1_{\C^{\infty}(X)} \simeq (\Ext^p_{\cD_{\C^{\infty}(X)}^{\#}}(\C^{\infty}(X)^{\#}, \cD_{\C^{\infty}(X)}^{\#}),Q)
\]
(i.e. turn off the differential $Q$, calculate $\Ext$, then restore $Q$),
where $p$ is the number of local generators of $\O_X$ of even parity; this identification follows along the same lines as the construction of the Berezinian in \cite{DeligneMorganSupersymm}. 

Similarly, $(\Ext^p_{\cD_{\C^{\infty}(X),\#}^{\#}}(\C^{\infty}(X)^{\#}_{\#}, \cD_{\C^{\infty}(X),\#}^{\#}),Q, \delta)$ gives a right $\cD$-module when $X$ is  a super  dg NQ-manifold, but the expression is not usually invariant under the equivalences of Remarks \ref{derivedstackyrmk} and Appendix \ref{equivapp}. It does, however, behave when the only even parity generators are in chain degree $0$, in which case it broadly corresponds to the dualising line bundle of \cite[\S 5.6]{gaitsgoryIndCoh}. 

Now, if $\omega$ is a strict line bundle with a right $\cD$-module structure, there is a standard isomorphism 
\[
  \cD_{\C^{\infty}(X)}^{\op}\cong \cDiff_{\C^{\infty}(X)}(\omega, \omega)
\]
of super chain cochain associative algebras, following from the observation that the elements of   $\cD_{\C^{\infty}(X)}^{\op}$ acting on $\omega$ on the right must act as differential operators. Moreover, a right $\cD$-module structure on any vector bundle $M$ then corresponds to a flat connection on $M\ten_{\C^{\infty}(X)}\omega^*$.

We now proceed as in \cite[\S \ref{DQvanish-sdsn}]{DQvanish}:

If $L$ is a strict line bundle with a right $\cD$-module structure on $L^{\ten 2}$ (so $L$ broadly corresponds to a spin structure), we then have
\begin{align*}
 \cDiff_{\C^{\infty}(X)}(L,L)^{\op} &\cong (L\ten \cD_{\C^{\infty}} \ten L^*)^{\op}\\
&\cong L^*\ten \cD_{\C^{\infty}}^{\op} \ten L\\
&\cong L^*\ten L^{\ten 2} \ten  \cD_{\C^{\infty}} \ten (L*)^{\ten 2} \ten L\\
&\cong \cDiff_{\C^{\infty}(X)}(L,L).
\end{align*}

\begin{definition}
For a line bundle $L$  with a  right $\cD$-module structure on $L^{\ten 2}$, writing
\[
 (-)^t \co \cDiff_{\C^{\infty}(X)}(L,L)^{\op} \to \cDiff_{\C^{\infty}(X)}(L,L)
\]
 for the natural anti-involution above, define
\[
 (-)^* \co Q\widehat{\Pol}(X,L,-1) \to Q\widehat{\Pol}(X,L,-1)
\]
by
\[
 \Delta^*(\hbar):= -\Delta^t(-\hbar).
\]

We then define $Q\widehat{\Pol}(X,L,-1)^{sd}$ to be the fixed points for the involution $*$, and set 
\[
Q\cP(X,L,-1)^{sd}:= \Lim_i \mmc( \tilde{F}^2 Q\widehat{\Pol}(X,L,-1)^{sd,=}/\tilde{F}^{i+2})
\]
\end{definition}

The reason for the choice of sign $- \hbar$ in the definition of $\Delta^*$ is that on the associated graded $\gr^F_p \cDiff_{\C^{\infty}(X)}(L,L) \cong \Symm^p_{\C^{\infty}(X)} T_{\C^{\infty}(X)}$, the operation $(-)^t$ is given  by $(-1)^p$. Thus $\Delta$ and $\Delta^*$ have  the same underlying Poisson structures. 

\begin{theorem}\label{quantpropsdneg1}
 For a  super dg NQ-manifold $X$ and a  line bundle $L$ on $X$  with $L^{\ten 2} $ a right $\cD$-module (such as any square root of the orientation bundle), there is a canonical weak equivalence
\[
  Q\cP(X,L,-1)^{\nondeg,sd} \simeq \cP(X,-1)^{\nondeg} \by \mmc(\hbar^2 \DR(X)\llbracket\hbar^2\rrbracket).
\]
In particular, every non-degenerate $(-1)$-shifted Poisson structure gives  a canonical choice of self-dual quantisation of $(X,L)$.
\end{theorem}
\begin{proof}
The main results of \cite{DQvanish} combine and adapt to give this statement. The key is to modify the argument from Theorem \ref{compatthmLie} and Remarks \ref{quantpropsd0} via a compatibility map defined on a variant of the de Rham complex. As in \cite[Definition \ref{DQvanish-DRprimedef}]{DQvanish}, the first step is to let $\C^{\infty}(X^n)^{\wedge}$ be the completion of $\C^{\infty}(X^n)$ with respect to the ideal of the diagonal map $\C^{\infty}(X^n) \to \C^{\infty}(X)$. 

We then let $\C^{\infty}(X^{\bt +1})^{\wedge} $ be the total super cochain complex of the super double cochain complex
\[
 \C^{\infty}(X)^{\wedge} \xra{d} \C^{\infty}(X^2)^{\wedge} \xra{d}\C^{\infty}(X^3)^{\wedge} \xra{d} \ldots,
\]
with  boundary map  $d \co \C^{\infty}(X^{m +1})^{\wedge} \to\C^{\infty}(X^{m +2})^{\wedge}  $  given by
\[
 df(x^0, \ldots, x^{m+1}) = \sum_{i=0}^{m+1}(-1)^i f(x^0, \ldots,x^{i-1}, x^{i+1}, \ldots,  x^{m+1}).
\]
Then  $\C^{\infty}(X^{\bt +1})^{\wedge} $ has an associative product given by the  Alexander--Whitney cup product 
\[
 (f \smile g)(x^0, \ldots, x^{m+n+1})= f(x^0, \ldots, x^m)g(x^m, \ldots, x^{m+n}).
\]

The next step is to set $\DR'(X) \subset \C^{\infty}(X^{\bt +1})^{\wedge} $ be the super cochain subcomplex given by cosimplicial conormalisation, so we only consider functions in  $\C^{\infty}(X^{m +1})^{\wedge}$ which vanish on all big diagonals $X^m \subset X^{m+1}$.  As in \cite[Lemma \ref{DQvanish-cfDRlemma}]{DQvanish}, there is a natural quasi-isomorphism $\DR'(X) \to \DR(X)$. 

For any $\Delta \in Q\cP(X,L,-1)$, we write $T_{\Delta} Q\widehat{\Pol}(X,L,-1) $ for the complex given by first taking $\hbar Q\widehat{\Pol}(X,L,-1)$ then adding $[\Delta,-]$ to the differential. We regard its cohomology as  quantised $(-1)$-shifted Poisson cohomology, and 
it contains a canonical $1$-cocycle $ \hbar^{2}\frac{\pd \Delta}{\pd \hbar}$.
As in \cite[Lemmas \ref{DQvanish-mulemma1} and \ref{DQvanish-QPolmudef}]{DQvanish}, any $\Delta \in Q\cP(X,L,-1)$ gives rise to  a compatibility map
\[
 \mu(-,\Delta) \co \DR'(X)\brh \to T_{\Delta}Q\widehat{\Pol}(X,L,-1),
\]
induced by the  continuous multiplicative map on $\C^{\infty}(X^{\bt +1})^{\wedge} $  determined by the properties that
$\mu(1\ten 1, \Delta) = \Delta$ and $\mu(a, \Delta)=a$ for $ a \in \C^{\infty}(X)$. 

Filtering by powers of $\hbar$, on associated graded pieces the map $\mu(-,\Delta)$ reduces to the composite
\[
  \DR'(X)\hbar^i \to \DR(X)\hbar^i \xra{\mu(-,\pi)}\widehat{\Pol}_{\pi}(X,-1),
\]
for $\pi$ the underlying Poisson structure. When $\Delta$ is non-degenerate, this is a filtered quasi-isomorphism as in Theorem \ref{compatthmdgLie}, so $\mu(-,\Delta)$ is a filtered quasi-isomorphism. To $\Delta$ we may thus associate the power series
\[
 [\mu(\Delta,-)^{-1}\hbar^{2}\frac{\pd \Delta}{\pd \hbar}] \in  \H^1\Fil^2\DR(X) \by  \hbar\H^1\DR(X)\brh.
\]

If we start from a power series in $\H^1\DR(X)$ and attempt to solve for $\Delta$, then
the leading term is given by the correspondence between  non-degenerate Poisson structures and symplectic structures in Theorem \ref{compatthmdgLie}. For higher terms, we filter powers of $\hbar$, and use the obstruction theory associated to filtered DGLAs (see Appendix \ref{towersn}). Calculation as in \cite[Proposition \ref{DQvanish-quantprop}]{DQvanish} shows that the only potential obstruction or ambiguity is in the first-order deformation of the Poisson structure, but as in
\cite[Lemma \ref{DQvanish-quantpropsd}]{DQvanish}, this vanishes when we restrict to self-dual quantisations, showing that the latter are parametrised by
\[
 \H^{1}\Fil^2\DR(X)^{=} \by \hbar^2 \H^1\DR(X)^{=}\brhh. \qedhere
\]
\end{proof}

\begin{example}\label{vanishex}
 For $M$ a manifold of dimension $p$, Examples \ref{DCritex} give a  canonical $(-1)$-shifted symplectic structure $\omega$ on the derived critical locus $X=\DCrit(M,f)$, and pulling back the determinant bundle $\Omega^p_M$ to $X$ gives a line bundle $L$ satisfying the conditions of Theorem \ref{quantpropsdneg1}. A natural self-dual quantisation of 
\begin{align*}
  L &= \Omega^p_{\C^{\infty}(M)} \ten_{\C^{\infty}(M)}\C^{\infty}(\DCrit(M,f))\\
&= \Omega^p_{\C^{\infty}(M)} \ten_{\C^{\infty}(M)} (\bigoplus_i  (\L^i_{ \C^{\infty}(M)}T_{ \C^{\infty}(M)})_{[-i]}, \lrcorner df)\\
 &\cong (\bigoplus_j \Omega^j_{\C^{\infty}(M), [j]}, df \wedge-)_{[-p]}
\end{align*}
 over this symplectic structure is then given by the twisted de Rham complex 
\[
 (\bigoplus_j  \Omega^j_{\C^{\infty}(M), [j]}\brh , \hbar d + df \wedge-)_{[-p]},
\]
and as in \cite[Lemma \ref{DQvanish-PVlemma}]{DQvanish}, this is the quantisation associated by Theorem \ref{quantpropsdneg1} to the constant power series $\omega$. As in \cite[Proposition 4.9]{DQvanish}, on inverting $\hbar$ the algebraic analogue of this quantisation recovers the sheaf of vanishing cycles from \cite{BBDJS}, which is often described as a quantisation but lacks a classical limit.

A volume form $\mu$ on $M$ is the same as a choice of isomorphism $\Omega^p_M \cong \O_M$. This leads to an isomorphism $L \cong \O_X$, and the quantisation above then becomes a quantisation of the trivial line bundle on $X$. This quantisation is precisely the quantum BV complex as described on  \href{https://ncatlab.org/nlab/show/BV-BRST+formalism}{the nlab} or \cite[\S 2.2]{gwilliamThesis}.

\end{example}

\subsection{Quantisation of shifted Lagrangians}\label{Lagsn}


In the smooth setting, there  is a natural analogue of the shifted Lagrangians of \cite{PTVV}:
\begin{definition}\label{lagdef}
Given an $n$-shifted  pre-symplectic structure $\omega$
\[
 \omega \in \z^{n+2}\Fil^2\DR(X)^{=},
\] 
on a super dg NQ-manifold   $X$, and a morphism $\phi \co Z \to X$ of super dg NQ-manifolds,   an 
 isotropic structure on $Z$ relative to $\omega$ is  an element $(\omega, \lambda)$ of 
\[
 \z^{n+2}\cocone(\Fil^2\DR(X)\to \Fil^2\DR(Z))^{=}
\]
lifting $\omega$. This structure is called Lagrangian if $\omega$ is symplectic and if contraction of $ T_{\C^{\infty}(Z)}$ with  the image of $(\omega_2, \lambda_2)$ in 
\[
 \z^{n-1}\cone(\Omega^1_{\C^{\infty}(X)}\ten_{\C^{\infty}(X)}\Omega^1_{\C^{\infty}(Z)}
 \to \Omega^1_{\C^{\infty}(Z)}\ten_{\C^{\infty}(Z)}\Omega^1_{\C^{\infty}(Z)}
 )
\]
 induces a quasi-isomorphism
\begin{align*}
 (\omega_2, \lambda_2)^{\flat} \co  \hatTot ( T_Z\ten_{\sO_Z}\sO_Z^0)
 \to  \hatTot \cone\left( (\phi^{-1}\Omega^1_X)\ten_{\phi^{-1}\sO_X}\sO_Z^0 \to   \Omega^1_Z\ten_{\sO_Z}\sO_Z^0\right)^{[n-1]}. 
\end{align*}

We define $n$-shifted structures of reversed parity similarly, replacing $=$ with $\ne$ and applying the parity reversion operator $\Pi$ to the target of $(\omega_2,\lambda_2)^{\flat}$.
\end{definition}
In particular, this means that Lagrangian submanifolds of symplectic manifolds are Lagrangians with respect to $0$-shifted symplectic structures, but note that the morphism $Z \to X$ in Definition \ref{lagdef} need not be in any sense injective. As an extreme example, an $n$-shifted Lagrangian structure on the unique morphism $Z \to \R^0$ corresponds to an $(n-1)$-shifted symplectic structure on $Z$.
For many examples of $n$-shifted Lagrangians on NQ-manifolds, see \cite{PymSafronov}; the prototypical example is given by the embedding of a supermanifold $M$ in its $n$-shifted cotangent bundle $T^*[n]M$. 


\subsubsection{Deformation quantisations}

Co-isotropic structures are harder to describe than Poisson structures, because they rely on the notion of a $P_{k+1}$-algebra acting on a $P_k$-algebra. In the algebraic setting they are formulated in \cite{MelaniSafronovI}, and those results translate to the smooth setting with minor adjustments. An $n$-shifted co-isotropic structure on a morphism $f \co  Z \to X$ consists of an $n$-shifted Poisson structure $\varpi$ on $X$, an $(n-1)$-shifted Poisson structure $\pi$ on $Z$ and a $\hat{P}_{n+1}$-algebra morphism $(\hatTot\C^{\infty}(X),\varpi) \to \widehat{\Pol}_{\pi}(Z,n-1)$ extending $f$, subject (in the dg NQ case) to boundedness constraints similar to those in Definition \ref{bipoldef}. By \cite[Theorem 4.23]{MelaniSafronovII},
$n$-shifted Lagrangians correspond precisely to  non-degenerate $n$-shifted co-isotropic structures; this gives many examples of degenerate $(n-1)$-shifted Poisson structures as $n$-shifted Lagrangians.

Similarly, a deformation quantisation of an $n$-shifted co-isotropic structure  on $Z \to X$ can be understood for $n>0$ as a $BD_{n+1}$ quantisation $\Phi$  of $X$, a $BD_{n}$ quantisation $\Delta$  of $Z$ and a strong homotopy $BD_{n+1}$ morphism $((\hatTot\C^{\infty}(X))\brh,\Phi) \to \hbar Q\widehat{\Pol}_{\Delta}(Z,n-1)$, where the target is defined similarly to Definition \ref{qpoldef0}, but  with differential twisted by $[\Delta,-]$, and using the $BD_n$-Hochschild complex for $n>1$. 
In particular, a quantisation of a $1$-shifted  co-isotropic structure  on $Z \to X$ can be formulated as a suitable brace algebra deformation of $\O_X$ equipped with a brace algebra map to the brace algebra of polydifferential operators on an associative algebra deformation of $\O_Z$. 

For $n>1$, \cite[Theorem 5.15]{MelaniSafronovII} shows that quantisation of co-isotropic structures follows from formality of the $E_n$ operad and Dunn additivity. 
Those arguments immediately give algebraic deformations in  the smooth setting, 
and can be refined in our setting by replacing coloured operads of linear morphisms with those of polydifferential operators.

For $n=1$, quantisation of co-isotropic structures follows directly from the proof of Theorem \ref{derived0quantthm}
as in \cite[\S \ref{DQpoisson-coisosn2}]{DQpoisson}. A choice of even associator $w$ gives us an involutive equivalence $p_w$ between  $BD_2$-algebras and  $P_2\brh$-algebras, so  a $\hat{P}_{2}$-algebra morphism $(\hatTot\C^{\infty}(X),\varpi) \to \widehat{\Pol}_{\pi}(Z,0)$ combines with the equivalence $\theta_w$ between $\cPol(X,0)$ and  $p_w\cD^{\poly}_{\oplus}$ from the proof of Theorem \ref{derived0quantthm} to give us a quantisation in the form of a strong homotopy $BD_2$-algebra morphism
\[
 p_w(\hatTot\C^{\infty}(X)\brh,\varpi) \to p_w(\widehat{\Pol}_{\pi}(Z,0)\brh) \simeq \hbar Q\widehat{\Pol}_{\theta_w(\pi)}(Z,0).
\]

We now focus on the case $n=0$, in which case \cite{DQLag} (or in the classical setting \cite{BaranovskyGinzburgKaledinPecharich}) defines  a quantisation of $f^{-1}Z \to X$ to be given by a $BD_1$ quantisation   of $(X,\O_X)$ in an analogous  sense to Definition \ref{E1QPdef} (so for us an associative deformation  $\tilde{\O}_X$ of $\O_X$ given by polydifferential operators), a $BD_0$ quantisation of  $(Z,L)$ in the sense of Definition \ref{Qpoissdef0} for a sheaf $L$ of sections of a line bundle (so a deformation $\tilde{L}$ of $L$,  given by differential operators with constraints on their orders) and a suitable $f^{-1}\tilde{\O}_X$-module structure on $\tilde{L}$.

Since the precise definitions  of quantised polyvectors and quantisations \cite[Definitions \ref{DQLag-QPoldef} and \ref{DQLag-QPdef}]{DQLag}  are quite involved, we omit them here. There is again a notion of self-dual quantisation, combining those of \S\S  \ref{quant0sn}, \ref{quantneg1sn}, and \cite[Theorem \ref{DQLag-quantpropsd}]{DQLag}  adapts to the smooth setting, replacing Hochschild complexes with polydifferential operators,  to give: 

\begin{theorem}\label{quantpropsdlag}
Take  a morphism $Z \to X$ of super dg NQ-manifolds, and  a strict line bundle $M$ on $Z$ with a right $\cD$-module structure on $M^{\ten 2}$. Then for any Lagrangian structure  $(Z,\lambda)$ over a $0$-shifted symplectic structure $(X,\omega)$, a  Levi decomposition $w$ for the Grothendieck--Teichm\"uller group corresponding to an even associator gives a parametrisation of self-dual quantisations of   $(Z,\lambda) \to (X,\omega)$ with line bundle $M$   by the group
\[
 \hbar^2\H^1\cone(\DR(X) \to \DR(Z))\brhh.
\]
In particular, $w$ associates a canonical choice of self-dual $BD_0$ quantisation of  $(Z,M)$ to   every Lagrangian structure. 
\end{theorem}

\begin{remark}\label{fukaya}
The characterisation of quantised Lagrangians on $X$ as modules over an  associative deformation  $\tilde{\O}_X$ of a $0$-shifted symplectic structure gives rise, for each choice of $\tilde{\O}_X$ to a dg category whose objects are quantised Lagrangians with spin structures, cf. \cite[\S \ref{DQLag-fukayasn}]{DQLag}. As in \cite[Proposition \ref{DQLag-cfvanish1}]{DQLag}, when a Lagrangian $(Z,\lambda)$ is compact, the complex of morphisms  from $(Z,\lambda,L)$ to another Lagrangian $(Z',\lambda',L')$ with spin structure  is given by a quantisation of a line bundle on the derived Lagrangian intersection $Z\by^h_XZ'$, which necessarily carries a $(-1)$-shifted symplectic structure $\lambda'-\lambda$. 

As in Examples \ref{vanishex} and \cite[Corollary \ref{DQLag-cfvanish2}]{DQLag}, when $Z\by^h_XZ'$ is a derived critical locus, the complex of morphisms is a twisted de Rham complex of a rank one local system. As discussed  in \cite[Remark 6.15]{BBDJS} and \cite[\S 3.3]{schapiraMicrolocalSurvey}, the resulting category thus resembles a Fukaya category, but is defined algebraically and includes objects corresponding to exotic Lagrangians.
\end{remark}

 \subsection{Quantisation of \texorpdfstring{$(-2)$}{-2}-shifted symplectic structures on dg NQ-manifolds}\label{quantneg2sn}

 Fix a super dg NQ-manifold $X$.
 
 \subsubsection{Formulation of \texorpdfstring{$(-2)$}{-2}-shifted quantisations}

The following definitions are adapted from  
\cite[Definitions \ref{DQ-2-htpyrightDmoddef}, \ref{DQ-2-DRrdef}, \ref{DQvanish-QFdef} and \ref{DQvanish-Qpoissdef}]{DQ-2}:

\begin{definition}\label{htpyrightDmoddef}
 We define a homotopy right $\cD$-module structure (or flat right connection) on $\cC^{\infty}(X)$ to be a sequence of maps $\nabla_{p+1} \in \hatTot\HHom_{\R}(\Lambda^pT_{\cC^{\infty}(X)}, \cC^{\infty}(X))^{1-p,=} $ for $p \ge 1$,  satisfying the following conditions:
\begin{enumerate}
 \item For $a \in \cC^{\infty}(X)$ and $\xi \in  T_{\cC^{\infty}}(X)$, we have $\nabla_2(a\xi)= a\nabla_2(\xi) - \xi(da)$;
\item For $p \ge 2$, the maps $\nabla_{p+1}$ are $\cC^{\infty}(X)$-linear;
\item The operations $(\nabla_2 -\id, \nabla_3, \nabla_4, \ldots)$ define an $L_{\infty}$-morphism from the DGLA $ \hatTot T_{\cC^{\infty}}(X)$ to the DGLA $\hatTot (\cC^{\infty}(X) \oplus  T_{\cC^{\infty}(X)})^{\op}= \hatTot F_1\cD_{\C^{\infty}(X)}^{\op}$ of first-order differential operators with bracket given by negating the commutator.
\end{enumerate}
\end{definition}

\begin{definition}\label{DRrdef}
 Given a flat right connection $\nabla$ on $\cC^{\infty}(X)$, we define the right de Rham complex $\DR^r(X,\nabla)$ associated to $\nabla$, and its increasing filtration $F$, by 
\[
 F_i \DR^r(A,\nabla):= \bigoplus_{p \le i}\hatTot\Lambda^pT_{\cC^{\infty}(X)}[p], 
\]
equipped with differential $D^{\nabla}= \sum_{k \ge 1} D^{\nabla}_k$ as follows.  Define  $D^{\nabla}_k \co \Lambda^pT_{\cC^{\infty}(X)} \to \Lambda^{p+1-k}T_{\cC^{\infty}(X)}$ by setting 
 (for 
$\omega \in \Omega^{p+1-k}_A$)  
\begin{align*}
 D^{\nabla}_k(\pi)\lrcorner \omega:= \begin{cases}
                              \nabla_k(\pi \lrcorner \omega) & k > 2,\\
                           \nabla_2(\pi \lrcorner \omega) +(-1)^{\deg \pi} \pi\lrcorner d\omega & k=2,\\
\delta(\pi\lrcorner \omega) \pm \pd(\pi \lrcorner\omega)   &k=1,
\end{cases}
\end{align*}
where $d$ is the de Rham differential and $\delta, \pd$ are induced by the differentials $\delta,\pd$ on $\cC^{\infty}(X)$.
\end{definition}
%
%
%
%

\begin{definition}
Given a flat right connection $\nabla$ on $\cC^{\infty}(X)$, define the complex 
of quantised $(-2)$-shifted polyvector fields  on $X$ by
\[
 Q\widehat{\Pol}(X,\nabla,-2):= \prod_j \hbar^{j-1}F_j\DR^r(X,\nabla).
\]
Define a decreasing filtration $\tilde{F}$ on  $Q\widehat{\Pol}(X,\nabla,-2)$ by 
\[
 \tilde{F}^iQ\widehat{\Pol}(X,\nabla,-2):= \prod_{j \ge i} \hbar^{j-1}F_j\DR^r(X,\nabla).
\]
\end{definition}
It follows as in \cite[Lemma \ref{DQ-2-DRrBVlemma}]{DQ-2}  (following \cite{kravchenko,vitagliano})
that $\DR^r(X,\nabla)$ is a form of filtered   $BV_{\infty}$-algebra. This induces an $L_{\infty}$-algebra structure on $\DR^r(X,\nabla)[-1]$ with brackets
 \[
  [a_1, \ldots, a_k]_{\nabla,k}:= [\ldots [\nabla, a_1], \ldots , a_k](1),
 \]
which extends naturally to an $\R\brh$-linear $L_{\infty}$-algebra structure on   $Q\widehat{\Pol}(X,\nabla,-2)[-1]$.

\begin{definition}\label{Qpoissdef}
Define  the space $Q\cP(X, \nabla,-2)$ of $BD_{-1}$ quantisations (named by analogy; there is no $BD_{-1}$ operad)
of $(X,\nabla)$   to be given by the simplicial
set
\[
 Q\cP(X, \nabla,-2):= \Lim_i \mmc( \tilde{F}^2 Q\widehat{\Pol}(X, \nabla,-2)^=[-1]/\tilde{F}^{i+2}).
\]

The morphism $Q\cP(X,\nabla,-2) \to \cP(X,-2)$, sending a quantisation  to its underlying $(-2)$-shifted Poisson structure, is given by the composite
 \[
  \prod_{j \ge 2} F_j\DR^r(X,\nabla)\hbar^{j-1} \to \prod_{j \ge 2} \gr^F_j\DR^r(X,\nabla)\hbar^{j-1}  \cong F^2\widehat{\Pol}(X,-2).
 \]
 since $ \gr^F_j\DR^r(X,\nabla) \cong \Symm_{\cC^{\infty}(X)}^i(T_{\cC^{\infty}(X)}[1])$.
\end{definition}

\begin{remark}[quantum BV structures]
As above, the $BV_{\infty}$-algebra $Q\widehat{\Pol}(X,\nabla,-2)$ is a deformation quantisation of the $P_0$-algebra $\widehat{\Pol}(X,-2)$, which we can think of as the algebra of functions on the formal shifted cotangent space $\hat{T}^*[-1]X$. Our $(-2)$-shifted quantisation then  deforms a given function on $\hat{T}^*[-1]X$ satisfying the Maurer--Cartan  (or master) equation to a function on its quantisation satisfying the Maurer--Cartan (or quantum master) equation.

Replacing $\Z$-grading with $\Z/2$-gradings, this means $BD_{-1}$ quantisations give rise to quantum BV structures  in the sense of \cite[3.8.1]{merkulovWheeledPropfileBV}, on an $\hbar$-deformation of the $(-1)$-shifted symplectic formal dg manifold $\hat{T}^*[-1]X$.

Quasi-isomorphisms of dg Lie algebras (or $L_{\infty}$-algebras) only induce equivalences on spaces of Maurer--Cartan solutions when the Lie algebras are nilpotent, so more general analogues of quantum BV structures in derived geometry could only be formulated  as Maurer--Cartan solutions in $\tilde{F}^2\sO_{\tilde{Y}}$ for quantisations $\tilde{Y}$  of $(-1)$-shifted symplectic formal  dg manifolds $Y$ with filtrations $F$
similar to that on $\hat{T}^*[-1]X$.
\end{remark}

\subsubsection{Quantisation for flat right connections}

Note that  $Q\widehat{\Pol}(X,\nabla,-2)/\hbar \cong \widehat{\Pol}(X,-2)$, giving a map $Q\cP(X, \nabla,-2) \to \cP(X,-2) $. We regard the fibres of this map over a shifted Poisson structure $\pi$ as quantisations of $\pi$. When $\pi$ is non-degenerate, we also refer to the fibre as the space of deformation quantisations of the corresponding shifted symplectic structure relative to $\nabla$.

\begin{theorem}\label{uniqueconn}
Given a  $(-2)$-shifted symplectic structure $\omega$ on $X$, the space of pairs $(\nabla, S)$, where $\nabla$ is a flat right connection on $\cC^{\infty}(X)$ and $S$ is a  deformation quantisation of $\omega$ relative to $\nabla$, 
is either empty or equivalent to (\emph{left} de Rham cohomology)
\[
 \hbar^2 \H^0\DR(X_0^=)\brh  
\]
 depending on whether any flat right connections on $\cC^{\infty}(X)$ exist; the potential  obstruction lies in $\H^2(F^1\DR(X))$. 

In particular, this shows that when $\pi^0X_0^=$ is connected and locally smoothly contractible, and flat right connections exist,  pairs $(\nabla, S)$  are essentially unique up to addition by $(0,\hbar^2 \R\brh)$.
\end{theorem}
\begin{proof}
 Combining \cite[Propositions  \ref{DQ-2-quantprop}, \ref{DQ-2-uniqueconn}, \ref{DQ-2-DMquantprop}]{DQ-2}, adapted to our setting along similar lines to \S \ref{quantneg1sn}, gives the obstruction and shows that the space is equivalent to  $\pi_j \prod_{i\ge 2}\mmc(\hbar^i\DR(X)^=[-1])$.

Homotopy groups of this space  are given by $\pi_j \prod_{i\ge 2}\mmc(\hbar^i\DR(X)^=[-1]) \cong \hbar^2\H^{-j}\DR(X)\brh$. Since we are willing to forget the filtration, we may compute $\DR(X)$ by completing $\DR(X^0)$ along $\pi^0X$ 
as in \cite{taroyanDRDerivedDG} (see Remark \ref{DRrmkdgNQ}).
Thus $\DR(X)$ has no negative cohomology groups, and our space of pairs $(\nabla, S)$ is discrete. Moreover, when $\pi^0X_0^=$ is  locally smoothly contractible, 
\cite[Proposition 4.11]{taroyanDRDerivedDG}
shows that $\DR(X_0^=)\simeq \oR\Gamma(\pi^0X^=,\R)$, so  $\H^0\DR(X)\cong \H^0(\pi^0X_0^=,\R)$.    
\end{proof}

\subsubsection{Virtual fundamental classes}

By  \cite[Theorem 3.7]{BraunLazarevHtpyBV}, there is an $L_{\infty}$-isomorphism from $\DR^r(X,\nabla)[-1] $ with the $L_{\infty}$-structure $[-]_{\nabla} $  to the complex $\DR^r(X,\nabla)[-1] $ with abelian $L_{\infty}$ structure. 
In particular, for $S \in \DR^r(A,\nabla)^0$, \cite[Remark 3.6]{BraunLazarevHtpyBV} shows that the expression  $\sum_n[S, \ldots,S]_{n, \nabla}/n!$ can be rewritten as 
$e^{-S}D_{S}^{\nabla}(e^S)$, so the  Maurer--Cartan equation $\sum_n[S, \ldots,S]_{n, \nabla}/n!=0$ is equivalent to the quantum master equation $D^{\nabla}(e^{S})=0$.

Our complex  $Q\widehat{\Pol}(X, \nabla,-2)$ is not itself a $BV_{\infty}$-algebra, but it is an $L_{\infty}$-subalgebra of $(\hbar(\DR^r(X,\nabla)[\hbar]/\hbar^r)[-1]; [-]_{\nabla})$. Therefore  sending $S$ to $e^S-1$ gives natural maps
\begin{align*}
  \pi_iQ\cP(X, \nabla,-2) &\to \hbar \H^{-i}(\DR^r(X,\nabla))\brh
\end{align*}
from quantisations to power series 
in right de Rham cohomology. Note that these are not isomorphisms, since quantisations have additional restrictions in terms of the filtration $F$, which in particular allow them to recover Poisson structures. 

In the algebraic setting, \cite[\S \ref{DQ-2-cfBJ}]{DQ-2} then interprets these as classes in Borel--Moore homology. However, that approach does not adapt to the smooth setting. We can still reduce to a right de Rham complex of a dg manifold generated in chain degrees $[0,1]$ by the methods of \cite[\S \ref{DQ-2-Gorensteinsn}]{DQ-2}, but this no longer has an interpretation as Borel--Moore homology, essentially because for co-ordinates $y_j$ in degree $1$, the right de Rham complex involves polynomials, rather than smooth functions, in the vectors $\pd_{y_j}$. 

\section{Functoriality, 
derived and higher Lie groupoids}\label{descentsn}

\subsection{\'Etale functoriality for dg NQ-manifolds}\label{functsn}
Any morphism $X \to Y$ of super dg NQ-manifolds gives rise to a filtered morphism $\DR(Y) \to \DR(X)$ of de Rham complexes, so functoriality for shifted symplectic structures and Lagrangians is straightforward.
Shifted Poisson structures and quantisations are functorial with respect to a generalisation of local diffeomorphisms, but  as in \cite[\S\S \ref{poisson-descentsn}, \ref{poisson-Artindiagramsn}]{poisson} (summarised and generalised in \cite[\S \ref{DQpoisson-intBCsn}]{DQpoisson}) this is subtle to formulate. 

Given a morphism $f \co X \to Y$, we can classify $n$-shifted Poisson structures on $X$ and $Y$ which are strictly compatible  with $f$ by replacing the space   $\Symm_{\cC^{\infty}(X)}^j(T_{\cC^{\infty}(X)}^{[-n-1]})$ of polyvectors in Definition \ref{poissdefLie} with the fibre product given by the limit of the diagram
\[
\xymatrix@R=0ex{ 
\Symm_{\cC^{\infty}(X)}^j(T_{\cC^{\infty}(X)}^{[-n-1]}) \ar[dr]\\
& \cC^{\infty}(X)\ten_{\cC^{\infty}(Y) }\Symm_{\cC^{\infty}(Y)}^j(T_{\cC^{\infty}(Y)}^{[-n-1]}). \\
\ar[ur]
\Symm_{\cC^{\infty}(Y)}^j(T_{\cC^{\infty}(Y)}^{[-n-1]})}
\]
This fibre product only behaves well when it is a homotopy fibre product, or equivalently when one of the maps in the diagram is surjective. There are two main ways this can occur: either because the map $ \cC^{\infty}(Y) \to \cC^{\infty}(X)$ is  surjective, or because the map $ T_{\cC^{\infty}(X)} \to \cC^{\infty}(X)\ten_{\cC^{\infty}(Y) }T_{\cC^{\infty}(Y)}  $ is surjective. The former is the approach taken in \cite{poisson}, but we take the latter to avoid having to use rings which are $\Z$-graded in the chain direction.  


\begin{definition}\label{quasisubdef}
 Say that a morphism $f \co X \to Y$ of dg super NQ-manifolds is a quasi-submersion if $X^{0,=}_0 \to Y^{0,=}_0$ is a submersion and $\C^{\infty}(X)^{\#}_{\#}$ is locally freely generated over $\C^{\infty}(Y)^{\#}_{\#}\ten_{\C^{\infty}(Y)^{0,=}_0}\C^{\infty}(X)^{0,=}_0$. We say that  a morphism of dg supermanifolds is a quasi-submersion if it is so when regarded as a morphism of  dg super NQ-manifolds.
\end{definition}
Note that this condition is not invariant under  $\delta$-quasi-isomorphisms; if we turn off $Q$, quasi-submersions are finitely presented cofibrations in the model structure on $G^+dg_+\C^{\infty}$ from Remark \ref{locmodelNQrmk},
so should be thought of as a computational convenience. Super dg NQ-manifolds themselves correspond to finitely presented cofibrant objects in such a model structure. (See for instance \cite{hovey} for background on model categories.)


We do not refer to these morphisms as submersions because they do not induce epimorphisms on tangent complexes in the derived category. In particular, any attempt to set up a theory of descent using quasi-submersions would destroy the notion of tangent complexes.\footnote{\label{htpysubmersiondef}The correct notion of homotopy submersion $X \to Y$ between dg manifolds  is the analogue of smoothness in derived algebraic geometry, requiring that $X$ be locally quasi-isomorphic to $X \by \R^n$; it is called smoothness in \cite[Definition 5.2.8]{nuitenThesis}. That is equivalent to requiring that
$\Omega^1_Y\ten_{\sO_Y}\sO_X \to \Omega^1_X$ have the left lifting property with respect to surjections in the homotopy category of $\sO_X$-modules in non-negatively graded chain complexes. The argument of \cite[Def 1.2.7.1 and Theorem 2.2.2.6]{hag2} gives the alternative characterisation that  $\pi^0X \to \pi^0Y$ be a submersion with $\sH_*\sO_X\cong \sH_*\sO_Y\ten_{\sH_0\sO_Y}\sH_0\sO_X$.} 


\begin{definition}\label{diagramPdef}  
Given a quasi-submersion $f \co X \to Y$ of super dg NQ-manifolds,
we define $\widehat{\Pol}(X\xra{f} Y,n)$ to be the product $\prod_{j \ge 0}$ of sum-product total complexes 
$\hatTot$ of   the fibre products given by limits of the diagrams
\[
\xymatrix@R=0ex{
  \Symm_{\cC^{\infty}(X)}^j((T_{\cC^{\infty}(X)})^{[-n-1]})\ar[dr]\\
&   \cC^{\infty}(X)\ten_{\cC^{\infty}(Y) }\Symm_{\cC^{\infty}(Y)}^j((T_{\cC^{\infty}(Y)})^{[-n-1]}). \\
   \Symm_{\cC^{\infty}(Y)}^j((T_{\cC^{\infty}(Y)})^{[-n-1]}) \ar[ur]} 
\]
We then follow Definition \ref{poissdefLie} in defining the space of Poisson structures on the diagram $f \co X \to Y$ as
\[
 \cP(X\xra{f} Y,n):= \Lim_i\mmc(\Fil^2 \widehat{\Pol}(X\xra{f} Y,n)^{[n+1]}/\Fil^{i+2}).
\]
We define $\cP(X\xra{f} Y ,\Pi n)$ and $Q\cP(X\xra{f} Y,n)$ ($n=0,-1$) similarly, adapting Definitions \ref{bipoldef}, \ref{Qpoissdef0}, \ref{E1QPdef}.
\end{definition}

Observe that restriction to either of the factors gives morphisms
\[
 \widehat{\Pol}(X,n) \la \widehat{\Pol}(X\xra{f} Y,n) \to \widehat{\Pol}(Y,n).
\]

The following is adapted from the notion of homotopy formally \'etale in \cite[Lemma \ref{smallet2-etlemma2}]{smallet2}, simplified because of our more restrictive hypotheses (cf. Remark \ref{htpyfdrmkd2})\footnote{In the general case, we have to use $\Tot^{\sqcap}$ instead of $\Tot$, but it can be replaced with $\Tot \sigma^{\le N}$, for brutal cotruncation $\sigma^{\le N}$ in the cochain direction, for $N \gg 0$.}.
\begin{definition}\label{locdiffeodef}
 A morphism $f \co X \to Y$ of super dg NQ-manifolds is said to be a homotopy  local diffeomorphism
if the map
\[
 \Tot  (\Omega_{\C^{\infty}(Y)}^1\ten_{\C^{\infty}(Y) }\C^{\infty}(X)^0) \to \Tot (\Omega_{\C^{\infty}(X) }^1\ten_{\C^{\infty}(X)}\C^{\infty}(X)^0)
\]
is a quasi-isomorphism.
\end{definition}

The following is the key to functoriality statements:
\begin{proposition}\label{hfetlemma}
 If a quasi-submersion $f \co X \to Y$ of super dg NQ-manifolds is a homotopy  local diffeomorphism
then the natural maps
\begin{align*}
 \cP(f \co  X \to Y,n) &\to \cP(Y,n)\\
\cP(f \co  X \to Y, \Pi n) &\to \cP(Y, \Pi n)\\
Q\cP(f \co  X \to Y,n) &\to Q\cP(Y,n)
\end{align*}
 are weak equivalences.

 If $f$ is a homotopy  local diffeomorphism and 
$f^0 \co \C^{\infty}(Y)^0 \to \C^{\infty}(X)^0$ is a chain quasi-isomorphism,
then the maps 
\begin{align*}
 \cP(f \co  X \to Y,n) &\to \cP(X,n)\\
\cP(f \co  X \to Y, \Pi n) &\to \cP(X, \Pi n)\\
Q\cP(f \co  X \to Y,n) &\to Q\cP(X,n)
\end{align*}
 are also  weak equivalences; in particular this holds if $f$ is a levelwise $\delta$-quasi-isomorphism.
\end{proposition}
\begin{proof}
By  \cite[Lemma \ref{poisson-bicalcTlemma2}]{poisson}, using the  the proof of \cite[Proposition \ref{poisson-tgtcor2}]{poisson}, the homotopy local diffeomorphism hypothesis leads to quasi-isomorphisms 
\[
\hatTot\Symm_{\cC^{\infty}(X)}^j(T_{\cC^{\infty}(X)}^{[-n-1]})\to \hatTot(\cC^{\infty}(X)\ten_{\cC^{\infty}(Y)} \Symm_{\cC^{\infty}(Y)}^j(T_{\cC^{\infty}(Y)}^{[-n-1]})),
\]
Specifically, that proof (simpler under our restrictions)\footnote{In the general case of Remark \ref{htpyfdrmkd2}, the tensors of tangent modules in the definition have to be replaced with $\Hom$s out of tensors of cotangent modules, and any instance of $\Tot$ has to be replaced with $\Tot \sigma^{ \le N}$, $\Tot \sigma^{\le Np}$ or $\Tot \sigma^{\le Nm}$ in the proof.} 
gives, for $A:= \C^{\infty}(X)$ or $\C^{\infty}(Y)$, $\Omega:=\Omega^1_A$ and $\bar{\Omega}:=\Omega\ten_AA^0$, a completely convergent spectral sequence  
\[
 \EExt^n_{A^0}(\Tot (\bar{\Omega})^{\ten_{A^0}p},\Tot^{\sqcap}(\gr_{\Fil}^i M)) \abuts \H^{n+i}\hatTot \cHom_A( \Omega^{\ten_A p},M)
\]
for all complete filtered $A$-modules $(M,\Fil)$ in double complexes with $M^{<0}=0$ and $A^{>0}\Fil^iM \subset \Fil^{i+1}M$, where $\cHom$ is internal $\Hom$ in double complexes. Taking $M=\cC^{\infty}(X)$ and $\Fil$ the brutal truncation filtration, then passing to symmetric invariants,  yields the required quasi-isomorphisms when we compare the two options for $A$. 

Thus the map  $\widehat{\Pol}(X\xra{f} Y,n) \to \widehat{\Pol}(Y,n)$ is a pullback along a surjective quasi-isomorphism, so is a quasi-isomorphism. The first statement for $\cP$ now follows by applying $\mmc$, and the others follow similarly, applying the towers of Appendix \ref{towersn}. 

For the second statement, set $A:=\C^{\infty}(Y)$ and $B:= \C^{\infty}(X)$, and apply the spectral sequence above, alternately taking $M$ to be $A$ or $B$, filtered by powers of the ideal $B^{>0}$.  It thus suffices for $f^0 \co A^0 \to B^0$ to be a quasi-isomorphism and for the maps $\Tot^{\sqcap}((A^{>0})^m/(A^{>0})^{m+1}) \to  \Tot^{\sqcap}((B^{>0})^m/(B^{>0})^{m+1})$ to be quasi-isomorphisms for all $m>0$. Now
 observe that $B^{>0}/(B^{>0})^2 \cong \ker(\Omega^1_{B}\ten_BB^0 \to \Omega^1_{B^0})$. When $f^0$ is a quasi-isomorphism,   $\Omega^1_{A^0}\ten_{A^0}B^0 \to \Omega^1_{B^0}$ is a quasi-isomorphism by Proposition \ref{QIMtanprop}, so when $f$ is also a homotopy local diffeomorphism we have  a quasi-isomorphism $ \Tot((A^{>0}/(A^{>0})^2)\ten_{A^0}B^0) \to \Tot( B^{>0}/(B^{>0})^2)$ and hence   $ \Tot((A^{>0})^m/(A^{>0})^{m+1})\ten_{A^0}B^0 \to \Tot( (B^{>0})^m/(B^{>0})^{m+1})$. Since $A^0 \to B^0$  is a quasi-isomorphism, this suffices. Thus the maps 
\[
 \hatTot\Symm_{\cC^{\infty}(Y)}^j(T_{\cC^{\infty}(Y)}^{[-n-1]})\to \hatTot(\cC^{\infty}(X)\ten_{\cC^{\infty}(Y)} \Symm_{\cC^{\infty}(Y)}^j(T_{\cC^{\infty}(Y)}^{[-n-1]}))
\]
are quasi-isomorphisms, yielding the second set of equivalences in the same way.

Finally, observe that if $f$ is a levelwise $\delta$-quasi-isomorphism then by Corollary \ref{locdiffeocor} it is a homotopy local diffeomorphism. 
\end{proof}

On equivalence classes of Poisson structures or of quantisations, functoriality for quasi-submersions which are homotopy  local diffeomorphisms is now clear, with a Poisson structure on $Y$ giving rise to a Poisson structure on $X$ via the maps
\[
 \pi_0\cP(Y,n) \cong  \pi_0\cP(f \co  X \to Y,n) \to \pi_0\cP( X,n).
\]

Abstract homotopy theory then permits us to extend this functor to all homotopy local diffeomorphisms of super dg NQ-manifolds, because the homotopy category of a model category is the same as that of the subcategory of cofibrations, with $\infty$-categorical generalisations as in \cite[\S \ref{DQpoisson-intBCsn}]{DQpoisson}.

Explicitly, for any $Y$, model structures guarantee existence of a ``path object'' $PY$ (cf. \cite[Definition 1.2.4]{hovey}), 
which comes with a quasi-submersion $PY \to Y \by Y$ and a levelwise quasi-isomorphism $Y \to PY$ which is a section of both the projection maps $PY \to Y$. For a sketch of the construction in our setting, see Example \ref{pathex}.

Then for any map $f \co X \to Y$, the first projection gives a quasi-submersive levelwise quasi-isomorphism $X \by_YPY \to X$ admitting a section, which combines with the second projection $X \by_YPY \to Y$ to factorise $f$. 
When $X$ is a homotopy local diffeomorphism,  repeated application of Proposition \ref{hfetlemma} gives
\begin{align*}
 \pi_0\cP(Y,n) \to &\pi_0\cP(  X\by_YPY \to Y,n)\\
&  \cong \pi_0\cP(X\by_YPY,n)\\
& \cong \pi_0\cP(X\by_YPY \to X,n)\\
& \cong  \pi_0\cP(X,n).
\end{align*}
This suffices in cases such as $2$-shifted structures   in Example \ref{exCasimir} 
where $\cP$ has no  homotopy groups, but in general descent and gluing arguments for these structures are required, involving $\infty$-categories, or $m$-categories if the higher homotopy groups $\pi_{>m}\cP(X,n)$ vanish.  As in \cite[\S \ref{poisson-descentsn}]{poisson}, the data here are best suited to   construct  complete Segal spaces in the sense of \cite[\S 6]{rezk}.

\subsection{Lie groupoids}\label{liegpdsn}

Say we have a Lie groupoid 
$\fX:=(X_1 \Rightarrow X_0)$, so  $X_0$ and $X_1$ are manifolds (regarded as the spaces of objects and of  morphisms), we have an identity section $\sigma_0 \co X_0 \to X_1$, source and target maps $\pd_0, \pd_1 \co X_1 \to X_0$, and an associative multiplication $X_1\by_{\pd_0,X_0,\pd_1}X_1 \to X_1$. We can then form the nerve
\[
B\fX:= (\xymatrix@1{ X_0 \ar@{.>}[r]& \ar@<1ex>[l] \ar@<-1ex>[l] X_1 \ar@{.>}@<0.75ex>[r] \ar@{.>}@<-0.75ex>[r]  & \ar[l] \ar@/^/@<0.5ex>[l] \ar@/_/@<-0.5ex>[l] 
X_2 &  &\ar@/^1pc/[ll] \ar@/_1pc/[ll] \ar@{}[ll]|{\cdot} \ar@{}@<1ex>[ll]|{\cdot} \ar@{}@<-1ex>[ll]|{\cdot}  X_3  & & \cdots\ar@/^1.2pc/[ll] \ar@/_1.2pc/[ll]\ar@{}[ll]|{\cdot} \ar@{}@<1ex>[ll]|{\cdot} \ar@{}@<-1ex>[ll]|{\cdot}
})
\]
by setting 
\[
 X_m := \overbrace{X_1\by_{\pd_0,X_0,\pd_1}X_1 \by_{\pd_0,X_0,\pd_1} \ldots \by_{\pd_1,X_0,\pd_0}X_1}^m,
\]
the space of $m$-strings of morphisms.

If we wish to define Poisson structures and quantisations for our Lie groupoid $\fX$,  
then we encounter the difficulty that Proposition \ref{hfetlemma} only applies to local diffeomorphisms rather than submersions, so we could only apply descent arguments directly  if the source and target maps were local diffeomorphisms, meaning $\fX$ would be an orbifold.

As in  \cite[\S \ref{poisson-cfArtinstackysn}]{poisson} or \cite{CPTVV}, the solution is to resolve the Lie groupoid by Lie algebroids, in the form of NQ-manifolds. 
\begin{definition}
 Given a Lie groupoid $ \fX:=(X_1 \Rightarrow X_0)$, we define the NQ-manifold $(\hat{X}_1\Rightarrow X_0)=(X_0, \O_{(\hat{X}_1\Rightarrow X_0) })$ 
by first forming the Lie algebroid $A\fX$ associated to the Lie groupoid $\fX$ as in \cite{mackenzieLieGpdLieAlgd}, then taking the associated NQ-manifold as in \cite{LiuWeinsteinXu}.
\end{definition}


\begin{example}
If $G$ is a Lie group acting on a manifold $M$,  there is a Lie groupoid  $[M/G]:= (G \by M \Rightarrow M)$, and then we have
\[
 (\widehat{G \by M} \Rightarrow M) \cong [M/\g],
\]
for the NQ-manifold $[M/\g]:= (M, \O_{[M,\g]})$, where $\O_{[M,\g]}$ is given by the Chevalley--Eilenberg complex
\[
\O_{[M,\g]}:= \mathrm{CE}(\g,\O_M)=(  \O_M \xra{Q} \O_M\ten\g^* \xra{Q} \O_M\ten\Lambda^2\g^*\xra{Q} \ldots).
\]
\end{example}

\begin{definition}\label{completiondef}
 Given a Lie groupoid $ \fX:=(X_1 \Rightarrow X_0) $ and a submersion $Y_0 \to X_0$, we define the NQ-manifold $\hat{\fX}_{Y_0}$ (the completion of $\fX$ along $Y_0$) by first setting $Y_1:=Y_0\by_{X_0,\pd_1}X_1\by_{\pd_0,X_0}Y_0$ and then
\[
 \hat{\fX}_{Y_0}:= ( \hat{Y}_1  \Rightarrow  Y_0),
\]
for the Lie groupoid $( Y_1  \Rightarrow  Y_0)$ homotopy equivalent to $\fX$ with objects $Y_0$.
\end{definition}
In particular, note that $ \hat{\fX}_{X_0}= (\hat{X}_1\Rightarrow X_0)$ and that the construction of $\hat{\fX}_{Y_0} $ is invariant if we replace $\fX$ with a homotopy equivalent Lie groupoid.

%
The examples we are interested in are $ \hat{\fX}_{X_j}$, which can alternatively be written as
$
 ( \widehat{(X^{\Delta^j})}_1 \Rightarrow (X^{\Delta^j})_0),
$
the completion of the groupoid 
\[
 ( (X^{\Delta^j})_1 \Rightarrow (X^{\Delta^j})_0) =(X_j\by_{(X_0)^{j+1}, \pd_0}(X_1)^{j+1} \Rightarrow X_j)
\]
  of $j$-strings of morphisms and commutative diagrams between them.

We are now in a position to define Poisson structures and quantisations for Lie groupoids. We can think of the Lie algebroid $A\fX$ as being an infinitesimal neighbourhood of $X_0$ in $\fX$, so that in particular we have a local diffeomorphism $A\fX \to \fX$, and then the diagram $j\mapsto  \hat{\fX}_{X_j}$ gives a simplicial resolution of $\fX$ by NQ-manifolds which are quasi-submersively locally diffeomorphic to it\footnote{Specifically, \cite[Corollary \ref{smallet2-gooddescent}]{smallet2} gives a fully faithful functor $D_*$ from differentiable stacks to stacks on   a site of $\C^{\infty}$-DGAs, and  \cite[Theorem \ref{smallet2-mainthm}]{smallet2} shows that $j\mapsto  \hat{\fX}_{X_j}$ gives an \'etale resolution of  $D_*\fX$. However, we will take a more pedestrian approach, with  Proposition \ref{biinftyFXwell0} showing by hand that constructions are independent of the choice of atlas $X_0 \to \fX$.}. This allows us  to pull back Poisson structures and quantisations, so as in \cite[Definition \ref{poisson-PdefArtin}]{poisson} we can infer what the spaces of  Poisson structures and quantisations on $\fX$ have to be by looking at the corresponding spaces on the NQ-manifolds $\hat{\fX}_{X_j}$:

\begin{definition}\label{preDstardef}
Given a Lie groupoid $\fX =(X_1 \Rightarrow X_0)$, with nerve $X_{\bt}:=B\fX$,  we define the space $\cP(\fX,n)$ of $n$-shifted Poisson structures and the space $Q\cP(\fX,n)$ of  $BD_{n+1}$ quantisations (the latter for $n=0,-1$) by first forming the simplicial NQ-manifold
\[
\hat{\fX}_{B\fX}:= (\xymatrix{ \hat{\fX}_{X_0} \ar@{.>}[r]& \ar@<1ex>[l] \ar@<-1ex>[l] \hat{\fX}_{X_1} \ar@{.>}@<0.75ex>[r] \ar@{.>}@<-0.75ex>[r]  & \ar[l] \ar@/^/@<0.5ex>[l] \ar@/_/@<-0.5ex>[l] 
\hat{\fX}_{X_2} &  &\ar@/^1pc/[ll] \ar@/_1pc/[ll] \ar@{}[ll]|{\cdot} \ar@{}@<1ex>[ll]|{\cdot} \ar@{}@<-1ex>[ll]|{\cdot}  \hat{\fX}_{X_3}  & & \cdots\ar@/^1.2pc/[ll] \ar@/_1.2pc/[ll]\ar@{}[ll]|{\cdot} \ar@{}@<1ex>[ll]|{\cdot} \ar@{}@<-1ex>[ll]|{\cdot}
}),
\]
then observing that the morphisms $\pd_i$ in the diagram are all homotopy local diffeomorphisms of NQ-manifolds in the sense of Proposition \ref{hfetlemma}, giving functoriality and allowing us to take homotopy limits
\begin{align*}
 \cP(\fX,n)&:= \ho\Lim_{j \in \Delta} \cP(\hat{\fX}_{X_j},n),\\
Q\cP(\fX,n)&:= \ho\Lim_{j \in \Delta} Q\cP(\hat{\fX}_{X_j},n).
\end{align*}
\end{definition}
These homotopy limits of cosimplicial spaces are given by the functor $\Tot_{\bS}$ of  \cite[\S VIII.1]{sht}; for  generalities on homotopy limits, see \cite{bousfieldkan,hovey,hirschhorn}, and for simplifications in special cases see e.g. \cite[Examples \ref{stacks2-categs}]{stacks2}. In particular, the homotopy limit of a cosimplicial diagram of chain complexes corresponds to a product total complex, with a similar statement for cosimplicial diagrams of simplicial abelian groups via the Dold--Kan correspondence.\footnote{In particular, that implies that for diagrams of simplicial abelian groups, the homotopy limit over $\Delta$ is equivalent to the homotopy limit over the subcategory $\Delta^+ \subset \Delta$ of injective morphisms, corresponding to semi-cosimplicial diagrams. It thus follows that we could just take homotopy limits over $\Delta^+$ in the definitions of $\cP(\fX,n)$ and $Q\cP(\fX,n)$, since the spectral sequences of groups and sets in \S \ref{filttowersn} reduce the comparison to abelian graded pieces.}

When the spaces $\cP(\hat{\fX}_{X_j},n)$  have trivial homotopy groups (but non-trivial $\pi_0$), as occurs when $n=2$, our homotopy limit above is just the equaliser of the maps
\[
 \pi_0\cP(\hat{\fX}_{X_0},n) \Rightarrow \pi_0\cP(\hat{\fX}_{X_1},n),
\]
and similarly for $ Q\cP(\fX,n)$. 

When the homotopy groups stop at $\pi_1$, as occurs for $n=1$, we instead have to include the datum of an isomorphism $g$ in the fundamental groupoid $\pi_f\cP(\hat{\fX}_{X_1},n)$ between the two images $\pd^0(\pi), \pd^1(\pi)$ of a Poisson structure $\pi \in \cP(\hat{\fX}_{X_0},n)$, satisfying the cocycle condition $\pd^1(g) \simeq \pd^2(g) \circ  \pd^0(g)$ in $\pi_f\cP(\hat{\fX}_{X_2},n) $. Note that the condition $\sigma^0(g)=\id$ is then automatically satisfied, because the cocycle condition implies $\sigma^0(g)= \sigma^0(g) \circ \sigma^0(g)$. 

In general, since the tangent complex is concentrated in cohomological degrees $[-1, \infty)$, the space $ \cP(\hat{\fX}_{X_1},n)$ is empty for $n \ge 3$ because the governing DGLA $\widehat{\Pol}(X,n)^{[n+1]}$ is concentrated in cohomological degrees $[n-1, \infty)$. Meanwhile, for $n<0$ the DGLA is unbounded below in general, so there is no bound on the homotopy groups,  
further complicating the description.

\begin{example}\label{exCasimirGpd}
When $\fX=[M/G]$, then $X_{\bt}:=B\fX$ is given by $X_m= M \by G^m$ and  as in \cite[Example \ref{poisson-DstarBG}]{poisson},  $\fX_{X_m} $
is  the NQ-manifold $[M \by G^m/\g^{\oplus(m+1)}]$ associated to the Lie groupoid  $[M \by G^m/G^{m+1}]$, with action  
\[
 (y,h_1, \ldots,h_m)(g_0, \ldots,g_m)= (y g_0, g_0^{-1}h_1g_1,g_1^{-1}h_2g_2, \ldots, g_{m-1}^{-1}h_mg_m).
\]
 
As in \cite[Example \ref{poisson-2PBG}]{poisson}, the triviality of homotopy groups then gives the space of $2$-shifted Poisson structures as 
\[
 \cP( [M/G],2) \cong \{ \pi \in (S^2\g \ten \C^{\infty}(M))^G ~:~ [\pi, a]=0 \in \g \ten \C^{\infty}(M) ~\forall a \in O(M)\}.
\]
For more general Lie groupoids, there is a similar description of $2$-shifted Poisson structures as invariant $2$-shifted Poisson structures on the associated Lie algebroid.
\end{example}

\begin{example}\label{exQuasiPoisson}
To construct a  $1$-shifted Poisson structure on $[M/G]$,  we start with a $1$-shifted Poisson structure $\pi$ on the NQ-manifold $[M/\g]$, which as in Example \ref{exQuasiLie} 
is the same as a quasi-Lie bialgebroid structure. For each of the morphisms $\pd_i \co [M \by G/\g\oplus \g] \to [M/G]$ ($i=0,1$) in the nerve,  Proposition \ref{hfetlemma} then gives essentially unique quasi-Lie bialgebroid structures $\pd^i\pi$ on $[M \by G/\g\oplus \g]$; in the terminology of \cite[\S 2.2]{BonechiCiccoliLaurentGengouxXu}, $\pd^i\pi$ is projectable to $[M/\g]$ along $\pd_i$. To complete the data required to define a $1$-shifted Poisson structure on $[M/G]$, we then need 
 a gauge transformation, i.e. a twist $\lambda \in \L^2(\g \oplus \g) \ten \C^{\infty}(M\by G)$ between the two pullbacks $\pd^0\pi, \pd^1\pi$ of the Poisson structure to $[M \by G/\g\oplus \g]$, and $\lambda$ must satisfy the cocycle condition $\pd^1\lambda = \pd^0\lambda + \pd^2\lambda$ on $[M \by G^2/\g^{\oplus(3)}]$. An isomorphism $(\pi, \lambda) \to (\pi',\lambda')$ is then given by a twist $\tau \in \L^2(\g) \ten \C^{\infty}(M)$ satisfying $ \pd^1\tau + \lambda = \lambda' + \pd^0\tau$.


When $M$ is a point, $\pi$ is trivial and $\lambda$ is a quasi-Poisson structure on $G$,
 a generalisation of a Poisson-Lie structure; see  \cite[Theorem 2.9]{safronovPoissonLie} for the corresponding statement in the setting of \cite{CPTVV}.

By \cite[Theorem 3.29]{safronovPoissonLie}, a source-connected smooth algebraic quasi-Poisson groupoid in the sense of \cite{IPLGX} corresponds to a $1$-shifted Poisson structure on the Lie groupoid, and a similar argument should apply in the $\C^{\infty}$ setting. We can certainly say that every quasi-Poisson structure gives rise to a $1$-shifted Poisson structure: by Morita equivalence, a quasi-Poisson structure  \cite[Theorem 3.11]{BonechiCiccoliLaurentGengouxXu} on  a Lie groupoid $\fX:=(X_1 \Rightarrow X_0)$ gives rise to essentially unique quasi-Poisson structures on the Lie groupoids $((X^{\Delta^j})_1 \Rightarrow X_j)$, and hence to compatible quasi-Lie bialgebroid structures on the Lie algebroids $\hat{\fX}_{X_j}$, by  \cite[Theorem 4.9]{IPLGX}, which is precisely what it means to give a $1$-shifted Poisson structure on $\fX$. To complete the equivalence, it would suffice to compare the tangent and obstruction spaces as in \cite[\S \ref{poisson-towersn}]{poisson}; in the notation of \cite{
BonechiCiccoliLaurentGengouxXu}, this would amount to showing that the  $2$-term complex $\Sigma^{k+1}(
A) \to \cT_{mult}^k\Gamma$ is good truncation of the complex $C^{\bt}(\Gamma, S^{k+1} T_{\fX})$ in degrees $\le -k$ for $k=1,2$, which in the source-connected case should follow from \cite[Proposition 2.35 and Equation (18)]{IPLGX} as in \cite{safronovPoissonLie}.


\end{example}

The following is analogous to a special case of \cite[Proposition \ref{poisson-biinftyFXwell}]{poisson}. In particular, it ensures that $n$-shifted Poisson structures and quantisations are invariant under Morita equivalence, so are invariants of differentiable stacks. For $\cP(\fX,1)$, this plays the same role as \cite[Theorem 3.11]{BonechiCiccoliLaurentGengouxXu} does for quasi-Poisson structures.
\begin{proposition}\label{biinftyFXwell0}
 Given a  Lie groupoid $ \fX:=(X_1 \Rightarrow X_0) $ and a submersion $f_0 \co Y_0 \to X_0$, let $\fY$ be the Lie groupoid $(Y_0 \by_{X_0}X_1\by_{X_0}Y_0 \Rightarrow Y_0)$ given by pulling $\fX$ back to $Y_0$. 
Then there are natural maps
\begin{align*}
 f^* \co \cP(\fX,n) &\to \cP(\fY,n)\\
f^* \co Q\cP(\fX,n) &\to \cP(\fY,n)
\end{align*}
  in the homotopy category of simplicial sets.
If $f_0$ is surjective,  
then these maps 
 are  weak equivalences.
\end{proposition}
\begin{proof}
This follows as in the proof of \cite[Proposition \ref{poisson-inftyFXwell}]{poisson}. By the construction of $\fY$, we have $\hat{\fY}_{Y_j}\cong \hat{\fX}_{Y_j}$.
Since $f_0$ is a submersion, the maps $\hat{\fY}_{Y_j} \to\hat{\fX}_{X_j}$ are then all homotopy  local diffeomorphisms, so Proposition \ref{hfetlemma} gives  compatible maps $\cP(\hat{\fX}_{X_j},n) \to \cP(\hat{\fY}_{Y_j},n)$, and hence $f^* \co \cP(\fX,n) \to \cP(\fY,n)$ on passing to homotopy limits. DGLA obstruction theory gives rise to towers of obstruction spaces for the homotopy limits, and cohomological descent ensures that these are isomorphisms when $f_0$ is surjective.
\end{proof}

\subsection{Higher Lie groupoids}\label{higherliegpdsn}



\subsubsection{Super Lie $k$-groupoids}

\begin{definition}
 Given a simplicial manifold $X_{\bt}$ and a simplicial set $K$, we follow \cite{zhu} in writing $\mathrm{hom}(K,X_{\bt})$ for the set of homomorphisms of simplicial manifolds from $K$ to $X_{\bt}$, with its natural topology. 
As in \cite[Lemma 2.1]{zhu}, when $K$ is a finite contractible simplicial set,  $\mathrm{hom}(K,X_{\bt})$ is naturally a manifold. 
\end{definition}
In particular, note that the combinatorial $m$-simplex $\Delta^m$ is contractible, with  $\mathrm{hom}(\Delta^m,X_{\bt})=X_m$. 
Another important class of contractible finite simplicial sets is given by the horns $\L^{m,i}\subset \Delta^m$ for all $m\ge 1$, $ 0 \le i \le m$, which are defined by removing the interior and the $i$th face from $\Delta^m$.

We now work with Lie $k$-groupoids in the sense of \cite[Definition 1.2]{zhu}, generalised in the obvious way with supermanifolds rather than just manifolds. In particular, for a simplicial supermanifold $X_{\bt}$ and a finite contractible simplicial set $K$, there is a natural supermanifold $\mathrm{hom}(K,X_{\bt})$, defined as a limit of the supermanifolds $X_m$. By \cite[Theorem \ref{stacks2-bigthm} and Examples \ref{stacks2-affhgpddef}]{stacks2}, the  $\infty$-category given by simplicially localising Lie $k$-groupoids at submersive hypercovers (see Remark \ref{highermoritarmk}) is equivalent to that of ($k$, surjective submersion)-geometric stacks in supermanifolds  defined inductively by analogy with \cite{hag2}.

\begin{definition}\label{superLiegpddef}
 A super Lie $k$-groupoid $X_{\bt}$ is a simplicial diagram 
\[
 \xymatrix{ X_0 \ar@{.>}[r]& \ar@<1ex>[l] \ar@<-1ex>[l] X_1 \ar@{.>}@<0.75ex>[r] \ar@{.>}@<-0.75ex>[r]  & \ar[l] \ar@/^/@<0.5ex>[l] \ar@/_/@<-0.5ex>[l] 
X_2 &  &\ar@/^1pc/[ll] \ar@/_1pc/[ll] \ar@{}[ll]|{\cdot} \ar@{}@<1ex>[ll]|{\cdot} \ar@{}@<-1ex>[ll]|{\cdot}  X_3  & & \cdots\ar@/^1.2pc/[ll] \ar@/_1.2pc/[ll]\ar@{}[ll]|{\cdot} \ar@{}@<1ex>[ll]|{\cdot} \ar@{}@<-1ex>[ll]|{\cdot}
}
\]
of supermanifolds, with the partial matching maps
\[
 X_m \to \mathrm{hom}(\L^{m,i},X_{\bt})
\]
being surjective submersions for all $m\ge 1$, $ 0 \le i \le m$, and diffeomorphisms for $m >k$.
\end{definition}

\begin{examples}
Giving a super Lie $0$-groupoid $X_{\bt}$ is equivalent to giving the supermanifold $X_0$, since the matching conditions imply that $X_m=X_0$ for all $m$.

Meanwhile, a super Lie $1$-groupoid is just equivalent to a super Lie groupoid: we have  supermanifolds $X_0$ and $X_1$ (regarded as the objects and the morphisms), an identity $\sigma_0 \co X_0 \to X_1$, source and target maps $\pd_0, \pd_1 \co X_1 \to X_0$ and diffeomorphisms
\[
 X_m \cong \overbrace{X_1\by_{X_0}X_1 \by_{X_0} \ldots \by_{X_0}X_1}^m,
\]
with the face map $X_1\by_{X_0}X_1 \cong  X_2  \xra{\pd_1} X_1$ thus giving rise to the multiplication operation, and the higher conditions ensuring associativity.
\end{examples}

The partial matching conditions on a Lie $k$-groupoid $X_{\bt}$ imply that the boundary maps $\pd_i \co X_m \to X_{m-1}$ are all submersions. When they are local diffeomorphisms, $X_{\bt}$ is a form of higher orbifold, and we can define Poisson structures and quantisations simply by appealing to Proposition \ref{hfetlemma}: we just have to have a compatible system of Poisson structures or quantisations on the diagram $X_{\bt}$.

We now generalise the construction of Definition \ref{preDstardef} to apply to super Lie $k$-groupoids.
For any commutative cochain algebra $B= B^{\ge 0}$, the Dold--Kan denormalisation $DB$ is naturally a cosimplicial commutative algebra via the Eilenberg--Zilber shuffle product, as for instance in \cite[Definition \ref{ddt1-nabla}]{ddt1}. This functor has a left adjoint $D^*$, which we now describe explicitly.
Given a finite set  $I$ of strictly positive integers, write $\pd^I= \pd^{i_s}\ldots\pd^{i_1}$, for $I=\{i_1, \ldots i_s\}$, with $1 \le i_1 < \ldots < i_s$. 

\begin{definition}\label{Dstardef}
Given a commutative cosimplicial super algebra $A$, we define the commutative super cochain algebra $D^*A$ as follows. We first consider the cochain complex $NA$ given by
\[
 N^mA:= \{a \in A^m ~:~ \sigma^ja = 0 \in A^{m-1}, ~\forall~ 0 \le i <m\},
\]
with differential $Qa:= \sum(i-1)^i \pd^ia$. We then define an associative (non-commutative) product $\smile$ (a variant of the Alexander--Whitney cup product) on $NA$ by 
\[
a \smile b := (\pd^{[m+1,m+n]}a)\cdot (\pd^{[1,m]}b)  
\]
for  $a \in N^mA$, $b \in N^nA$. 

The commutative cochain algebra  $D^*A$ is then  the quotient of $NA$ by the relations
\[
(\pd^Ia)\cdot (\pd^J b) \sim \left\{ \begin{matrix} (-1)^{(J, I)} (a\smile b) & a \in A^{|J|},~ b\in A^{ |I|},\\ 0 & \text{ otherwise},\end{matrix} \right.
\]
for (possibly empty) sets $I,J$ with $I \cap J= \emptyset$,
where for disjoint sets $S,T$ of integers, $(-1)^{(S,T)}$ is the sign of the shuffle permutation of $S \sqcup T $ which sends the first $|S|$ elements to $S$ (in order), and the remaining $|T|$ elements to $T$ (in order). 
\end{definition}

\begin{remarks}
When a cosimplicial $\R$-super algebra $A^{\bt}$ is quasi-freely generated over $A^0$ in the sense that 
we have super vector spaces $V^m \subset A^m$ with $A^m\cong A^0\ten\Symm_{\R}V^m$, closed under degeneracy maps, so $\sigma^i(V^m) \subset V^{m-1}$, then $D^*A \cong A^0\ten\Symm_{\R}(NV)$ if we forget the differential $Q$.   

As observed in \cite[\S \ref{poisson-stackyCDGAsn}]{poisson}, $D^*A$ depends only on the formal completion of $A$ with respect to $\ker(A \to A^0)$, and then the description above also applies if the completion is a quasi-freely generated cosimplicial power series ring over $A^0$.

As in \cite[Example \ref{poisson-DstarBG}]{poisson}, if $A$ is the cosimplicial ring of functions on the nerve of the Lie groupoid $[M/G]$ (so $A^m = \C^{\infty}(M\by G^m)$), then $D^*A$ is the Chevalley--Eilenberg complex $\mathrm{CE}(\g, \C^{\infty}(M))$.

\end{remarks}

The following generalise the passage from Lie groupoids to Lie algebroids (cf. Definition \ref{preDstardef}):

\begin{definition}
 Given a simplicial supermanifold $X_{\bt}$, define the normalisation $NX$ to be manifold $X^{=}_0$ equipped with the super cochain algebra $D^*( (\sigma^0)^{-\bt}\O_X)$, where $(\sigma^0)^{-\bt}\O_X $ is the cosimplicial sheaf $ ((\sigma^0)^{-\bt}\O_X)^m:= ((\sigma^0)^m)^{-1}\O_{X_m}$ on $X^{=}_0$.
\end{definition}


\begin{lemma}\label{Dstarlemma}
If $X_{\bt}$ is a super Lie $k$-groupoid, then $NX$ is a super NQ-manifold, with $\O_{NX}$ generated in cochain degrees $\le k$.
\end{lemma}
\begin{proof} 
 This follows directly from  properties of Dold--Kan normalisation, as observed in \cite[\S \ref{poisson-cfArtinstackysn}]{poisson}; we now flesh out the argument explicitly.

 We begin by looking at the almost simplicial supermanifold $X_{\#}$ underlying $X$, meaning that we forget $\pd_0$ (terminology following \cite[\S 2]{millerCorrectionSullivanConj}). The adjunction $D^* \dashv D$  evidently descends to  an adjunction between graded superalgebras (forgetting $Q$) and  almost commutative  superalgebras. In particular, for $A:=(\sigma^0)^{-\bt}\sO_{X}$,  the graded superalgebra $(D^*A)^{\#}=\O_{NX}^{\#}$ underlying $\O_{NX}$ only depends on $X_{\#}$.
 
 Since the functor of graded  $A^0$-superalgebra derivations from $D^*A$ to graded $A^0$-supermodules $M$ is represented by $\Omega^1_{D^*A}\ten_{D^*A}A^0$, we must have 
\[
 \Hom_{A^0}(\Omega^1_{D^*A}\ten_{D^*A}A^0,M) \cong \Hom_{A^0}(\Omega^1_{A}\ten_A A^0, DM),
 \]
 for $\Hom$ taken respectively in the graded and almost cosimplicial categories, and tensor products on the right taken levelwise, so 
 \[
  (\Omega^1_{D^*A}\ten_{D^*A}A^0)^i \cong N^i (\Omega^1_{A}\ten_{A} A^0) \cong ((\sigma^0)^{-i}\Omega^1_{X_i/\mathrm{hom}(\L^{i,0},X)}\ten_{A^i} A^0)
 \]
which by hypothesis is a finitely generated projective $A^0$ module  which vanishes for $i>k$.
 
Because the structure of $\O_{NX}^{\#}$ is purely algebraic in positive degrees,  for  $K:= \ker(D^*A^{\#} \to A^0)$  we have $\Omega^1_{D^*A/A^0}\ten_{D^*A}A^0 \cong K/(K\cdot K)$. The map  $K \to  K/(K \cdot K)$ admits an $A^0$-linear section since the target is projective, and the images generate $D^*A^{\#}$ as an $A^0$-algebra. 

It only remains to show that $D^*A^{\#}$ is freely generated  over $A^0$ by that graded projective supermodule. This amounts to showing that for any commutative graded $A^0$-algebra $B$ with $B^0=A^0$, every homomorphism $D^*A^{\#} \to B^{<n}$ lifts to a homomorphism $D^*A^{\#} \to B^{\le n}$, for all $n$. By adjunction and \cite[Lemma \ref{stacks2-truncate2}]{stacks2}, this means lifting a map $f \co A^{\#} \to D(B^{<n})$ of almost cosimplicial $\C^{\infty}$-rings to a map $\tilde{f} \co A^{\#} \to D(B^{\le n})$. This in turn just entails lifting the map $f^n \co A^n \to D^n(B^{<n})$ to a map $\tilde{f}^n \co A^n \to D^n(B^{\le n})$ satisfying $\tilde{f}^n \circ \pd^i = \pd^i \circ f^{n-1}$ for all $i >0$; we automatically have $\sigma^j \circ \tilde{f}^n = f^{n-1} \circ \sigma^j$. In other words, surjectivity of the map
\[
 X_n(D^n(B^{\le n})) \to \mathrm{hom}(\L^{n,0},X)(D^n(B^{\le n}))\by_{\mathrm{hom}(\L^{n,0},X)(D^n(B^{< n}))} X_n(D^n(B^{< n}))
\]
suffices. Since the map $D^n(B^{\le n}) \to D^n(B^{< n})$ is surjective, with square-zero kernel $B^n$, the desired surjectivity follows because $X_n \to \mathrm{hom}(\L^{n,0},X)$ is a submersion.
 \end{proof}

We are now in a position to define Poisson structures and quantisations for super Lie $k$-groupoids.

\begin{definition}\label{matching}
 Given a simplicial supermanifold $X_{\bt}$ and a finite contractible simplicial set $K$, define the simplicial supermanifold $X^K_{\bt}$ by 
$ (X_{\bt}^K)_m:= \mathrm{hom}(K \by \Delta^m,X_{\bt})$.
\end{definition}

For a super Lie $k$-groupoid $X_{\bt}$, we can now form a simplicial super NQ-manifold
\[
\xymatrix@1{ NX \ar@{.>}[r]|-{\sigma_0}& \ar@<1ex>[l]^-{\pd_0} \ar@<-1ex>[l]_-{\pd_1} N(X^{\Delta^1}) \ar@{.>}@<0.75ex>[r] \ar@{.>}@<-0.75ex>[r]  & \ar[l] \ar@/^/@<0.5ex>[l] \ar@/_/@<-0.5ex>[l] 
 N(X^{\Delta^2}) &  &\ar@/^1pc/[ll] \ar@/_1pc/[ll] \ar@{}[ll]|{\cdot} \ar@{}@<1ex>[ll]|{\cdot} \ar@{}@<-1ex>[ll]|{\cdot}  N(X^{\Delta^2}) & \ldots&\ldots;}
\]
this resolves $X_{\bt}$ in the sense of \cite[Proposition \ref{poisson-replaceprop}]{poisson}, and the morphisms $\pd_i$ are all quasi-submersive homotopy local diffeomorphisms in the sense of Definitions \ref{quasisubdef} and \ref{locdiffeodef}.  Poisson structures and quantisations are functorial with respect to quasi-submersive homotopy local diffeomorphisms, so  we proceed as in Definition \ref{preDstardef}, following \cite[Definition \ref{poisson-PdefArtin}]{poisson}: 

\begin{definition}\label{PdefArtin}
Given a super  Lie $k$-groupoid $X_{\bt}$,  we define the spaces $\cP(X_{\bt},n)$ of $n$-shifted Poisson structures, $\cP(X_{\bt},\Pi n)$ of parity-reversed $n$-shifted Poisson structures  and the space $Q\cP(X_{\bt},n)$ of  $BD_{n+1}$ quantisations (the latter for $n=0,-1$) by  taking homotopy limits
\begin{align*}
 \cP(X_{\bt},n)&:= \ho\Lim_{j \in \Delta} \cP( N(X^{\Delta^j}),n),\\
\cP(X_{\bt}, \Pi n)&:= \ho\Lim_{j \in \Delta} \cP( N(X^{\Delta^j}), \Pi n),\\
Q\cP(X_{\bt},n)&:= \ho\Lim_{j \in \Delta} Q\cP(N(X^{\Delta^j}) ,n).
\end{align*}
Define their subspaces of non-degenerate elements similarly.
\end{definition}

Theorem \ref{compatthmLie} 
 then gives rise to  an equivalence between shifted symplectic and non-degenerate Poisson structures on super  Lie $k$-groupoids, by reasoning  as in \cite[Theorem \ref{poisson-Artinthm}]{poisson}. Likewise, the results of \S \ref{quantnsn} extend via these constructions from super NQ-manifolds to super Lie $k$-groupoids, as summarised at the end of this section.

\begin{remark}\label{highermoritarmk}
 Since submersions and local diffeomorphisms are the analogues in differential geometry of smooth and \'etale maps, in the terminology of \cite{stacks2,stacksintro} a  super Lie $k$-groupoid $X_{\bt}$ is an Artin $k$-hypergroupoid in supermanifolds, and the construction above has allowed us to replace $X_{\bt}$ with a Deligne--Mumford $k$-hypergroupoid $N(X^{\Delta^{\bt}})$ in super NQ-manifolds. A further crucial property of this construction is that it preserves homotopy equivalences and hypercovers, generalising Proposition \ref{biinftyFXwell0}, so the spaces of Poisson structures and of quantisations  depend only on the hypersheafification of the super  Lie $k$-groupoid on the big site of supermanifolds (with covers generated by surjective submersions); in other words, they depend only on the Morita equivalence class, so are invariants of differentiable superstacks.

 Precisely, Morita equivalences of super Lie $k$-groupoids are generated by submersive hypercovers, i.e. morphisms $X_{\bt} \to Y_{\bt}$ for which the relative matching maps
 \[
 X_m \to \mathrm{hom}(\pd \Delta^m,X_{\bt})\by_{\mathrm{hom}(\pd \Delta^m,Y_{\bt})}Y_m
 \]
are surjective submersions for all $m \ge 0$ (the $m=0$ case is $X_0 \to Y_0$, and the maps are in fact automatically diffeomorphisms for $m \ge k$). \cite[Proposition \ref{poisson-biinftyFXwell}]{poisson} ensures that these morphisms induce equivalences on spaces of Poisson structures.
 \end{remark}

\subsubsection{Derived Lie supergroupoids}

We now give a flavour of the global constructions of  \cite[\S \ref{poisson-Artindiagramsn}]{poisson} and elsewhere, defining shifted Poisson structures and quantisations on global objects which are both derived and stacky. The main motivating examples tend to be cumbersome to write down explicitly, but   as in \cite[\S \ref{DStein-stacksn}]{DStein} we can apply \cite[Theorem \ref{stacks2-bigthm}]{stacks2} to the model categories of Appendix \ref{equivapp}  to identify the resulting $\infty$-category with that of ($k$, $\pi^0$-surjective homotopy submersion)-geometric stacks in dg supermanifolds  defined inductively by analogy with \cite{hag2}, for homotopy submersions as in footnote \ref{htpysubmersiondef} (after Definition \ref{quasisubdef}).

In particular, every strongly quasi-compact algebraic derived Artin stack $X$ gives rise to such a stack, which  sends each dg manifold $U$ to  $X(\C^{\infty}(U))$. Examples of such include all components of the derived moduli stack of perfect complexes $\Perf$. The canonical $2$-shifted symplectic structure on $\Perf$ is the source of very many constructions, with transgression as in Example \ref{transgressex} giving a  $(2-d)$-shifted symplectic structure on the derived moduli stack $\Perf_X$ of perfect complexes on an NQ manifold $X$ with a suitable $d$-dimensional cohomology theory. For instance, if we take $X$ to be the NQ manifold associated to the tangent Lie algebroid of a manifold $M$, then $\Perf_X$ is the derived stack of perfect complexes on $M$ with flat hyperconnection, and carries a $(2-d)$-shifted symplectic structure when $M$ is compact and orientable of dimension $d$.

\begin{definition}\label{derivedsuperLiegpddef}
 Define a dg super Lie $k$-groupoid to be a simplicial dg supermanifold $X_{\bt}$ (so each $X_m$ is a dg supermanifold in the sense of Definition \ref{derivedsupermfd}) such that
\begin{enumerate}
\item  the partial matching maps
\[
 \pi^0X_m \to \mathrm{hom}(\L^{m,i},\pi^0X_{\bt})
\]
are surjective submersions of super $\C^{\infty}$-spaces for all  $m\ge 1$, $ 0 \le i \le m$, and  diffeomorphisms  for $m >k$; 

\item  the sheaf $\sH_*\O_X$ on $X_{\bt}^{0,=}$ is $\sH_0\sO_X$-Cartesian in the sense that for each of the face maps $\pd_i \co X_m \to X_{m-1}$, we have
\[
 \sH_*\O_{X_m} \cong \sH_0\O_{X_m}\ten_{(\pd_i^{-1}\sH_0\O_{X_{m-1}})}(\pd_i^{-1}\sH_*\O_{X_{m-1}});
\]

\item  the partial matching maps
\[
 X_m \to \mathrm{hom}(\L^{m,i},X_{\bt})
\]
are quasi-submersions (Definition \ref{quasisubdef}) of dg supermanifolds for all $m\ge 1$, $ 0 \le i \le m$, and local  diffeomorphisms for $m \gg 0$. 

\end{enumerate}
\end{definition}

 A restricted class of objects satisfying the conditions of Definition \ref{derivedsuperLiegpddef} is to take a simplicial dg supermanifold $X_{\bt}$ (so each $X_m$ is a dg supermanifold in the sense of Definition \ref{derivedsupermfd}) such that $X_{\bt}^0$ is a super Lie $k$-groupoid and the  sheaf $\O_X$ on $X_{\bt}^{0,=}$ is $\sO_X^0$-Cartesian. Under Lemma \ref{Dstarlemma2}, these give rise to the class of dg NQ manifolds  of Remark \ref{Rstex}. Examples include products of Lie $k$-groupoids with dg supermanfolds, but there are many more involved constructions.

\begin{remarks} 
If we take $\cS$ to be the class of $\pi^0$-surjective homotopy submersions, a  homotopy $(k,\cS)$-hypergroupoid in dg (super)manifolds  in the sense of \cite{stacks2,stacksintro} is a simplicial dg (super)manifold satisfying the first two conditions of  Definition \ref{derivedsuperLiegpddef}.
By \cite[Theorem \ref{stacks2-bigthm} and Examples \ref{stacks2-affhgpddef}]{stacks2}, the $\infty$-category given by simplicial localisation of dg super Lie $k$-groupoids  at homotopy $\cS$-hypercovers is equivalent to that of ($k,\cS$)-geometric stacks in dg (super)manifolds  defined inductively by analogy with \cite{hag2}.
The derived Lie $n$-groupoids of \cite{nuitenThesis} are similar, but  requiring finiteness conditions to hold only up to homotopy, similarly to Remark \ref{htpyfdrmkd1}.
 
The sole purpose of the third condition is to ensure that the formal completion of $X$ along $X_0$ is a dg super NQ-manifold in the sense of Definition \ref{derivedlie}, a definition which is itself unnecessarily restrictive. On the other hand, the  (strict) $(k,\cS)$-hypergroupoids in dg supermanifolds  of \cite{stacks2,stacksintro}  satisfy all of Definition \ref{derivedsuperLiegpddef} except the $m \gg 0$ part of the final condition, because all of their matching (not just partial matching) maps are quasi-submersions. Every homotopy $(k,\cS)$-hypergroupoid $X$ is levelwise quasi-isomorphic to a  $(k,\cS)$-hypergroupoid $\hat{X}$ by Reedy fibrant replacement, and then $\C^{\infty}(N\hat{X})$ below satisfies $\H^*_Q(\Omega^1_{C^{\infty}(N\hat{X})}\ten_{ C^{\infty}(N\hat{X})}\C^{\infty}(\hat{X}_0))=0$, so it is cofibrant in the model structure of Remark \ref{locmodelNQrmk}.

 Precisely, a morphism $f \co X_{\bt} \to Y_{\bt}$ of  dg super Lie $k$-groupoids is a  homotopy $\cS$-hypercover if $\pi^0X_{\bt} \to \pi^0Y_{\bt}$ is a submersive hypercover (Remark \ref{highermoritarmk}) and the maps $\sH_0\O_{X_m}\ten_{(f_m^{-1}\sH_0\O_{Y_{m}})}(f_m^{-1}\sH_*\O_{Y_{m}}) \to \sH_*\O_{Y_m}$ are isomorphisms for all $m$. See \cite[Lemma 6.68]{2021lect} for alternative characterisations.
 \end{remarks}

\begin{definition}
 Given a simplicial dg supermanifold $X_{\bt}$, define the normalisation $NX$ to be manifold $X^{0,=}_0$ equipped with the super chain cochain algebra $D^*( (\sigma^0)^{-\bt}\O_X)$, where $(\sigma^0)^{-\bt}\O_X $ is the cosimplicial sheaf $ ((\sigma^0)^{-\bt}\O_X)^m:= ((\sigma^0)^m)^{-1}\O_{X_m}$ on $X^{=}_0$.
\end{definition}

Adapting the argument for super Lie $k$-groupoids (Lemma \ref{Dstarlemma}) gives:
\begin{lemma}\label{Dstarlemma2}
If $X_{\bt}$ is a dg super Lie $k$-groupoid, then $NX$ is a super dg NQ-manifold, with $\O_{NX}$ generated in cochain degrees $\le k$.
\end{lemma}
\begin{proof}
 The proof of Lemma \ref{Dstarlemma} ensures that the first condition implies that as a bigraded superalgebra $\sO_{NX}$ is locally  freely generated over $\sO_{X,0}^0$ by  finitely many generators (the quasi-submersion conditions ensure finitely many generators in each cochain degree, and the $m \gg 0$ condition ensures finitely many such degrees).
 
 The proof of Lemma \ref{Dstarlemma} also ensures that the second condition implies that as a graded superalgebra, $\sH_0\sO_{NX}$ is locally freely and finitely generated over $\sH_0\sO_X^0$, with generators in cochain degrees up to $k$. The third condition then ensures $\sH_*\sO_{NX}^{\#}\cong (\sH_*\sO_X^0)\ten_{\sH_0\sO_X^0}\sH_0\sO_{NX}$, completing the Artin conditions.
\end{proof}

Definitions \ref{matching} and \ref{PdefArtin} then adapt verbatim to give expressions for shifted Poisson structures and quantisations on dg super Lie $k$-groupoids. As in \cite[Proposition \ref{poisson-biinftyFXwell}]{poisson}, these spaces of depend only on the hypersheafification of the dg super  Lie $k$-groupoid with respect to $\pi^0$-surjective  homotopy submersions (a higher, derived analogue of Morita equivalence).
Theorem \ref{compatthmdgLie}
gives rise to an equivalence between shifted symplectic and non-degenerate Poisson structures on  dg super Lie $k$-groupoids, reasoning as in \cite[Theorem \ref{poisson-Artinthm}]{poisson}. The deformation quantisations of $X$  are significantly stronger than just deformations of derived global sections of functions on $X$; by \cite[Proposition \ref{DQnonneg-Perprop2}]{DQnonneg}, they give rise to deformations of the dg category of perfect complexes on $X$.

The results of \S \ref{quantnsn} all extend via these constructions from dg super NQ-manifolds to give quantisations for shifted Poisson structures on dg super Lie $k$-groupoids. Specifically, quantisation of $n$-shifted Poisson structures for $n>0$ follows immediately from formality of the $E_{n+1}$ operad, and quantisation of $n$-shifted co-isotropic structures for $n>1$ follows from similar equivalences for the Swiss cheese operads as in  \cite{MelaniSafronovII}. For the non-formal cases, 
we have quantisation of $0$-shifted  Poisson structures \cite[Corollary \ref{DQpoisson-Artinquantcor}]{DQpoisson}, of non-degenerate $(-1)$-shifted symplectic structures given a spin structure (a  line bundle with a right $\cD$-module structure on its square) \cite[Proposition \ref{DQvanish-quantpropsd}]{DQvanish}, of $1$-shifted co-isotropic structures \cite[Corollary \ref{DQpoisson-coisocor2}]{DQpoisson} and of $0$-shifted Lagrangians with a spin structure on the Lagrangian \cite[Theorem \ref{DQLag-quantpropsd}]{DQLag}.

\appendix

\section{Notions of equivalence}\label{equivapp}

 Homotopy theory and $\infty$-categories arise simply by specifying a notion of equivalence weaker than isomorphism. Although dg and NQ-manifolds look superficially similar, the respective notions of equivalence are very different. We saw an early instance of this in the behaviour of their de Rham complexes, since dg manifolds satisfy $\DR(X) \simeq \DR(\pi^0X)$, while NQ-manifolds satisfy $\DR(X) \simeq \DR(X_0)$, which tells us that $\DR$ preserves quasi-isomorphisms for dg manifolds, but not for NQ-manifolds. 
 
 Quasi-isomorphism invariance of functions for objects such as dg manifolds  is the  principle underpinning derived geometry. The motivating example is given by the derived vanishing locus $\oR s^{-1}(0)$ of Example \ref{DCritex}, which must behave equivalently to the classical vanishing locus  $s^{-1}(\{0\})$ whenever $s$ intersects the zero section transversely.

 NQ manifolds are generalised quotients of manifolds, whereas dg manifolds are generalised subobjects.
For NQ manifolds, quasi-isomorphism of functions is too crude to give reasonable behaviour, as Examples \ref{NQbadex} below demonstrate. In order to ensure a correspondence of derived categories of vector bundles (Examples \ref{vbundleex}), the the correct notion of equivalence for them is based on tangent quasi-isomorphism (Remarks \ref{weakerequivrmks}),  a common phenomenon in Koszul duality.

\begin{examples}\label{NQbadex}
  If $X$ is the NQ manifold associated to the tangent Lie algebroid of the circle, then  we have $\C^{\infty}(X) \cong \Gamma(S^1, \Omega^{\bt}_{S^1})$ as in Examples \ref{exBg}. The element $d\theta \in \Omega^1_{S^1}$ gives us a morphism $X \to B\g_a$ to the NQ manifold associated to the Lie algebra $\g_a=\R$, corresponding to the morphism $\R \oplus \R d\theta \to \C^{\infty}(X)$ of dg $\C^{\infty}$-rings. That morphism is a quasi-isomorphism between functions of NQ manifolds which we do not wish to regard as equivalent. In particular, their tangent complexes are far from quasi-isomorphic, respectively corresponding at each point to the acyclic complex $(\id \co \R \to \R)$  and to $\R[-1]$. Via Examples \ref{vbundleex}, vector bundles on $X$ correspond to bundles on $S^1$ with connection, or equivalently to $\Z$-representations, whereas vector bundles on $B\g_a$ are $\g_a$-representations; the pullback functor corresponds to $\exp \co \mathfrak{gl}_n(\R) \to \GL_n(\R)$, so is not an equivalence.
 
 The same procedure generalises by replacing $S^1$ with any connected manifold $M$ whose rational homotopy groups $\bigoplus_n \pi_n(M\ten \Q)$ are finite-dimensional. The Sullivan minimal model $\cM$ \cite{Sullivan} of $M$ can then be interpreted as $\C^{\infty}(Y)$ for an NQ manifold $Y$ with $Y_0=\ast$, and $\C^{\infty}(Y) \to \C^{\infty}(X)$ is a quasi-isomorphism despite $Y$ having tangent homology groups $\pi_n(M\ten \Q)\ten_{\Q}\R$ while those of $X$ are $0$.
 
 As another example, consider the the NQ manifold $B\mathfrak{sl}_2$ associated to the Lie algebra $\mathfrak{sl}_2$. This has $\H^*(B\mathfrak{sl}_2)\cong \H^*(\mathfrak{sl}_2,\R) \cong \R \oplus \R [-3]$, so choosing a generator in $ \L^3(\mathfrak{sl}_2)^*$ gives us a morphism $B\mathfrak{sl}_2 \to B^3\g_a$ to the NQ manifold associated to the DGLA $\g_a[2]=\R[2]$, and the resulting map $\H^*\C^{\infty}(B^3\g_a) \to \H^*\C^{\infty}(B\mathfrak{sl}_2)$ is an isomorphism. However, the respective tangent complexes at each point are $\mathfrak{sl}_2[-1]$ and $\R[-3]$, and vector bundles on $B\mathfrak{sl}_2$ and $B^3\g_a$ are representations of the DGLAs $\mathfrak{sl}_2$ and $\R[2]$, respectively.
 \end{examples}

 \subsection{Quasi-isomorphisms of dg supermanifolds}\label{dgmfdequivsn}
 
 Since we are assuming that manifolds $M$ satisfy Whitney's embedding theorem, they behave similarly in some respect to smooth affine algebraic varieties. The global sections functor $\Gamma(M,-)$ is exact on quasi-coherent sheaves of $\sO_M$-modules, by the existence of partitions of unity, so gives an equivalence of categories between quasi-coherent sheaves of $\sO_M$-modules and $\C^{\infty}(M)$-modules. Since $\C^{\infty}$ K\"ahler differentials are defined by adjunction, that also gives a canonical isomorphism $\Omega^1_{\C^{\infty}(X)} \cong \Gamma(X^{0,=} _0, \Omega^1_X)$ for any dg NQ supermanifold $X$, and similarly for tensor powers and $\Hom$s.
 
 \begin{lemma}\label{sheaflemma}
  Given a dg manifold $X$ and an $\O_X$-module $\sF$ in chain complexes of sheaves which is quasi-coherent in the sense that $\sF(V) \cong \sF(U)\ten_{\sO_X(U)}\sO_X(V)$ is an isomorphism for all open immersions $U \to V$ in $X^0$,
  the natural map 
  \[
   \Gamma(X^0, \sF) \to \Gamma(\pi^0X, i^{-1}\sF)
  \]
is a quasi-isomorphism, where $i \co \pi^0X \to X^0$ denotes the canonical closed immersion.

\end{lemma}
\begin{proof}
The quasi-coherence condition for $\sF$  is equivalent to quasi-coherence as a complex of  $\sO_{X,0}$-modules. Since $\Gamma(X^0,-)$ is exact for such modules (and similarly for $\pi^0X$), we have
\begin{align*}
 \H_n \Gamma(X^0, \sF) \cong   \Gamma(X^0, \sH_n\sF) \cong   \Gamma(X^0, i_*i^{-1}\sH_n\sF)\\
 \cong \Gamma(\pi^0X, \sH_n(i^{-1}\sF)) \cong \H_n\Gamma(\pi^0X, i^{-1}\sF),
\end{align*}
where the middle equivalences follow because  $\sH_n\sF$ is an $\sH_0\sO_X$-module, so is supported on $\pi^0X$.
\end{proof}

\begin{proposition}\label{QIMtanprop}
 For a morphism $f \co X \to Y$ of dg supermanifolds, the following conditions are equivalent:
 \begin{enumerate}
  \item $f$ induces homology isomorphisms $\H_*\C^{\infty}(Y) \to \H_*\C^{\infty}(X)$ of supervector spaces;
 
 \item $f$  induces an isomorphism $\pi^0X^= \to \pi^0Y^=$ of topological spaces and homology isomorphisms  $\H_*\sO_{Y,f(x)} \to \H_*\sO_{X,x}$ of stalks for all $x \in \pi^0X^=$;
 
  \item $f$ induces an isomorphism 
  $\H_0\C^{\infty}(Y^=) \to  \H_0\C^{\infty}(X^=)$ 
  of $\C^{\infty}$-rings and homology isomorphisms $\H_*(\Omega^{1}_{\C^{\infty}(Y)}\ten_{\C^{\infty}(Y)}\C^{\infty}(X^0)) \to \H_*(\Omega^{1}_{\C^{\infty}(X)}\ten_{\C^{\infty}(X)}\C^{\infty}(X^0))$.
   
 \item $f$ induces an isomorphism 
  $\H_0\C^{\infty}(Y^=) \to  \H_0\C^{\infty}(X^=)$ 
  of $\C^{\infty}$-rings and homology isomorphisms $\H_*(\Omega^{1}_{\C^{\infty}(Y)}\ten_{\C^{\infty}(Y)}\H_0\C^{\infty}(Y^=)) \to \H_*(\Omega^{1}_{\C^{\infty}(X)}\ten_{\C^{\infty}(X)}\H_0\C^{\infty}(X^=))$, where $ \H_0\C^{\infty}(X^=)$ is the quotient of $\H_0\C^{\infty}(X)$ by the ideal $(\H_0\C^{\infty}(X)^{\ne})$.
  
  \item $f$ induces an isomorphism $\pi^0X^= \to \pi^0Y^=$ of $\C^{\infty}$-spaces and homology isomorphisms $\H_*(\Omega^1_{Y,f(x)}\ten_{\sO_{Y,fx}}\sO_{\pi^0X^=,x}) \to \H_*(\Omega^1_{Y,f(x)}\ten_{\sO_{Y,fx}}\sO_{\pi^0X^=,x})$  of stalks for all $x \in \pi^0X^=$.
  \end{enumerate}
\end{proposition}
\begin{proof}
 The equivalences of (1) with (2) and of (4) with (5) follow immediately from Lemma \ref{sheaflemma} and our assumption that manifolds satisfy Whitney's embedding theorem, so quasi-coherent sheaves of $\sO_{X^0}$-modules are generated by global sections. It thus suffices to show that (1), (3) and (4) are equivalent.

There is a projective model structure on the category $dg_+S\C^{\infty}$ of super dg $\C^{\infty}$-rings $A_{\ge 0}$ concentrated in non-negative chain degrees, in which weak equivalences are quasi-isomorphisms and fibrations are surjective in strictly positive degrees --- see \cite[Theorem 6.10 and Remark 6.13]{CarchediRoytenbergHomological}, or just apply Kan's transfer theorem \cite[Theorem 11.3.2]{hirschhorn} to the forgetful functor to super chain complexes. As in \cite[\S 2.2.2]{nuitenThesis} (incorporating an additional $\Z/2$-grading), following \cite{Q} the functor $A \mapsto \Omega^1_A\ten_AB$ is then a left Quillen functor from the overcategory $dg_+S\C^{\infty}/B$ to the category $dg_+S\Mod(B)$ of $B$-modules in non-negatively graded supercomplexes, with the projective model structure. Its left-derived functor sends $A$ to $(\oL \Omega^1_A)\ten^{\oL}_AB$, for $\oL\Omega^1_A$ the $\C^{\infty}$ cotangent complex, and necessarily preserves quasi-isomorphisms.

Thus (1) implies that  
\begin{align*}
\oL\Omega^1_{\C^{\infty}(Y)}\ten^{\oL}_{\C^{\infty}(Y)}\C^{\infty}(X^0) &\simeq \oL\Omega^1_{\C^{\infty}(X)}\ten^{\oL}_{\C^{\infty}(X)}\H_0\C^{\infty}(X^0),\quad \text{ and}\\
\oL\Omega^1_{\C^{\infty}(Y)}\ten^{\oL}_{\C^{\infty}(Y)}\H_0\C^{\infty}(X^=) &\simeq \oL\Omega^1_{\C^{\infty}(X)}\ten^{\oL}_{\C^{\infty}(X)}\H_0\C^{\infty}(X^=).
 \end{align*}
By \cite[Example 2.2.17]{nuitenThesis},  the cotangent module $\Omega^1_{\C^{\infty}(X^{0,=})}$ is a model for $\oL \Omega^1_{\C^{\infty}(X^{0,=})}$, essentially because $\C^{\infty}(\R^n)$ is cofibrant and $X^{0,=}$ is locally isomorphic to $\R^n$. Now, the morphism $\C^{\infty}(X^{0,=}) \to  \C^{\infty}(X)$ is a cofibration, being freely generated by projective modules, so we have $\oL\Omega^1_{\C^{\infty}(X)/\C^{\infty}(X^{0,=})}\simeq \Omega^1_{\C^{\infty}(X)/\C^{\infty}(X^{0,=})}$, and hence $\oL\Omega^1_{\C^{\infty}(X)}\simeq \Omega^1_{\C^{\infty}(X)}$ by the fundamental exact sequence of cotangent complexes, so (3) and (4) follow. Also note that (3) implies (4), by applying $\ten^{\oL}_{\C^{\infty}(X^0)}\H_0\C^{\infty}(X^=)$

The implication of (1) from (4)  now follows because the morphism $g \co \C^{\infty}(X) \to \H_0\C^{\infty}(X^=)$ and its counterpart for $Y$ are homotopy nilpotent extensions, so lifts along $g$ in the homotopy category $\Ho(dg_+S\C^{\infty})$ are governed by the cotangent complex, as for instance in \cite[Property 3.29, \S 6.4.1 and Lemma 6.30]{2021lect}, including the subsequent comment. 

The argument can loosely be summarised as saying that, for  $I_A$  the ideal $(\H_0A^{\ne}) \subset \H_0A$,  we have a spectral sequence of groups and sets, similar to the sequences  in Appendix \ref{towersn}, with initial page consisting of terms  $\Hom_{\C^{\infty}}(\H_0A/I_A,\H_0B/I_B)$ and $\EExt^i_{\H_0(A)/I_A }(\oL\Omega^1_A\ten^{\oL}_A\H_0(A)/I_A, \H_nB\ten_{\H_0B}(I_B^r/I_B^{r+1}))^=$, and final page the homotopy groups $\pi_n\oR\map_{dg_+S\C^{\infty}}(A,B)$ of the simplicial mapping space. Letting  $B$ be arbitrary, we deduce that  $\oR\map_{dg_+S\C^{\infty}}(\C^{\infty}(Y),-) \to \oR\map_{dg_+S\C^{\infty}}(\C^{\infty}(X),-)$  is a weak equivalence, so $\C^{\infty}(Y) \to \C^{\infty}(X)$ induces an isomorphism in $\Ho(dg_+S\C^{\infty})$.
%
%
%
%
 \end{proof}


 \subsection{Equivalences of dg super NQ-manifolds} 
 
 \subsubsection{Levelwise $\delta$-quasi-isomorphisms and homotopy local diffeomorphisms}

\begin{proposition}\label{QIMtanprop2} 
 For a morphism $f \co X \to Y$ of dg super NQ-manifolds, the following conditions are equivalent:
 \begin{enumerate}
  \item $f$ induces homology isomorphisms $\H_*\C^{\infty}(Y)^i \to \H_*\C^{\infty}(X)^i$ of supervector spaces for all $i$;
 
 \item $f$  induces an isomorphism $\pi^0X_0^= \to \pi^0Y_0^=$ of topological spaces and homology isomorphisms  $\H_*\sO_{Y,f(x)}^i \to \H_*\sO_{X,x}^i$ of stalks for all $x \in \pi^0X^=$ and all $i$;
 
 \item $f$ induces an isomorphism 
  $\H_0\C^{\infty}(Y_0^=) \to  \H_0\C^{\infty}(X_0^=)$ 
  of $\C^{\infty}$-rings and homology isomorphisms $\H_*(\Omega^{1}_{\C^{\infty}(Y)}\ten_{\C^{\infty}(Y)}\H_0\C^{\infty}(Y_0^=))^i \to \H_*(\Omega^{1}_{\C^{\infty}(X)}\ten_{\C^{\infty}(X)}\H_0\C^{\infty}(X_0^=))^i$ for all $i$, where $ \H_0\C^{\infty}(X_0^=)$ is the quotient of $\H_0\C^{\infty}(X_0)$ by the ideal $(\H_0\C^{\infty}(X_0)^{\ne})$.
  
  \item $f$ induces an isomorphism $\pi^0X_0^= \to \pi^0Y_0^=$ of $\C^{\infty}$-spaces and homology isomorphisms $\H_*(\Omega^1_{Y,f(x)}\ten_{\sO_{Y,fx}}\sO_{\pi^0X_0^=,x})^i \to \H_*(\Omega^1_{Y,f(x)}\ten_{\sO_{Y,fx}}\sO_{\pi^0X_0^=,x})^i$  of stalks for all $x \in \pi^0X^=$ and all $i$.
  \end{enumerate}
\end{proposition}
\begin{proof}
 The equivalences of (1) with (2) and of (3) with (4) follow immediately from Lemma \ref{sheaflemma}, exactly as in Proposition \ref{QIMtanprop}.
 
 
 There is a projective model structure on  the category $DG^+dg_+S\C^{\infty}$ of super bidifferential bigraded $\C^{\infty}$-rings $A_{\ge 0}^{\ge 0}$ concentrated in non-negative bidegrees, obtained by applying  Kan's transfer theorem \cite[Theorem 11.3.2]{hirschhorn} to the forgetful functor  to super graded  chain complexes (forgetting $Q$). In this model structure, fibrations are surjections in positive chain degrees, and weak equivalences are levelwise $\delta$-quasi-isomorphisms.
 
 However, cofibration is a very strong condition in the model structure, incorporating the additional  condition $\H^*_Q(\Omega^1_A\ten_AB^0)\cong \H^*_Q(\Omega^1_B\ten_BB^0)$.\footnote{Thus the only dg super NQ-manifold $X$ for which $\C^{\infty}(X)$ is cofibrant is the point $\R^0$. However, cofibrant replacements of $\C^{\infty}(X)$ will always exist in the larger class described in Remark \ref{htpyfdrmkd2}.}
 Since none of our statements involves $Q$, we discard it for the purposes of this proof  and work with  the category $G^+dg_+S\C^{\infty}$ of super chain-differential bigraded $\C^{\infty}$-rings $A_{\ge 0}^{\ge 0}$ concentrated in non-negative bidegrees. This again has a projective  model structure by transfer, with the same  fibrations and weak equivalences.
 It is easy to see that the morphisms $\sC^{\infty}(X)^0 \to \sC^{\infty}(X)^{\#}$ (which do not exist in  $DG^+dg_+S\C^{\infty}$) are cofibrations in this model structure, so the relative K\"ahler differentials $\Omega^1_{\sC^{\infty}(X)^{\#}/\sC^{\infty}(X)^0}$ are a model for the relative cotangent complex. It is also immediate that any cofibrant object of $dg_+S\C^{\infty}$ is cofibrant in $G^+dg_+S\C^{\infty}$, so cotangent complexes in $G^+dg_+S\C^{\infty}$ agree with those of $dg_+S\C^{\infty}$ for objects concentrated cochain degree $0$; thus $\Omega^1_{\sC^{\infty}(X)^0}$ models the cotangent complex of $\sC^{\infty}(X)^0$, by the proof of  Proposition \ref{QIMtanprop}. The fundamental exact sequence of cotangent complexes then implies that $\Omega^1_{\C^{\infty}(X)}$ models the cotangent complex of  $\C^{\infty}(X)$ in $G^+dg_+S\C^{\infty}$.\footnote{Incidentally, the forgetful functor from  $DG^+dg_+S\C^{\infty}$ to $G^+dg_+S\C^{\infty}$ is left Quillen, so preserves cotangent complexes, implying that this also models the cotangent complex in the former.}

 The rest of the proof now follows exactly as in  Proposition \ref{QIMtanprop} once we observe that  the kernel of $A \to A^0$ is a complete ideal   for  all $A \in G^+dg_+S\C^{\infty}$.
 \end{proof}

The following is an immediate consequence.
\begin{corollary}\label{locdiffeocor}
 If  a morphism $f \co X \to Y$ of dg super NQ-manifolds is a levelwise $\delta$-quasi-isomorphism in the sense that it induces quasi-isomorphisms $\C^{\infty}(Y)^i \to \C^{\infty}(X)^i$ for all $i$, then it is a homotopy local diffeomorphism in the sense of Definition \ref{locdiffeodef}. 
\end{corollary}

\begin{remarks}[Weaker equivalences]\label{weakerequivrmks}
Following the philosophy laid out at the start of this appendix, it is essential that derived geometrical constructions send levelwise $\delta$-quasi-isomorphisms to weak equivalences, which is why so many constructions involve the sum-product total complex $\hatTot$ of Definition \ref{hatTotdef}. The natural notion of equivalence for dg super NQ-manifolds is weaker still: quasi-isomorphism on $X_0$ combined with homotopy local diffeomorphism, which we saw in Proposition \ref{hfetlemma} preserves all of our constructions of interest.\footnote{The author has not checked whether  taking the right Bousfield localisation \cite[Theorem 5.1.1]{hirschhorn} of  $DG^+dg_+S\C^{\infty}$ with respect to some generalisation of  dg super NQ-manifolds  will give a model structure for which weak equivalences of cofibrant objects are precisely the homotopy local diffeomorphisms. Since the theory of $\infty$-categories is modelled by relative categories (i.e. categories equipped with a class of weak equivalences), we  have a well-defined $\infty$-category of  dg super NQ-manifolds (or rather their generalisation as in Remarks \ref{htpyfdrmkd2}) up to homotopy local diffeomorphism regardless of whether such a model structure exists.} 

There are very many natural constructions which send homotopy local diffeomorphisms $f \co X \to Y$ with $f_0\co X_0 \to Y_0$ a quasi-isomorphism   to weak equivalences. For any functor  $F$ on $dg_+S\C^{\infty}$, the construction $D_*F$ of \cite[\S 4.1.3]{smallet2} gives an extension to a functor on $DG^+dg_+S\C^{\infty}$.  When $F$ is a simplicial set-valued functor with   a cotangent complex $\bL^F$,  the standard obstruction  method  of  \cite[\S 6.4.1]{2021lect}, applied to the filtration by powers of $(A^{>0})$, gives us a  spectral sequence of groups and sets with initial page $F(A^0)$ and $\EExt_{A^0}(\bL^F_{A^0}, \Tot^{\sqcap}((A^{>0})^n/(A^{>0})^{n+1}))$, converging  to $D_*F(A)$. For $f$ as above, the proof of Proposition \ref{hfetlemma} gives quasi-isomorphisms $\Tot^{\sqcap}(((\C^{\infty}(Y)^{>0})^n/((\C^{\infty}(Y)^{>0})^{n+1}) \to \Tot^{\sqcap}(((\C^{\infty}(X)^{>0})^n/((\C^{\infty}(X)^{>0})^{n+1})$, so $D_*F(\C^{\infty}(Y)) \to D_*F(\C^{\infty}(X))$ is a weak equivalence. 

In particular, this applies when $F$ is the functor associated to any dg super Lie $k$-groupoid. Every algebraic derived Artin $k$-stack gives rise to such, regarding it as  a functor on dg $\C^{\infty}$-rings by forgetting the $\C^{\infty}$ structure. 

\end{remarks}

\begin{examples}\label{vbundleex}
For application to Remarks \ref{weakerequivrmks}, examples of derived Artin stacks of particular interest are the moduli stacks $B\GL_n$ and $\Perf$ of rank $n$ vector bundles and perfect complexes, respectively. As in  \cite[Examples \ref{smallet2-BGLnex} and \ref{smallet2-perex}]{smallet2}, objects of  $D_*B\GL_n(A)$ and  $D_*\Perf(A)$ correspond to   $A$-modules $M^{\ge 0}_{\bt}$ in chain cochain complexes  for which $M^0\ten^{\oL}_{A^0}\H_0A$ is quasi-isomorphic to a projective rank $n$ $\H_0A$-module (resp. perfect $\H_0A$-complex) and the maps $M^0\ten_{A^0}^{\oL}A^i \to M^i$ are all quasi-isomorphisms. When $A = \C^{\infty}(X)$ for an NQ manifold $X$,   the former correspond simply to vector bundles on  $X_0$ with flat $Q$-connection.  
\end{examples}

\subsubsection{Filtered Q-supermanifolds}\label{filtsn}

A homotopically equivalent  alternative to  dg NQ supermanifolds is given by filtered Q-supermanifolds. We now summarise the main features of that approach, which tends not to be followed in most of the literature.

\smallskip
\paragraph{\it Filtrations $\sim$ bigradings.}
As for instance in \cite[Lemma 1.5]{DQLag}, the $\infty$-category of chain cochain complexes localised at levelwise chain quasi-isomorphisms is equivalent to the $\infty$-category of complete exhaustively filtered chain complexes localised at filtered quasi-isomorphisms. The functor sends a double complex $V^{\bt}_{\bt}$ to the sum-product total cochain complex $\hatTot V$ of Definition \ref{hatTotdef}, equipped with the filtration $\Fil_{\Tot}^p\hatTot V:= \Tot^{\sqcap}(V^{\ge p})$. That functor is monoidal under $\ten$, as is its homotopy inverse functor $\g\fr$, given by setting the chain complex  $\g\fr^i(U,F)_{[i]}$ to be the cone of $F^{i+1}U \to F^iU$ reindexed as a chain complex, with the obvious composite $\g\fr^i(U,F) \to (F^{i+1}U)_{[-i-1]} \to   \g\fr^{i+1}(U,F)$ defining the cochain differential $Q$. 

Note that there are natural quasi-isomorphisms $\g\fr^i(U,F)_{[i]} \onto \gr_F^iU$ for all $i$, and 
isomorphisms $V^i_{\bt} \cong \gr_{\Fil_{\Tot}}^i(\hatTot V)[i]$ for all chain cochain complexes $V$.

\smallskip
\paragraph{\it Filtered $Q$-supermanifolds $\sim$ homotopy dg super NQ-manifolds.}
That pair of functors then gives rise to an $\infty$-equivalence on the associated categories of $\C^{\infty}$-algebras. One side of the equivalence is the localisation of $DG^+dg_+S\C^{\infty}$ at levelwise $\delta$-quasi-isomorphisms. The other side is the category of completely filtered super  dg $\C^{\infty}$-rings $(A^{\bt},F)$ with $F^0A=A$ and  $ \H^j(\gr_{F}^iA)=0$ for $j>i$, localised at filtered quasi-isomorphisms,\footnote{The forgetful functor $(A,F) \leadsto \{F^nA\}_n$ to graded chain complexes should yield a projective model structure via Kan's transfer theorem, but this argument has no need for such.} or equivalently  at maps which induce quasi-isomorphisms on $\gr_{F}^pA$ for all $A$. Here, the filtration is required to satisfy $(F^pA)(F^qA) \subset F^{p+q}A$.

Under this correspondence, dg super NQ-manifolds $X$ give rise to  complete filtered super  dg $\C^{\infty}$-rings $(A, F)$ with $F^0A=A$, $A^i \subset F^iA$, such that $\gr_{F}^0A^0$ is isomorphic to $\C^{\infty}(X_0^0)$ for some supermanifold $X_0^0$, with  the algebra $(\gr_{F}^{\#}A)_{\#}$ freely generated over $\gr_{F}^0A_{\#}$ by a finite projective bigraded supermodule $M$. 

The analogue for $(A,F)$ of the underived truncation $\pi^0X$ is then given by $(\bigoplus_{p \ge 0} \H^p(\gr_{F}^pA)[-p], Q_1)$, where $Q_1$ is the differential on the first page of the spectral sequence associated to the filtration $F$. Similarly, the analogue of the $\sO_{\pi^0X}$-module  $\sH_n\sO_X$  is  $(\bigoplus_{p \ge 0} \H^{p-n}(\gr_{F}^pA)[-p], Q_1)$. We thus impose  Artin conditions in these terms, requiring  that the graded-commutative algebra $\bigoplus_p \H^p(\gr_{F}^{p}A)[-p]$ be freely generated over $\H^0(\gr_{F}^0A)$, with multiplication $    \H^p(\gr_{F}^{p}A)\ten_{\H^0(\gr_{F}^0A)}\H^i(\gr_{F}^0A) \to \H^{p+i}( \gr_{F}^{p}A)$ being an isomorphism for all $p,i$.

The generalised dg super NQ-manifolds of Remarks \ref{htpyfdrmkd2} precisely correspond under this $\infty$-equivalence to filtered super  dg $\C^{\infty}$-rings generated similarly, but with $M$ allowed to be infinitely generated and $\gr_{F}^0A$ just satisfying the conditions of  Remark \ref{htpyfdrmkd1}, and  the Artin conditions incorporating finite generation of $\bigoplus_p\H^p(\gr_{F}^{p}A)[-p]$  over $\H^0(\gr_{F}^0A)$.

 \smallskip
\paragraph{\it Cotangent modules and local diffeomorphisms.}
The relevant cotangent module of $(A,F)$,  corresponding to the construction $\C^{\infty}(X) \leadsto  (\hatTot \Omega^1_{\C^{\infty}(X)}, \Fil_{\Tot})$ is then given by the completion $\hat{\Omega}^1_A$ of $\Omega^1_A$ with respect to the induced filtration $F$. Tensor products also have to be completed with respect to the filtration, and we write $M\hten_A N$ for the completion of $M\ten_A N$

While levelwise $\delta$-quasi-isomorphism corresponds to filtered quasi-isomorphism under the equivalence above, homotopy local diffeomorphism in the sense of Definition \ref{locdiffeodef} corresponds to  morphisms $(A,F) \to (B,G)$ which induce a
quasi-isomorphism 
$(\hat{\Omega}^1_A\hten_A\gr_G^0B)/F^N \to (\hat{\Omega}^1_B\hten_B\gr_G^0B)/G^N$ for $N \gg 0$. The weaker equivalences of Remarks \ref{weakerequivrmks} combine this condition with $\gr_F^0A \to \gr_G^0B$ being a quasi-isomorphism.

The Artin conditions can be rewritten as saying that 
 $\H^j(\hat{\Omega}^1_{\gr_{\Fil}A}\hten_{ \gr_{\Fil}A} \H^0(\gr_{\Fil}^0))$ should be a 
 projective $\H^0(\gr_{\Fil}^0A)$-module of finite rank  concentrated in filtration weight $j$ for $j>0$, with 
  only finitely many of these  being non-zero, and that it should be concentrated in filtration weight $0$ for $j\le 0$,


 \smallskip
\paragraph{\it Hom constructions.}
For $\Hom$-complexes,  the situation is a little more subtle because for complete exhaustively filtered $A$-modules  $M,N$, the obvious filtration on the complex complex $\HHom_A(M,N)$ is not exhaustive. $\Fil^i\HHom_A(M,N)$ consists of morphisms $f \co M \to N$ satisfying $f(\Fil^jM) \subset \Fil^{i+j}N)$ for all $j$, and we let $\check{\HHom}_A(M,N):= \LLim_{i<0} \Lim_{j>0}( \Fil^i\HHom_A(M,N)/ \Fil^{i+j})$, with the filtration inherited from $\HHom$. When $\Fil^{<0}M=0$, this simplifies to $\Lim_{i<0} \HHom_A(M,N)$. 

In particular, the relevant tangent module,  corresponding to the construction $\C^{\infty}(X) \leadsto  (\hatTot T_{\C^{\infty}(X)}, \Fil_{\Tot})$, is given by $\check{\HHom}_A(\hat{\Omega}^1_A,A)$ with the filtration $\Fil$ as above.

\smallskip
\paragraph{\it Poisson structures.}
One aspect which the filtration perspective does elucidate is the data of shifted Poisson structures and their deformation quantisations. An $n$-shifted Poisson structure on a dg NQ-supermanifold $X$ (or on a generalisation as in Remarks \ref{htpyfdrmkd2}) is given by a $P_{k+1, \infty}$-algebra structure on $\hatTot \sO_X$ for which all of the operations are filtration-bounded below, i.e. lie in  $\check{\HHom}$. That definition and the proofs extend intrinsically to filtered $Q$-manifolds as above. The descriptions of $BD_n$-quantisations are similar, but  with some curvature terms (and twisting by a line bundle for $n=-1$).

\subsection{Alternative model structures}

As we have already seen in the proof of Proposition \ref{QIMtanprop}, there is a projective model structure on the  category $dg_+S\C^{\infty}$ of super dg $\C^{\infty}$-rings $A_{\ge 0}$ concentrated in non-negative chain degrees, in which weak equivalences are quasi-isomorphisms and fibrations are surjective in strictly positive degrees. We saw in that proof that for dg supermanifolds $X$, the cotangent complex can be calculated as  $\Omega^1_{\C^{\infty}(X)}$, even though $\C^{\infty}(X)$ is seldom cofibrant. (Specifically, it is cofibrant only if the supermanifold $X^0$ is a $\C^{\infty}$ retract of $\R^{n|m}$ for some $n,m$.) 

We now introduce two model structures Quillen equivalent to the projective model structure, but with more cofibrant objects, so that the algebras associated to dg manifolds are cofibrant in both. 

The following is the special case of \cite[Definition \ref{DStein-locdef}]{DStein} for $\oT$ the Fermat theory EFC, and $\oE$ the class of  morphisms of finitely presented $\C^{\infty}$-algebras corresponding to open immersions of finitely presented affine $\C^{\infty}$-schemes, as in \cite[Examples \ref{DStein-openetaleexamples}]{DStein}.

\begin{definition}\label{locdef}
 Given a morphism $A \to B$ of finitely presented $\C^{\infty}$-algebras, with associated morphism $f \co Y \to X$ of affine $\C^{\infty}$-schemes, 
 define the localisation $(A/B)^{\loc}$ of $A$ along $B$ to be the $\C^{\infty}$-algebra $\Gamma(Y, f^{-1}\sO_X)$, i.e. $\LLim_U \C^{\infty}(U)$ where $U$ ranges over all open affine $\C^{\infty}$-subschemes of $X$ containing $f(Y)$.

 We then define $(A/B)^{\loc}$ for arbitrary morphisms of $\C^{\infty}$-algebras by passing to filtered colimits.
 \end{definition}

 \begin{definition}\label{locmorphismdef}
  Say that a morphism $A \to B$ of $\C^{\infty}$-algebras is \emph{local} if the canonical map $A \to (A/B)^{\loc}$ is an isomorphism.\footnote{We can characterise this purely in terms of algebraic operations. Since open immersions are generated by the inclusion of an open interval in $\R$, or equivalently by the morphism $\arctan \co \R \to \R$, a morphism $f \co A \to B$ is  local if and only if $\tan(a)$ exists in $A$ whenever $\tan(f(a))$ exists in $B$ (more precisely, there exists $t \in A$ with $\arctan(t)=A$ whenever there exists $s \in B$ with $\arctan(s)=f(a)$).}
 \end{definition}

The following is then \cite[Proposition \ref{DStein-locmodelprop}]{DStein} restricted to this setting. The same proof works for superalgebras, where we can replace $\H_0A$ with $\H_0A/(\H_0A^{\ne})$ because nilpotent extensions are automatically local.
\begin{proposition}\label{locmodelprop}
 There is a cofibrantly generated model structure on the full subcategory $dg_+\C^{\infty}_{\loc} \subset dg_+\C^{\infty}_{\loc}$  consisting of objects  $A$ for which  $A_0 \to \H_0A$ is local, in which weak equivalences are quasi-isomorphisms and
 fibrations are surjective in strictly positive chain degrees.  The inclusion  functor to  $dg_+\C^{\infty}$ with its projective model structure  is  a right Quillen equivalence.
 
 In this model structure, every open submanifold $U \subset \R^n$ gives a cofibrant object $\C^{\infty}(U)$.
\end{proposition}

For the closed immersion $i \co \pi^0X \to X^0$, we then have the following.
\begin{corollary}\label{locmodelcofcor}
 For any dg manifold $X$, the left Quillen equivalence $dg_+\C^{\infty} \to dg_+\C^{\infty}_{\loc}$ sends $\C^{\infty}(X)$ to the cofibrant object $\Gamma(\pi^0X,i^{-1}\sO_X)$.
\end{corollary}
\begin{proof}
The left Quillen equivalence  sends $A$ to $A^l:=A\ten_{A_0}(A_0/\H_0A)^{\loc}$; the description of $\C^{\infty}(X)^l$ follows by substitution. To see that $\C^{\infty}(X)^l$ is cofibrant, observe that since $\C^{\infty}(X)_0 \to \C^{\infty}(X)$ is a cofibration and $(-)^l$ left Quillen, it suffices to show that $\C^{\infty}(X)_0^l$ is cofibrant in $dg_+\C^{\infty}_{\loc}$, but $\C^{\infty}(X)_0^l=\C^{\infty}(X^0)$. Now,  choose a closed embedding $X^0 \into \R^n$, and extend it to an open immersion $NX^0 \to \R^n$ of the normal bundle. By Proposition \ref{locmodelprop}, $\C^{\infty}(NX^0)$ is cofibrant in  $dg_+\C^{\infty}_{\loc}$, but $X^0$ is a retract of $NX^0$, so $\C^{\infty}(X^0)$ must also be cofibrant.
\end{proof}

In the setting of \cite[Proposition \ref{DStein-locmodelprop}]{DStein}, we have the following variant of the local model structure.
Instead of taking a smaller category to increase the number of cofibrations without changing the weak equivalences, we now reduce the number of fibrations instead. 
\begin{proposition}\label{locmodelprop2}
  There is a   model structure  on the category $dg_+\oT$ of non-negatively graded $\oT$-DGAs  in which  weak equivalences are quasi-isomorphisms and
 fibrations $A \to B$ are those projective fibrations (i.e.  surjections in strictly positive chain degrees) $A \to B$ for which  $(\dagger)$ the map $A_0 \to (A_0/(B_0 \by_{\H_0B} \H_0A) )^{\loc_{\oE}}$ is an isomorphism. 
 All $\oE$-morphisms of $\oT$-algebras are cofibrations in this model structure.

 The identity  functor  is then a right Quillen equivalence to the projective model structure, and the functor $ A \mapsto A^{l_{\oE}}$ is a left Quillen equivalence to the model structure of \cite[Proposition \ref{DStein-locmodelprop}]{DStein}.
\end{proposition}
\begin{proof}
Defining cofibrations to be morphisms with the left lifting property (LLP) with respect to trivial fibrations, it suffices to prove the factorisation axioms and that trivial cofibrations have LLP with respect to fibrations.
\begin{enumerate}
\item\label{Ecof}
We begin  by showing that morphisms in $\oE$ are all cofibrations. If $f \co A \to B$ is a trivial fibration, then $\H_0A \cong \H_0B$, so  $(\dagger)$ says that $A_0 \to (A_0/B_0)^{\loc_{\oE}}$ is an isomorphism.  By \cite[Lemma \ref{DStein-UFSlemma}]{DStein}, $f_0$ (and hence $f$) thus has  the unique RLP with respect to all $\oE$-morphisms.

\item\label{TETcof} Now, if $i \co C \to D$ is quasi-isomorphism in $dg_+\oT$, then it has  LLP with respect to a morphism $f \co A \to B$ whenever it has LLP with respect to $A \by B\by_{\H_0B}\H_0A$. If $i$ is moreover 
 a pushout of an $\oE$-morphism (so in particular $D \cong C\ten_{C_0}D_0$), then $i$ has LLP with respect to $f$ whenever  $C_0 \to D_0$ has LLP with respect to $A_0 \to B_0\by_{\H_0B}\H_0A$. If $f$ is a fibration, then  $(\dagger)$ and  \cite[Lemma \ref{DStein-UFSlemma}]{DStein} again imply such a lift. 

\item\label{dgshrink2} The proof of \cite[Proposition \ref{DStein-dgshrink}]{DStein} shows  that if we have $A \in dg_+\oT$ and an ideal $I \subset \delta A_1$ in $A_0$,  then for $D:= A_0/I$ the morphism $A \to A\ten_{A_0}(A_0/D)^{\loc}$ is a quasi-isomorphism.

\item\label{lambdaprops} Given a projective fibration $f \co A \to B$, we can take the factorisation $A \xra{\lambda} \phi(A,B):= A\ten_{A_0}(A_0/(B_0 \by_{\H_0B} \H_0A) )^{\loc_{\oE}} \xra{\mu} B$. Since $f$ is a projective fibration, the morphism $A \to B\by_{\H_0B}\H_0A$ is surjective. Then $\lambda$ is a quasi-isomorphism by (\ref{dgshrink2}) and $\mu$ satisfies $(\dagger)$ by \cite[Lemma \ref{DStein-UFSlemma}]{DStein}, so is a fibration. Moreover, since $\lambda$ is constructed a pushout of $\oE$-morphisms, it is  a cofibration, by (\ref{Ecof}), and has LLP with respect to all fibrations, by (\ref{TETcof}). 

\item\label{factn} Factorisation is now straightforward. Given a morphism $A \to B$ in  $dg_+\oT$, factorise in the projective model structure as a composite $A \xra{\alpha} C \xra{\beta} B$ of  a projective cofibration and a projective fibration, with one of the maps being a quasi-isomorphism. Then the factorisation $A \xra{\lambda \circ \alpha} \phi(C,B) \xra{\mu} B$ suffices because $\alpha$ is a cofibration \emph{a fortiori},  $\lambda$  a trivial cofibration and  $\mu$  a fibration.

\item It remains to show that trivial cofibrations $i \co A \to B$ have LLP with respect to fibrations.  Applying the factorisation procedure of (\ref{factn}) to $i$, the morphism $\lambda$ satisfies LLP by (\ref{lambdaprops}), so $\lambda \circ \alpha$ also does, since   projective trivial cofibrations satisfy LLP  \emph{a fortiori}. Because $i$ is a cofibration and $\mu$ a trivial fibration, we have a section of $\mu$ under $A$. Thus  $i$ is a retract of $\lambda \circ \alpha$, so also satisfies LLP with respect to all fibrations. 
\end{enumerate}
Finally, observe that the fibrations of Proposition \ref{locmodelprop} are all fibrations in this model structure, so the inclusion functor is right Quillen. similarly, the identity functor sends our fibrations to projective fibrations, so is also a right Quillen functor between the respective model structures. These are both Quillen equivalences because weak equivalences are preserved by these functors and their left adjoints.
 \end{proof}

 \begin{corollary}\label{locmodelfpcor}
  There is a model category whose full subcategory of finitely cogenerated fibrant objects is equivalent to the category of dg manifolds, with weak equivalences corresponding to quasi-isomorphisms.  Fibrations between these are precisely the quasi-submersions of Definition \ref{quasisubdef}.
 \end{corollary}
\begin{proof}
 This is just given by the  category $(dg_+\C^{\infty})^{\op}$ equipped with the (opposite of) the model structure of Proposition \ref{locmodelprop2}, with $\oT=\C^{\infty}$ and $\oE$ the class of open immersions. A finitely generated cofibrant object $A$ in that model structure is a retract of a $\C^{\infty}$-DGA finitely freely graded as a graded algebra over $\C^{\infty}(U)$ for some open submanifold $U \subset \R^n$; this is precisely equivalent to saying that $A \cong \C^{\infty}(X)$ for some dg manifold $X$. The characterisation of fibrations follows similarly.
\end{proof}

 \begin{example}\label{pathex}
  Corollary \ref{locmodelfpcor} implies that for each dg manifold $X$, there must exist path objects $PX$ equipped with a quasi-isomorphism $X \to PX$ and a fibration $PX \to X$. Since $X$ is finitely cogenerated as a fibrant object, this can also be chosen to be finitely cogenerated, and hence a dg manifold,  although that constraint cannot be imposed functorially.
  
  For an explicit construction, start with the case where $X$ is a manifold. We can consider the diagonal embedding $X \to X \by X$, and choose an extension of this to an open immersion $j_X \co  NX \into  X \by X$ of the normal bundle. We then take $PX$ to be the derived vanishing locus (Definition \ref{DCritex}) of the diagonal section $NX \to NX\by_XNX$, regarded as a bundle via projection to the first factor. Since this section $y \mapsto (y,y)$ intersects the zero section $y \mapsto (y,0)$ transversely, the inclusion $X \to PX$ is indeed a quasi-isomorphism.  The quasi-submersion $PX \to X \by X$ is the composite map $PX \to NX \xra{j_X} X \by X$.  
  
  The extension of this construction  to arbitrary dg manifolds is then just an algebraic exercise,   by factorising the morphism  $j_{X^0}^{-1}\sO_{X \by X} \to s_{0*} \sO_X$ of sheaves of dgas  on $NX^0$ (for $s_0 \co X^0 \to NX^0$ the zero section) as a quasi-free morphism followed by a quasi-isomorphism by  freely adding boundary generators in positive degrees, shrinking $j_X$ if necessary.
 \end{example}

 \begin{remark}
  Note that the adjunction defining the cotangent complex as in the proof of Proposition \ref{QIMtanprop} also gives a Quillen adjunction for the model structures of Propositions \ref{locmodelprop} and \ref{locmodelprop2}
(cf. \cite[Lemma \ref{DStein-cotlemma}]{DStein}),
 so as in \cite[\S \ref{DStein-cotsn}]{DStein} they can be used directly to show that $\Omega^1_{\C^{\infty}(X)}$ and $\Gamma(\pi^0X, \Omega^1_X)$ model the $\C^{\infty}$ cotangent complex of a dg manifold $X$, since $\C^{\infty}(X)$ and $\Gamma(\pi^0X, i^{-1}\sO_X)$ are cofibrant in the respective model structures.
   \end{remark}

\begin{remark}\label{locmodelNQrmk}
We can also apply   Kan's transfer theorem \cite[Theorem 11.3.2]{hirschhorn} to the forgetful functors from  $DG^+dg_+\C^{\infty}$ and $G^+dg_+\C^{\infty}$ (i.e. forgetting $Q$)  to $dg_+\C^{\infty}\by \prod_{i>0} dg_+\Vect_{\R}$, sending $A$ to $(A^0, A^1,A^2, \ldots)$, where $dg_+\C^{\infty}$ is given the model structure of Proposition \ref{locmodelprop2}, and $dg_+\Vect_{\R}$ the projective model structure. 

The  $G^+dg_+\C^{\infty}$-valued left adjoint sends $(A,V^1,V^2,\ldots)$ to the algebra $A\ten_{\R}\Symm_{\R}(\bigoplus_{i>0} (V^i)^{[-i]}$, while the $DG^+dg_+\C^{\infty}$-valued left adjoint is given by the de Rham algebra of that, i.e.  $\Omega^{\bt}_A\ten_{\R}\Symm_{\R}(\bigoplus_{i>0} (V^i)^{[-i]} \oplus (V^i)^{[-i-1]})$, with $Q$ given by the de Rham differential on $\Omega^{\bt}_A$ and by the identity map from $ (V^i)^{[-i]}$ to $(V^i)^{[-i-1]}) $. 

This gives model structures in which weak equivalences are levelwise $\delta$-quasi-isomorphisms and  fibrations $A \to B$ are surjective in positive chain degrees, with $A^0_0 \to B^0_0\by_{\H_0B^0} \H_0A^0$ being local. 
The cofibrations of $G^+dg_+\C^{\infty}$ are generated by  quasi-submersions as in  Definition \ref{quasisubdef} (with $Q=0$), so in particular any dg N-manifold $X$ (i.e. turning off $Q$) gives a cofibrant object $\C^{\infty}(X)^{\#}_{\bt}$  in this model structure.

The cofibrations in the model structure on $DG^+dg_+\C^{\infty}$ are cofibrations in $G^+dg_+\C^{\infty}$, 
but with the additional restriction of $Q$-acyclicity of the relative cotangent module, which prevents finite generation.
 \end{remark}

\section{Obstruction towers associated to filtered DGLAs}\label{towersn}

Our constructions of spaces of shifted Poisson structures and their quantisations are formulated in terms of Maurer--Cartan functors on pro-nilpotent differential graded Lie algebras. We now summarise the standard obstruction theory of such DGLAs, which underpins most of the proofs in \S\S \ref{poisssn}--\ref{descentsn}. 

\subsection{Obstructions}

The following is well-known (see e.g. \cite[Lemma 3.3]{dmc} for an explicit proof, following \cite[\S 3]{Man}). 
\begin{lemma}\label{obsdgla}
If a surjection $e:L \onto M$ of DGLAs has abelian kernel $K$,  then for any $\omega \in \mc(M)$, the obstruction to lifting $\omega$ to $\mc(L)$ lies in
$
 \H^2(K_{\omega}). 
$
\end{lemma}
Here,  we write $K_{\omega}$ for the cochain complex $(K, d+[\omega,-])$, noting that because $K$ is abelian, the Lie bracket action of $L$ on $K$ descends to $M$. 

In particular, the lemma  implies that if $\H^*K_{\omega}=0$, then $\omega$ must lift to $\mc(L)$, and if $K$ has an $L$-equivariant contracting homotopy (for instance if $K$ is acyclic and $[K,L]=0$), then  the map $\mc(L) \to \mc(M)$ is surjective. 

Applying this to (partial) matching maps immediately gives the following.
\begin{lemma}
 If a surjection $e:L \onto M$ of DGLAs has abelian kernel $K$,  then the map $\mmc(L) \to \mmc(M)$ is a Kan fibration. Moreover, if $K$  has an $M$-equivariant contracting homotopy, then $\mmc(L) \to \mmc(M)$ is a weak equivalence.
 \end{lemma}

 We now have the following standard statement of DGLA obstruction theory, generalising
 \cite[proof of Theorem 3.1, step 3]{Man2}.
 
\begin{proposition}\label{obsDGLA}
Take  surjections $ L \xra{e} M \xra{q} \bar{M}$   of DGLAs with $K = \ker(e)$, such that  $[K,\ker(q)]=0$. There is then a  map $o_e \co \mmc(M) \to \mmc(\bar{M} \oplus K^{[1]})$ in the homotopy category of simplicial sets, and a weak equivalence between $\mmc(L)$ and the homotopy fibre product given by the homotopy limit of the diagram
\[
  \mmc(M) \xra{o_e} \mmc(\bar{M} \oplus K^{[1]}) \xla{(\id,0)}\mmc(\bar{M}).
\]
\end{proposition}
\begin{proof}
Let $\tilde{M}:= \cone(K \to L)$, with Lie bracket $[(\kappa, \lambda), (\kappa', \lambda')]= ([\kappa, \lambda']+[\lambda,\kappa'], [\lambda,\lambda'])$ for $\kappa, \kappa'\in K$, $\lambda, \lambda \in L$.  Observe that $\ker(\tilde{M}\to M)\cong \cone(\id \co K \to K)$ has an $\tilde{M}$-equivariant contracting homotopy, so $\tau\co \co  \mmc(\tilde{M}) \to \mmc(M)$ is a trivial fibration. 

Consider the map $u=(\id, q \circ e) \co \tilde{M} \to \cone(K \to \bar{M}) \cong \bar{M} \oplus K^{[1]}$.  This is surjective, and the target is a DGLA because $[K,\ker(q)]=0$. We have $\ker(u)=K$, which is abelian,  so $u$ induces a Kan fibration $\mmc(\tilde{M}) \to \mmc( \bar{M} \oplus K^{[1]})$. We define $o_e$ to be the composite 
\[
 \mmc(M) \xra{\tau^{-1}} \mmc(\tilde{M}) \xra{u} \mmc( \bar{M} \oplus K^{[1]})
\]
in the homotopy category of simplicial sets

Since $u$ is a Kan fibration, 
\begin{align*}
 \mmc(\tilde{M})\by^h_{u,  \mmc(\bar{M} \oplus K^{[1]}), (\id,0)}\mmc(\bar{M}) &\simeq  \mmc(\tilde{M})\by_{u,  \mmc(\bar{M} \oplus K^{[1]}), (\id,0)}\mmc(\bar{M})\\
 &\cong \mmc( \tilde{M}\by_{u,(\bar{M} \oplus K^{[1]}), (\id,0)} \bar{M})\\
 &\cong \mmc(L),
\end{align*}
as required.
 \end{proof}

 As in \cite[Example 2.37]{2021lect}, a double application of the long exact sequence of  \cite[Lemma I.7.3]{sht} then gives the following.
\begin{corollary}\label{obsDGLAcor}
In the setting of Proposition \ref{obsDGLA}, we have a canonical long exact sequence
\[
 \xymatrix@R=0ex{
\cdots  \ar[r]^-{e_*}&\pi_i(\mmc(M), \omega) \ar[r]^-{o_e}& \H^{2-i}(K_{\omega}) \ar[r] &\pi_{i-1}(\mmc(L), \tilde{\omega})\ar[r]^-{e_*}&\cdots\\ 
\cdots \ar[r]^-{e_*}&\pi_1(\mmc(M), \omega) \ar[r]^-{o_e}& \H^1(K_{\omega})  \ar[r] &\pi_0\mmc(L) \ar[r]^-{e_*}& \pi_0\mmc(M) \ar[r]^-{o_e}& \H^{2}(K_{?})
}
\]
 of homotopy groups and sets, for $\tilde{\omega} \in \mc(L)$ lifting $\omega \in \mc(M)$. 
\end{corollary}
 Explicitly, this means that
 \begin{itemize}
    \item a class $[\omega] \in \pi_0\mmc(M) $ lifts to  a class
    $[\tilde{\omega}] \in \pi_0\mc(L)$ if and only if $o_e([\omega])=0 \in \H^2(K_{\omega})$;
    \item the group  $  \H^1(K_{\omega}) $ acts transitively on the fibre over $[\omega]$;
    \item taking homotopy groups at basepoints $\omega$ and  $\tilde{\omega}$,  the rest of the sequence is a long exact sequence of groups, ending with the stabiliser of $[\tilde{\omega}]$ in $ \H^1(K_{\omega})$.
    \end{itemize}

\subsection{Filtrations and towers}\label{filttowersn}

Given a DGLA $L$, complete  with respect to a filtration $\Fil$ by subcomplexes satisfying $[\Fil^iL,\Fil^jL] \subset \Fil^{i+j}L$, we can iteratively apply Corollary \ref{obsDGLAcor} to the diagrams $L/\Fil^{p+1}L \to L/\Fil^pL \to L/\Fil^1L$. Each of these gives a long exact sequence 
\[
 \ldots \to \H^{1-i}(\gr_{\Fil}^pL_{\omega}) \to \pi_i\mmc(L/\Fil^{p+1}L, \tilde{\omega}) \to \pi_i\mmc(L/\Fil^pL, \omega) \to \H^{2-i}(\gr_{\Fil}^pL_{\omega}) \to \ldots.
\]
 In particular,  vanishing of the cohomology groups  $\H^{2}(\gr_{\Fil}^pL_{\omega})$ guarantees that $\omega \in \pi_0\mc(L/F^1L)$ lifts compatibly to all $\pi_0\mmc(L/\Fil^{p+1}L)$. Such vanishing is the source of many deformation theorems, but those of \S \ref{quantnsn} rely on more involved arguments for the vanishing of obstructions.
 
Following our convention for $\mmc$ of pro-nilpotent DGLAs, we have $\mmc(L)=\Lim_p\mmc(L/\Fil^p)$, and the transition maps are all Kan fibrations, so this limit is a homotopy limit. The Milnor exact sequence  of \cite[Proposition VI.2.15]{sht} thus gives us exact sequences
    \[
    \ast \to {\Lim_p}^1\, \pi_{i+1}\mmc(L/\Fil^p)  \to \pi_i\mmc(L)
    \to \Lim_p \pi_{i}\mmc(L/\Fil^p) \to \ast
    \]
    of groups and pointed sets (basepoints omitted from the notation, but must be compatible).
    
Combining all these data, we have a form of convergent spectral sequence of groups and sets, with first page consisting of the homotopy groups and sets $\pi_i\mmc(L/\Fil^1L)$ and the cohomology groups $\H^{1-i}( \gr_{\Fil}^pL_{\omega})$, converging to $\pi_i\mmc(L)$. As a shorthand, we denote this by
\[
 (\pi_i\mmc(L/\Fil^1L); \H^{1-i}( \gr_{\Fil}^pL_{\omega})~(p \ge 1)) \abuts \pi_i\mmc(L).
\]

In particular,  isomorphisms of cohomology groups $\H^{1-i}( \gr_{\Fil}^pL_{\omega})$ can be used to infer weak equivalences of Maurer--Cartan spaces.

If $L=\Fil^1L$, note that $\mc(L/\Fil^1L)=0$, so all extensions are central and the cohomology groups $\H^{1-i}( \gr_{\Fil}^pL_{\omega})$ are just  $\H^{1-i}( \gr_{\Fil}^pL)$ (since $\omega\in \mc(\gr_{\Fil}^0L)$ must be $0$). That gives the theory described in \cite[\S \ref{poisson-towersn}]{poisson}.

\begin{examples}
The DGLAs of \S\S \ref{poisssn}--\ref{quantnsn} give the following filtered pro-nilpotent DGLAs to which these sequences apply. 
\begin{enumerate}
 \item The pro-nilpotent DGLAs $F^2\widehat{\Pol}(X,n)^{[n+1]}$, with filtration $\Fil^i:=F^{i+1}$ for $i>0$. Since the filtration starts at $\Fil^1$, all extensions are central and we have a convergent spectral sequence
 \[
  \H^{n+2-i}( \Symm_{\cC^{\infty}(X)}^{j+2}(T_{\cC^{\infty}(X)}^{[-n-1]})) \abuts \pi_i\cP(X,n).
 \]
of groups and sets, for $j \ge 0$
 
 \item The pro-nilpotent DGLAs $\tilde{F}^2Q\widehat{\Pol}(X,n)^{[n+1]}$, with filtrations $\Fil^i:=\tilde{F}^{i+1}$ for $i>0$. Again, all extensions are central, and since $\tilde{F}^pQ\widehat{\Pol}(X,n)$ in each of our cases is defined as $\prod_{i \ge p}\hbar^{i-1}F_i\cD$ for some increasingly filtered DGLA $\cD^{[n+1]}$, these give convergent spectral sequences (for $j \ge 0$) 
 \[
 \hbar^{j+1} \H^{n+2-i}(F_{j+2}\cD) \abuts \pi_iQ\cP(X,n).
 \]

 \item Of more interest is the filtration on $\tilde{F}^2Q\widehat{\Pol}(X,n)$ given by $\Fil^0=\tilde{F}^2$, $\Fil^1=\hbar\tilde{F}^1$ and $\Fil^i=\hbar^i\tilde{F}^0$ for $i\ge 2$. This has graded pieces 
 \[
  \gr_{\Fil}^j \cong \prod_{p \ge 2-j} \hbar^{p-1}\gr^F_p\cD \simeq F^{2-j}\widehat{\Pol} (X,n)
 \]
 (where $F^{2-j}=F^0$ for $j \ge 2$)
so gives a convergent spectral sequence 
\[
(\pi_i\cP(X,n), \H^{n+2-i}F^{1}\widehat{\Pol}_{\pi}(X,n) ;\H^{n+2-i}\widehat{\Pol}_{\pi}(X,n)~(j \ge 2)) \abuts \pi_iQ\cP(X,n)
\]
of groups and sets, where $(\phantom{x})_{\pi}$ denotes twisting by the relevant element $\pi \in \pi_0\cP(X,n)$.

 \item For the corresponding filtration on the sub-DGLA of self-dual elements of $\tilde{F}^2Q\widehat{\Pol}(X,n)^{[n+1]}$, the graded pieces $\gr_{\Fil}^j$ are acyclic for $j$ odd and are unaffected for $j$ even, so we have a convergent spectral sequence
 \[
   (\pi_i\cP(X,n),0, \H^{n+2-i}\widehat{\Pol}_{\pi}(X,n),0,\H^{n+2-i}\widehat{\Pol}_{\pi}(X,n),0, \ldots  ) \abuts \pi_iQ\cP(X,n)^{sd}.
 \]
\end{enumerate}

\end{examples}

\bibliographystyle{alphanum}
\bibliography{references.bib}
\end{document}

%% file: newcommands.tex
\newcommand\la{\leftarrow}

\newcommand\id{\mathrm{id}}

\newcommand\ten{\otimes}
\newcommand\hten{\hat{\otimes}}

\newcommand\CC{\mathrm{C}}

\renewcommand\H{\mathrm{H}}
\newcommand\z{\mathrm{Z}}

\newcommand\N{\mathbb{N}}
\newcommand\Z{\mathbb{Z}}
\newcommand\Q{\mathbb{Q}}

\newcommand\R{\mathbb{R}}

\newcommand\bL{\mathbb{L}}

\newcommand\bS{\mathbb{S}}

\newcommand\C{\mathcal{C}}

\newcommand\cC{\mathcal{C}}
\newcommand\cD{\mathcal{D}}

\newcommand\cL{\mathcal{L}}
\newcommand\cM{\mathcal{M}}

\newcommand\cP{\mathcal{P}}

\newcommand\cS{\mathcal{S}}
\newcommand\cT{\mathcal{T}}

\renewcommand\O{\mathscr{O}}

\newcommand\sC{\mathscr{C}}

\newcommand\sF{\mathscr{F}}

\newcommand\sH{\mathscr{H}}

\newcommand\sO{\mathscr{O}}

\newcommand\fX{\mathfrak{X}}
\newcommand\fY{\mathfrak{Y}}

\renewcommand\L{\Lambda}

\newcommand\g{\mathfrak{g}}

\newcommand\fr{\mathfrak{r}}

\renewcommand\hom{\mathscr{H}\!\mathit{om}}

\newcommand\cHom{\mathcal{H}\!\mathit{om}}

\newcommand\cDiff{\mathcal{D}\!\mathit{iff}}

\newcommand\Ho{\mathrm{Ho}}

\newcommand\Ass{\mathrm{Ass}}
\newcommand\Br{\mathrm{Br}}

\newcommand\Mod{\mathrm{Mod}}

\newcommand\Hom{\mathrm{Hom}}

\newcommand\map{\mathrm{map}}

\newcommand\HHom{\underline{\mathrm{Hom}}}

\newcommand\Ext{\mathrm{Ext}}
\newcommand\EExt{\mathbb{E}\mathrm{xt}}

\newcommand\cone{\mathrm{cone}}
\newcommand\cocone{\mathrm{cocone}}

\newcommand{\brh}{\llbracket \hbar \rrbracket}
\newcommand{\brhh}{\llbracket \hbar^2 \rrbracket}
\newcommand{\llb}{\llbracket}
\newcommand{\rrb}{\rrbracket}

\newcommand\coker{\mathrm{coker\,}}

\newcommand\loc{\mathrm{loc}}

\newcommand\Sp{\mathrm{Sp}}
\newcommand\PreSp{\mathrm{PreSp}}

\newcommand\Pol{\mathrm{Pol}}
\newcommand\cPol{\mathcal{P}ol}

\newcommand\poly{\mathrm{poly}}
\newcommand\Com{\mathrm{Com}}
\newcommand\Comp{\mathrm{Comp}}

\newcommand\nondeg{\mathrm{nondeg}}

\newcommand\ad{\mathrm{ad}}

\newcommand\Lim{\varprojlim}
\newcommand\LLim{\varinjlim}

\newcommand\ho{\mathrm{ho}\!}
\newcommand\into{\hookrightarrow}
\newcommand\onto{\twoheadrightarrow}
\newcommand\abuts{\implies}
\newcommand\xra{\xrightarrow}
\newcommand\xla{\xleftarrow}

\newcommand\pr{\mathrm{pr}}

\newcommand\odd{\mathrm{odd}}
\newcommand\even{\mathrm{even}}

\newcommand\bt{\bullet}
\newcommand\by{\times}

\DeclareMathOperator\mc{\mathrm{MC}}
\DeclareMathOperator\mmc{\underline{\mathrm{MC}}}

\newcommand\Vect{\mathrm{Vect}}
\newcommand\Perf{\mathrm{Perf}}

\newcommand\Symm{\mathrm{Symm}}

\newcommand\GL{\mathrm{GL}}

\newcommand\Tot{\mathrm{Tot}\,}

\newcommand\pd{\partial}

\newcommand\half{\frac{1}{2}}

\newcommand\DCrit{\mathrm{DCrit}}

\newcommand\gr{\mathrm{gr}}

\newcommand\Fil{\mathrm{Fil}}

\newcommand\DR{\mathrm{DR}}

\newcommand\op{\mathrm{opp}}

\newcommand\co{\colon\thinspace}

\newcommand\oE{\mathbf{E}}

\newcommand\oR{\mathbf{R}}

\newcommand\oL{\mathbf{L}}

\newcommand\oT{\mathbf{T}}

\newcommand\uleft\underleftarrow
\newcommand\uline\underline
\newcommand\uright\underrightarrow

\DeclareMathOperator{\hatTot}{\mathrm{T}\widehat{\mathrm{o}}\mathrm{t}}

%% file: newthms.tex
\newtheorem{theorem}{Theorem}[section]
\newtheorem{proposition}[theorem]{Proposition}
\newtheorem{corollary}[theorem]{Corollary}

\newtheorem{lemma}[theorem]{Lemma}
\newtheorem*{theorem*}{Theorem}
\newtheorem*{proposition*}{Proposition}
\newtheorem*{corollary*}{Corollary}
\newtheorem*{lemma*}{Lemma}
\newtheorem*{conjecture*}{Conjecture}

\theoremstyle{definition}
\newtheorem{definition}[theorem]{Definition}
\newtheorem*{definition*}{Definition}

\newtheorem*{notation*}{Notation}

\theoremstyle{remark}
\newtheorem{example}[theorem]{Example}
\newtheorem{examples}[theorem]{Examples}
\newtheorem{remark}[theorem]{Remark}
\newtheorem{remarks}[theorem]{Remarks}

\newtheorem*{example*}{Example}
\newtheorem*{examples*}{Examples}
\newtheorem*{remark*}{Remark}
\newtheorem*{remarks*}{Remarks}
\newtheorem*{exercise*}{Exercise}
\newtheorem*{property*}{Property}
\newtheorem*{properties*}{Properties}